\documentclass{article}
\usepackage[left=3cm,right=3cm,top=3cm,bottom=3cm]{geometry} 
\usepackage{amsmath,amssymb} 
\usepackage{amsmath, amsthm,amssymb,amsfonts,latexsym}
\usepackage{amsthm}
\usepackage{tikz}
\usepackage{lipsum}
\usepackage{color}
\usepackage{hyperref}
\usepackage[capitalise]{cleveref}
\usetikzlibrary{patterns}
\usetikzlibrary{intersections}
\usetikzlibrary{arrows,positioning}
\usepackage{graphicx}
\usepackage{pgfplots}
\pgfplotsset{compat=1.14}
\usetikzlibrary{plotmarks}
\usepackage{caption}
\usepackage{subcaption}
  \usepgfplotslibrary{patchplots}
  \usepackage{grffile}
\usepgfplotslibrary{fillbetween}
\usepackage{float}
\usepackage{caption}
\usepackage{subcaption}
\usepackage{comment}

\newcommand*{\rom}[1]{\expandafter\@slowromancap\romannumeral #1@}

\newcommand{\gettikzxy}[3]{%
  \tikz@scan@one@point\pgfutil@firstofone#1\relax
  \edef#2{\the\pgf@x}%
  \edef#3{\the\pgf@y}%
}
\newcommand{\cblue}{\color{blue}}
\newcommand{\cred}{\color{red}}

\newcommand{\tE}{\tilde{\sE}}
\newcommand{\ulz}{\underline{z}}
\newcommand{\ulu}{\underline{u}}
\newcommand{\uls}{\underline{s}}
\newcommand{\oolz}{\overline{\olz}}

\newcommand{\uulz}{\underline{\ulz}}
\newcommand{\uuls}{\underline{\uls}}
\newcommand{\uulu}{\underline{\ulu}}

\newcommand{\olz}{\overline{z}}

\newcommand{\ols}{\overline{s}}
\newcommand{\tz}{\tilde{z}}

\newcommand{\tp}{\tilde{\phi}}

\newcommand{\sA}{\mathcal A}
\newcommand{\sB}{\mathcal B}

\newcommand{\sD}{\mathcal D}
\newcommand{\sE}{\mathcal E}
\newcommand{\sF}{\mathcal F}

\newcommand{\sI}{\mathcal I}
\newcommand{\sL}{\mathcal L}
\newcommand{\sM}{\mathcal M}
\newcommand{\sN}{\mathcal N}
\newcommand{\sP}{\mathcal P}

\newcommand{\sT}{\mathcal T}

\newcommand{\R}{\mathbb R}
\newcommand{\E}{\mathbb E}
\newcommand{\N}{\mathbb N}
\newcommand{\F}{\mathbb F}

\newcommand{\bM}{\mathbb M}
\newcommand{\Prob}{\mathbb P}

\newcommand{\arginf}{\mbox{arginf}}

\newcommand{\Leb}{\mbox{Leb}}
\newtheorem{thm}{Theorem}[section]
\newtheorem{prop}[thm]{Proposition}

\newtheorem{lem}[thm]{Lemma}
\newtheorem{cor}[thm]{Corollary}
\newtheorem{rem}[thm]{Remark}

\newtheorem{defn}[thm]{Definition}

\begin{document}
\title{A construction of the left-curtain coupling}
\author{
David Hobson\thanks{University of Warwick \textit{d.hobson@warwick.ac.uk}} \hspace{10mm}
Dominykas Norgilas\thanks{University of Michigan \textit{dnorgila@umich.edu}}
}
\date{\today}
\maketitle

\begin{abstract}
In a martingale optimal transport (MOT) problem mass distributed according to the law $\mu$ is transported to the law $\nu$ in such a way that the martingale property is respected.
Beiglb\"ock and Juillet (On a problem of optimal transport under marginal martingale constraints, Annals of Probability, 44(1):42-106, 2016) introduced a solution to the MOT problem which they baptised the left-curtain coupling. The left-curtain coupling has been widely studied and shown to have many applications, including to martingale inequalities and the model-independent pricing of American options.
Beiglb\"ock and Juillet proved existence and uniqueness, proved optimality for a family of cost functions, and proved that when $\mu$ is a continuous distribution, mass at $x$ is mapped to one of at most two points, giving lower and upper functions. Henry-Labord\`{e}re and Touzi (An explicit martingale version of Brenier's theorem, Finance and Stochastics, 20:635-668, 2016) showed that the left-curtain coupling is optimal for an extended family of cost functions and gave a construction of the upper and lower functions under an assumption that $\mu$ and $\nu$ are continuous, together with further simplifying assumptions of a technical nature.


In this article we construct these upper and lower functions in the general case of arbitrary centred measures in convex order, and thereby give a complete construction of the left-curtain coupling. In the case where $\mu$ has atoms these upper and lower functions are to be interpreted in the sense of a lifted martingale.  \\

\end{abstract}

\section{Introduction}
In the classical formulation of the optimal transport (OT) problem due to Kantorovich \cite{kantorovich:42} one seeks a joint law $\pi$ for random variables $X\sim\mu$ and $Y\sim\nu$ which, for a given cost function $c:\R^d\times\R^d\mapsto\R$, minimises $\E^\pi[c(X,Y)]$. The cornerstone result in $\R^d$ with an Euclidean cost $c(x,y)=\lvert x-y\lvert^2$ is Brenier's Theorem, see Brenier \cite{Brenier:87} and R\"{u}schendorf and Rachev \cite{RachRu}. Brenier's Theorem states that, under some regularity assumptions, the optimal coupling is deterministic and realised by a map which corresponds to a gradient of a convex function. In the one-dimensional setting this coupling is identified with the Fr\'{e}chet-Hoeffding (or quantile) coupling $\pi_{HF}$. An important feature of this coupling is that it is optimal for the large class of cost functions satisfying the Spence-Mirrlees condition $c_{xy}>0$.

In recent years, there has been significant interest in optimal transport problems where the transport plan is constrained to be a martingale. The basic problem of martingale optimal transport (MOT) is, given probability measures $\mu$ and $\nu$ on $\R$ which are in convex order and a cost function $c$, to construct a joint law $\pi$ for $X\sim\mu$ and $Y\sim\nu$ satisfying the martingale constraint $\E^\pi[Y\lvert X]=X$ and such that $\E^\pi[c(X,Y)]$ is minimised. Such problems arise naturally in the context of robust (or model-independent) mathematical finance and were first considered by Hobson and Neuberger~\cite{HobsonNeuberger:12} and Hobson and Klimmek~\cite{HobsonKlimmek:15} for the specific (financially relevant) cost functions $c(x,y)=-|y-x|$ and $c(x,y)=|y-x|$. Subsequently,
Beiglb\"{o}ck et al. \cite{BeiglbockHenryLaborderePenkner:13} (in a discrete time setting) and Galichon et al. \cite{GalichonHenryLabordereTouzi:14} (in continuous time) extended the problem to a more general setting, made the connection to optimal transport problems, and proved duality theorems. The class of MOT problems is of wide mathematical interest and, as well as the application to finance, is closely related to, and has important consequences for, the study of martingale inequalities (see Beiglb\"{o}ck and Nutz \cite{BeiglbockNutz:14}, Henry-Labord{\`e}re et al. \cite{HLabordereOblojSpoida:16}, Ob{\l}{\'o}j et al. \cite{OblojSpoidaTouzi:15}) and the Skorokhod embedding problem (see Beiglb\"{o}ck et al. \cite{BeiglbockCoxHuesmann:17}, K\"{a}llblad et al. \cite{KallbladTanTouzi:17}).

Using an extension of the notion of cyclical monotonicity from the classical OT setting, Beiglb\"{o}ck and Juillet \cite{BeiglbockJuillet:16} introduced the left-monotone martingale coupling, which can be viewed as a martingale analogue of the monotone Fr\'{e}chet-Hoeffding coupling. The authors then prove the existence and uniqueness of the left-monotone coupling, which they baptise the \textit{left-curtain} martingale coupling, together with the optimality of this joint probability measure for some specific cost functions. Henry-Labord\`{e}re and Touzi~\cite{HenryLabordereTouzi:16} extended their results to show that the left-curtain coupling is optimal for a wide class of payoffs (essentially\footnote{Some authors seek to minimise $\E^{\pi}[c(X,Y)]$, whereas others seek to maximise, and therefore some care is needed when moving between articles. Henry-Labord\`{e}re and Touzi~\cite{HenryLabordereTouzi:16} maximise, and therefore the condition they state is $c_{xyy}>0$, but this becomes the condition in the main text if we move to $-c$ and a minimisation problem.} those satisfying $c_{xyy}<0$). Several other authors further investigate the properties and extensions of the left-curtain coupling, see Beiglb\"{o}ck et al. \cite{BeiglbockHenryLabordereTouzi:17,BeiglbockCox:17}, Juillet \cite{Juillet:16}, Nutz et al. \cite{NutzStebegg:18,NutzStebeggTan:17}, Campi et al. \cite{Campi:17}. In the case of a continuum of marginals which are increasing in convex order, Henry-Labord\`{e}re et al.~\cite{HenryLabordereTanTouzi:16}, Juillet~\cite{Juillet:18} and Br\"{u}ckerhoff at al.~\cite{BHJ:20} recently showed, amongst other things, how the left-curtain coupling can be used to construct a martingale that fits given marginals at any given time and how it solves a continuous-time version of the martingale optimal transport problem.

In the light of the Brenier's Theorem in classical OT theory, a natural question arises as to whether the aforementioned optimal martingale transport plans possess similar nice structural properties and whether they can be explicitly constructed. Due to the martingale constraint, if the marginals are such that $\mu\neq\nu$, then no martingale coupling can be realized by a single map. Instead, the best one can hope for is a binomial map. Hobson and Neuberger \cite{HobsonNeuberger:12} showed that if $\mu$ is continuous then there exists a pair of increasing functions on which $\pi_{HN}$, their optimal coupling, is concentrated. In particular, if we write $\pi_{HN}(dx,dy) = \mu(dx) \pi_x^{HN}(dy)$ we find $\pi_x^{HN}$ has support in a two-point set $\{ g(x), f(x) \}$ and the functions $x  \mapsto g(x)$ and $x \mapsto f(x)$ are both increasing. Hobson and Klimmek \cite{HobsonKlimmek:15}, on the other hand, worked under the dispersion assumption and showed that their coupling $\pi_{HK}$ is such that each portion of mass from initial law $\mu$ is mapped to $\nu$ by splitting it into three points at most. In particular, $\pi_{HK}$ is supported on the diagonal and the graphs of two explicitly constructed decreasing functions. When $\mu$ is atom free, Beiglb\"{o}ck and Juillet \cite{BeiglbockJuillet:16} showed that for the left-curtain coupling $\pi_{lc}$ there exist lower and upper functions $T_d,T_u$ with $T_d(x)\leq x\leq T_u(x)$ such that if we write $\pi_{lc}(dx,dy) = \mu(dx) \pi^{lc}_x(dy)$ then $\pi^{lc}_x$ has support concentrated on the set $\{ T_d(x), T_u(x) \}$. Here $T_u$ is non-decreasing while $T_d$ satisfies a particular left-monotonicity property.

There are two methods in the literature used to obtain the upper and lower functions that characterise the left-curtain coupling $\pi_{lc}$. The first method is non-constructive and is to approximate $\mu$ by a family of discrete measures and to consider the limits of the resulting upper and lower functions. The second method requires additional regularity assumptions on $\mu$ and $\nu$ and is to characterise $T_d$ and $T_u$ via differential equations (or integral equations). In the case where both marginals $\mu$ and $\nu$ are continuous (and satisfy a further particular structural property), Henry-Labord\`{e}re and Touzi \cite{HenryLabordereTouzi:16} (among other things) construct the upper and lower functions $T_d,T_u$ as the solutions of a pair of coupled ordinary differential equations. Indeed, after some further ingenious manipulations, they show that $T_d$ is the root of an integral equation. However, there are several barriers which make it difficult to extend the construction via differential equations to the general case. First, it cannot cope with atoms in the initial or terminal laws, and requires both measures to have positive densities. Second, it requires a starting condition to initialise the differential equations. Third, the method works by solving for $T_d$ and $T_u$ on a family of intervals, but there may be countably many such intervals, and the set of right-endpoints of these intervals may have (countably many) accumulation points beyond each of which it is difficult to extend the solutions to the differential equations. The same sorts of issues apply to the construction via the integral equation of Henry-Labord\`{e}re and Touzi \cite{HenryLabordereTouzi:16}

The first point is absolutely fundamental.
When $\mu$ has an atom at $x$ the probability kernel $\pi_{lc}^x(\cdot)$ in the disintegration $\pi_{lc}(dx,dy)=\mu(dx)\pi^x_{lc}(dy)$ becomes a measure with support on non-trivial subsets of $\R$ and not just on a two-point set. In this case $T_d,T_u$ cannot be constructed, unless we allow them to be multi-valued. By changing the viewpoint, Hobson and Norgilas \cite{HobsonNorgilas:17} showed how to recover the property that $Y$ takes values in a two-point set. The idea is to write $X=G_\mu(U)$, where $G_\mu$ is a quantile function of $\mu$ and $U\sim U[0,1]$, and then to seek functions $R,S$ satisfying certain monotonicity properties such that $Y\in\{R(U),S(U)\}$. While $T_d$ and $T_u$ are multivalued on the atoms of $\mu$, $R$ and $S$ remain well defined. Hobson and Norgilas explicitly constructed such $R$ and $S$ in the case the initial law $\mu$ is finitely supported (while $\nu$ is arbitrary), and then by approximating general $\mu$ with atomic probability measures showed that the limiting functions give rise to a generalised, or \textit{lifted}, left-curtain martingale coupling.

The goal of this paper is {\em to give a direct, geometric construction of the lifted left-curtain martingale coupling for general measures $\mu$ and $\nu$}. In particular we construct $R$ and $S$ (and $(T_d,T_u)$ in the case where $\mu$ is continuous). Our methods rely neither on differential equations nor on the delicate approximation of measures, but rather on a representation of the initial and target laws via potentials. The potential of a measure involves integrating the measure against a test function and this has a smoothing effect. It is this extra smoothness which allows us to give a global construction of the key quantities. Nonetheless, some delicate arguments are needed to prove that the quantities we construct have the appropriate monotonicity properties and do indeed yield the left-curtain martingale coupling, especially since we place no assumptions on the initial or terminal laws.

The power of the potential representation is well-recognised in related settings.
There is a close connection between martingale optimal transport and solutions of the Skorokhod embedding problem (SEP) for Brownian motion (especially for non-trivial initial laws), based on the idea of viewing a martingale as a time-change of Brownian motion. One productive source of elegant solutions to the SEP is the potential-theoretic representation of measures and a geometric description due to Chacon and Walsh~\cite{ChaconWalsh:76}. Many of the classical solutions of the Skorokhod embedding problem (and therefore many of the martingale optimal transports which have been proposed in the literature) can be described by drawing tangents (or supporting hyperplanes in the atomic case) on a suitable picture, see Chacon and Walsh~\cite{ChaconWalsh:76} and Hobson~\cite{Hobson:98a}.
The constructions of the SEP due to Dubins~\cite{Dubins:68}, Az{\'e}ma and Yor~\cite{AzemaYor:79}, Jacka~\cite{Jacka:89}, Vallois~\cite{Vallois:83}, 
Hobson~\cite{Hobson:98b}, Hobson and Pedersen~\cite{HobsonPedersen:01} (at least) have a representation in this form.
The geometric approach has a clear advantage in bypassing many of the technical issues which arise in approximation arguments.

The classical result by Strassen \cite{Strassen:65} states that it is possible to transport $\mu$ to $\nu$ using a martingale if and only if the marginal laws respect the convex order condition $\mu\leq_{cx}\nu$ (i.e., $\mu$ is less than $\nu$ in convex order). In order to study transport plans in the martingale setting, Beiglb\"{o}ck and Juillet \cite{BeiglbockJuillet:16} introduced the notion of \textit{extended} convex order of two measures, denoted by $\leq_{E}$, which compares measures of possibly different total mass. If a pair of measures $\mu,\nu$ is such that $\mu\leq_E\nu$, than there exists a martingale that transports $\mu$ into $\nu$ (without necessarily covering all of $\nu$). In particular, the set of measures $\eta$ with  $\mu\leq_{cx}\eta \leq \nu$, is non-empty, and each such $\eta$ corresponds to a terminal law of a martingale that embeds $\mu$ into $\nu$. Beiglb\"{o}ck and Juillet \cite{BeiglbockJuillet:16} proved that there exists a canonical choice of such $\eta$ with respect to $\leq_{cx}$. In particular there exists the unique measure $S^\nu(\mu)$, the \textit{shadow} of $\mu$ in $\nu$, that is the smallest measure with respect to convex order among measures $\eta$ satisfying $\mu\leq_{cx}\eta \leq \nu$. Our interest in the shadow measure lies in the fact that the left-curtain martingale coupling can be defined as the unique measure $\pi_{lc}$ on $\R^2$ such that, for each $x\in\R$, $\pi_{lc}\lvert_{(-\infty,x])\times\R}$ has the first marginal $\mu\lvert_{(-\infty,x]}$ and the second marginal $S^\nu(\mu\lvert_{(-\infty,x]})$. (For other martingale transports arising using the shadow measure, but different parametrisations of $\mu$, see Beiglb\"{o}ck and Juillet \cite{BeiglbockJuillet:16s}.) Recently Beiglb\"{o}ck et al. \cite{BeiglbockHobsonNorgilas:20} showed how to explicitly construct, via potential-geometric arguments, the shadow measure $S^\nu(\mu)$ for arbitrary $\mu$ and $\nu$ with $\mu\leq_E\nu$. This turns out to be the main ingredient of our construction of the upper and lower functions characterising the left-curtain coupling.

One feature of our construction is that in the case where $\mu$ is continuous (respectively the general case) for each $x \in \R$ (respectively $u \in (0,1)$) we find $\{T_d(x),T_u(x)\}$ (respectively $\{ R(u),S(u) \}$) by considering the convex hull of a certain function. Since we can do this for each $x$ individually, the construction for a single $x$ immediately extends to a construction for all $x$ and to pairs of functions $\{T_d(\cdot),T_u(\cdot)\}$ (with a similar conclusion for the lifted martingale). In some applications, for example the model-independent pricing of American options (see Hobson and Norgilas~\cite{HobsonNorgilas:17}), it is sufficient to find a triple $(x,T_d(x),T_u(x))$ associated with the left-curtain coupling and with one further property, and the full structure of the construction of the left-curtain coupling is not required. Then the direct nature of our construction is particularly useful.

In recent work Bayraktar et al. \cite{BayraktarSprmg:21} provide a geometric construction of the so-called \textit{increasing} supermartingale coupling $\pi_I$ (introduced by Nutz and Stebegg \cite{NutzStebegg:18}), which builds upon and extends the construction presented in this article. ($\pi_{I}$ can be viewed as supermartingale counterpart of the left-curtain coupling $\pi_{lc}$). The increasing coupling $\pi_I$ has a feature that it behaves as $\pi_{lc}$ on a certain part of the space, and thus the construction of Bayraktar et al. \cite{BayraktarSprmg:21} heavily relies on the ability to construct $\pi_{lc}$ on these `martingale' intervals.  This is achieved in the present paper and in full generality.

The paper is structured as follows. In \cref{sec:prelims} we discuss the relevant notions of probability measures and (extended) convex order, and discuss some important (for our main theorems) results regarding the convex hull of a function. In \cref{sec:shadow} we introduce the shadow measure and the left-curtain martingale coupling. \cref{sec:geometric,sec:LeftCont,sec:embed} are dedicated to our main results. (\cref{sec:phiprop} contains preliminary results for \cref{sec:LeftCont,sec:embed}.) In \cref{sec:geometric} we construct a candidate pair of functions that characterise the left-curtain coupling (\cref{thm:monotone}), while in \cref{sec:LeftCont} we study their regularity properties. Finally in \cref{sec:embed}, we show (\cref{thm:construction2}) that our construction yields the (lifted) left-curtain martingale coupling. Some proofs are deferred until the appendix.

\section{Preliminaries}\label{sec:prelims}
\subsection{Measures and Convex order}
\label{prelims:convex}
Let $\sM$ (respectively $\sP$) be the set of measures (respectively probability measures) on $\R$ with finite total mass and finite first moment, i.e. if $\eta\in\sM$, then $\eta(\R)<\infty$ and $\int_\R\lvert x\lvert\eta(dx)<\infty$. Given a measure $\eta\in\sM$ (not necessarily a probability measure), define $\bar{\eta} = \int_\R x \eta(dx)$ to be the first moment of $\eta$ (and then $\bar\eta/\eta(\R)$ is the barycentre of $\eta$). The support of $\eta\in\sM$ is denoted by supp$(\eta)$; it is the smallest closed set $E\subseteq\R$ with $\eta(E)=\eta(\R)$. Let $\ell_\eta:=\inf\{k\in\textrm{supp}(\eta)\}$ and $r_\eta:=\sup\{k\in\textrm{supp}(\eta)\}$. Let $\sI_\eta$ be the smallest interval containing the support of $\eta$, so that $\{ \ell_\eta, r_\eta \}$ are the endpoints of $\sI_\eta$. If $\eta$ has an atom at $\ell_\eta$ then $\ell_\eta$ is included in $\sI_\eta$, and otherwise it is excluded, and similarly for $r_\eta$. 

For $\eta \in \sM$, let $F_\eta:\R\mapsto[0,\eta(\R)]$ and $G_\eta:(0,\eta(\R)) \mapsto \R$ be the distribution and quantile functions of $\eta$, respectively. As usual $F_\eta$ is right-continuous, however we take an arbitrary version of $G_\eta$ until further notice. (In \cref{sec:LeftCont,{sec:embed}} we will work with a left-continuous version of $G_\eta$.)

For $\alpha\geq0$ and $\beta \in \R$, let $\sD(\alpha, \beta)$ denote the set of non-negative, non-decreasing convex functions $f:\R \mapsto \R_+$ such that
\[ \lim_{ z \downarrow -\infty} f(z)=0, \hspace{10mm} \lim_{z \uparrow \infty} \{ f(z) - (\alpha z- \beta) \} =0. \]
When $\alpha=0$, $\sD(0,\beta)$ is empty unless $\beta=0$ and then $\sD(0,0)$ contains one element.

For $\eta\in\sM$, define the function $P_\eta : \R \mapsto \R_+$ by
\begin{equation*}
P_\eta(k) := \int_{\R} (k-x)^+ \eta(dx),\quad k\in\R.
\end{equation*}
(The notation $P_{\cdot}$ arises from the connection with the expected payoff of a \emph{put option}.)
The following properties of $P_\eta$ can be found in Chacon~\cite{Chacon:77}, and Chacon and Walsh~\cite{ChaconWalsh:76}:  $P_\eta \in \sD(\eta(\R), \bar{\eta})$ and $\{k : P_{\eta}(k) > (\eta(\R)k - \bar{\eta})^+ \} = (\ell_\eta,r_\eta)$. Conversely (see, for example, Hirsch et al. \cite[Proposition 2.1]{peacock}),  if $h \in \sD(k_m,k_f)$ for some numbers $k_m\geq0$ and $k_f\in\R$ (with $k_f=0$ if $k_m=0$), then there exists a unique measure $\eta\in\sM$, with total mass $\eta(\R)=k_m$ and first moment $\bar{\eta}=k_f$, such that $h=P_{\eta}$. In particular, $\eta$ is uniquely identified by the second derivative of $h$ in the sense of distributions. Note that $P_{\eta}$ is related to the potential $U_\eta$, defined by \begin{equation*}
 U_\eta(k) : =  - \int_{\R} |k-x| \eta(dx),\quad k\in\R,
 \end{equation*}
via $P_\eta(k) = \frac{1}{2}(-U_\eta(k) + ( \eta(\R)k- \bar{\eta}))$. We will call $P_\eta$ a modified potential. Finally note that both second derivatives $P''_\eta$ and $-U''_\eta/2$ identify the same underlying measure $\eta$.

For $\eta,\chi\in\sM$, we write $\eta\leq\chi$ if $\eta(A) \leq \chi(A)$ for all Borel measurable subsets $A$ of $\R$, or equivalently if
\begin{equation*}
\int fd\eta\leq\int fd\chi,\quad \textrm{for all non-negative }f:\R\mapsto\R_+.
\end{equation*}
Since $\eta$ and $\chi$ can be identified as second derivatives of $P_\chi$ and $P_\eta$ respectively, we have $\eta\leq\chi$ if and only if $P_\chi-P_\eta$ is convex, i.e. $P_\eta$ has a smaller curvature than $P_\chi$.

Two measures $\eta,\chi\in\sM$ are in convex order, and we write $\eta \leq_{cx} \chi$, if
\begin{equation}\label{eq:cx}
\int fd\eta\leq\int fd\chi,\quad\textrm{for all convex }f:\R\mapsto\R.
\end{equation}
Since we can apply \eqref{eq:cx} to all affine functions, including $f(x)=\pm 1$ and $f(x)=\pm x$, we obtain that
if $\eta \leq_{cx} \chi$ then $\eta$ and $\chi$ have the same total mass ($\eta(\R)= \chi(\R)$) and the same first moment ($\bar{\eta}= \bar{\chi}$).
Moreover, necessarily we must have $\ell_\chi \leq \ell_\eta \leq r_\eta \leq r_\chi$. From simple approximation arguments (see Hirsch et al. \cite{peacock}) we also have that, if $\eta$ and $\chi$ have the same total mass and the same first moment, then $\eta \leq_{cx} \chi$ if and only if $P_{\eta}(k) \leq P_{\chi}(k)$, $k\in\R$.

For our purposes in the sequel we need a generalisation of the convex order of two measures. We follow Beiglb\"{o}ck and Juillet \cite{BeiglbockJuillet:16} and say $\eta,\chi\in\sM$ are in an \textit{extended} convex order, and write $\eta\leq_{E}\chi$, if
 \begin{equation*}
 \int fd\eta\leq\int fd\chi,\quad\textrm{for all \textit{non-negative}, convex }f:\R\mapsto\R_+.
 \end{equation*}
The partial order $\leq_E$ generalises both $\leq$ and $\leq_{cx}$ in the sense that it preserves existing orderings and gives rise to new ones. If $\eta\leq_{cx}\chi$ then also $\eta\leq_E\chi$ (since non-negative convex functions are convex), while if $\eta\leq\chi$, we also have that $\eta\leq_E\chi$ (since non-negative convex functions are non-negative). Note that, if $\eta\leq_E\chi$, then $\eta(\R)\leq\chi(\R)$ (apply the non-negative convex function $\phi(x)=1$ in the definition of $\leq_E$). It is also easy to prove that, if $\eta(\R)=\chi(\R)$, then $\eta\leq_{E}\chi$ is equivalent to $\eta\leq_{cx}\chi$.

For $\eta,\chi\in\sP$, let ${\Pi}(\eta,\chi)$ be the set of probability measures on $\R^2$ with the first marginal $\eta$ and second marginal $\chi$. Let ${\Pi}_M(\eta,\chi)$ be the set of martingale couplings of $\eta$ and $\chi$.
Then
\begin{equation*}
  {\Pi}_M(\eta,\chi) = \big\{ \pi \in {\Pi}(\eta,\chi) : \mbox{\eqref{eq:martingalepi} holds} \big\},
\end{equation*}
where \eqref{eq:martingalepi} is the martingale condition
\begin{equation}
\int_{x \in B} \int_{y \in \R} y \pi(dx,dy) = \int_{x \in B} \int_{y \in \R} x \pi(dx,dy) = \int_B x \eta(dx), \quad
\mbox{$\forall$ Borel $B \subseteq \R$}.
\label{eq:martingalepi}
\end{equation}
Equivalently, $\Pi_M(\eta,\chi)$ consists of all transport plans $\pi$ (i.e. elements of ${\Pi}(\eta,\chi)$) such that the disintegration in probability measures $(\pi_x)_{x\in\R}$ with respect to $\eta$ satisfies $\int_\R y\pi_x(dy)=x$ for $\eta$-almost every $x$. 

If we ignore the martingale requirement \eqref{eq:martingalepi}, it is easy to see that the set of probability measures with given marginals is non-empty, i.e. ${\Pi}(\eta,\chi)\neq\emptyset$ (consider the product measure $\eta\otimes\chi$). However, the fundamental question whether, for given $\eta$ and $\chi$, the set of martingale couplings $\Pi_M(\eta,\chi)$ is non-empty, is more delicate. For any $\pi\in\Pi_M(\eta,\chi)$ and convex $f:\R\mapsto\R$, by (conditional) Jensen's inequality we have that
\begin{equation*}
\int_\R f(x)\eta(dx)\leq \int_\R\int_\R f(y)\pi_x(dy)\eta(dx)=\int_\R f(y)\pi(\R,dy)=\int_\R f(y)\chi(dy),
\end{equation*}
so that $\eta\leq_{cx}\chi$. On the other hand, Strassen \cite{Strassen:65} showed that a converse is also true (i.e. $\eta\leq_{cx}\chi$ implies that ${\Pi}_M(\eta,\chi)\neq\emptyset$), so that ${\Pi}_M(\eta,\chi)$ is non-empty if and only if $\eta\leq_{cx}\chi$.

For a pair of measures $\eta,\chi\in\sM$, let the function $D = D_{\eta,\chi}: \R \mapsto \R$ be defined by $D_{\eta,\chi}(k) = P_\chi(k) - P_\eta(k)$. Note that if $\eta,\chi$ have equal mass and equal first moment then $\eta \leq_{cx} \chi$ is equivalent to $D \geq 0$ on $\R$.
Let $(\ell_D,r_D)$ be the smallest interval containing $\{ k : D_{\eta,\chi}(k)>0 \}$; let $\sI_D$ be the open interval $(\ell_D,r_D)$ together with $\{\ell_D\}$ if $\ell_D>-\infty$ and $D'(\ell_D+):=\lim_{k\downarrow\ell_D}(D(k)-D(\ell_D))/(k-\ell_D)>0$ and $\{ r_D \}$ if $r_D<\infty$ and $D'(r_D-):=\lim_{k\uparrow\ell_D}(D(k)-D(\ell_D))/(k-\ell_D) < 0$. Note that, if $\eta\leq_c\chi$, then $\ell_\nu\leq\ell_\mu\leq r_\mu\leq r_\nu$ and $\sI_D\subseteq[\ell_\nu,r_\nu]$.

The following result (see Hobson~\cite[page 254]{Hobson:98b} or Beiglb\"{o}ck and Juillet ~\cite[Section A.1]{BeiglbockJuillet:16}) tells us that, if $D_{\eta,\chi}(x)=0$ for some $x$, then in any martingale coupling of $\eta$ and $\chi$ no mass can cross $x$.

\begin{lem} Suppose $\eta$ and $\chi$ are probability measures with $\eta \leq_{cx} \chi$. Suppose that $D(x)=0$. If $\pi \in {\Pi_M}(\eta,\chi)$, then we have $\pi((-\infty,x),(x,\infty)) + \pi((x,\infty),(-\infty,x))=0$.
\label{lem:chacon}
\end{lem}

It follows from Lemma \ref{lem:chacon} that, if there is a point $x$ in the interior of the interval $\sI_\eta$ such that $D_{\eta,\chi}(x)=0$, then we can separate the problem of constructing martingale couplings of $\eta$ to $\chi$ into a pair of subproblems involving mass to the left and right of $x$, respectively, always taking care to allocate mass of $\chi$ at $x$ appropriately. Indeed, if there are multiple $\{ x_j \}_{j\geq1}$ with $D_{\eta,\chi}(x_j)=0$, then we can divide the problem into a sequence of `irreducible' problems, each taking place on an interval $\sI_i$ such that $D>0$ on the interior of $\sI_i$ and $D=0$ at the endpoints. All mass starting in a given interval is transported to a point in the same interval. Moreover, by the martingale property, any mass starting at a finite endpoint of $\sI_i$ must stay there. Putting this together we may restrict attention to intervals $I$ on which $D>0$ (with $\lim_{x \rightarrow e_I} D(x)=0$ at endpoints $e_I$ of $I$), and we may assume that the starting law has support within the interior of $I$ and the target law has support within the closure of $I$ (and $I$ is the smallest set with this last property).

{\em Notation:} For $x \in \R$ let $\delta_x$ denote the unit point mass at $x$. For real numbers $c,x,d$ with $c \leq x \leq d$ define the probability measure $\chi_{c,x,d}$ by $\chi_{c,x,d} = \frac{d-x}{d-c}\delta_c + \frac{x-c}{d-c} \delta_d$ with $\chi_{c,x,d} = \delta_x$ if $(d-x)(x-c)=0$. Note that $\chi_{c,x,d}$ has mean $x$ and is the law of a Brownian motion started at $x$ evaluated on the first exit from $(c,d)$.

\subsection{Convex hull}

Our key results will make extensive use of the convex hull. For $f : \R \mapsto \R$ let $f^c$ be the largest convex function which lies below $f$. In our typical application $f$ will be non-negative and this property will be inherited by $f^c$. However, in general we
may have $f^c$ equal to $-\infty$ on $\R$, and the results of this section are stated in a way which
includes this case. Note that if a function $g$ is equal to $-\infty$ (or $\infty$) everywhere, then we deem it to be both linear and convex, and set $g^c$ equal to $g$.

Fix $x,z\in\R$ with $x \leq z$, and define $L^f_{x,z}:\R\mapsto\R$ by
\begin{equation}\label{eq:L1}
L^f_{x,z}(y)=\begin{cases}
f(x)+\frac{f(z)-f(x)}{z-x}(y-x),&\textrm{ if }x<z,\\
f(x),&\textrm{ if }x=z.
\end{cases}
\end{equation}
Then (see Rockafellar \cite[Corollary 17.1.5]{Rockafellar:72}),
\begin{equation}\label{eq:chull}
f^c(y)=\inf_{x\leq y\leq z}L^f_{x,z}(y),\quad y\in\R.
\end{equation}
(Note that for \eqref{eq:chull}, the definition of $L^f_{x,z}$ outside $[x,z]$ is irrelevant and we could restrict the domain of $L^f_{x,z}$ to $[x,z]$. However, in \cref{sec:geometric,sec:LeftCont,sec:embed} we will need $L^f_{x,z}$ to be defined on $\R$.)

Moreover, it is not hard to see (at least geometrically, by drawing the graphs of $f$ and $f^c$) that $f^c$ replaces the non-convex segments of $f$ by straight lines. (Proofs of lemmas in this section are given in \cref{ssec:proofsCH}.)
\begin{lem}
	\label{lem:linear}
Let $f:\R\mapsto\R$ be a lower semi-continuous function. Suppose $f>f^c$ on $(a,b)\subseteq\R$. Then $f^c$ is linear on $(a,b)$.
\end{lem}

In the sequel, for a given function $f$ and $y \in \R$, we will want to identify the values $x,z\in\R$ with $x< y < z$ which attain the infimum in \eqref{eq:chull}. For this, however, we need to allow $x$ and $z$ to take values in the extended real line. Therefore we extend the definition of $L^f_{x,z}$ to $L^f_{-\infty,z}$ and $L^f_{x,\infty}$ by taking appropriate limits in \eqref{eq:L1}. In particular, we define, for each $y\in\R$,
\begin{align*}
L^{f}_{-\infty,z}(y) &= (z-y) \liminf_{x \downarrow - \infty} \frac{f(x)}{|x|} +  f(z),\\
L^{f}_{x,\infty}(y) &=  f(x) + (y-x) \liminf_{z \uparrow \infty} \frac{ f(z)}{z}.
\end{align*}

Let $\phi = \liminf_{z \uparrow \infty} \frac{f(z)}{z} \in [-\infty,\infty]$, $\psi = \liminf_{x \downarrow - \infty} \frac{f(x)}{|x|} \in [-\infty,\infty]$, and if $\phi \in (-\infty,\infty)$, $\gamma=\inf_{w \in \R} (f(w) - \phi w)$. Then we also define $L^f_{-\infty,\infty}$ by
\[ L^{f}_{-\infty,\infty}(y) =  \left\{ \begin{array}{lcl}
\infty, & \; & \psi + \phi > 0,   \\
\gamma + \phi y && \psi+ \phi=0,   \\
-\infty, &&  \psi + \phi  < 0,
\end{array}\quad y\in\R, \right. \]
with the convention that $-\infty+\infty = \infty+ (-\infty) = - \infty$.

Let $\sB(y) = \{ (x,z) : -\infty \leq x < y < z \leq \infty \}$ be the set of open intervals containing $y$.
\begin{defn}\label{defn:convh}
Let $f:\R\mapsto \R$ be a measurable function and $f^c$ denote its convex hull. For $y\in\R$, define
\begin{align*}
X^f(y)=X(y)&=\sup\{x:x\leq y, f^c(x)=f(x)\},\\
Z^f(y)=Z(y)&=\inf\{z:z\geq y, f^c(z)=f(z)\},
\end{align*}
with the convention that $\sup\emptyset=-\infty$ and $\inf\emptyset = \infty$.
\end{defn}

The following result is a slight extension of \eqref{eq:chull}.
\begin{lem} \label{lem:convh1}
Suppose $f:\R\mapsto \R$ is continuous. Then for $y \in \R$ we have $f^c(y)=L^f_{X(y),Z(y)}(y)$.
\end{lem}
We will need one more result regarding the locations of $X$ and $Z$.
\begin{lem} \label{lem:convh2}
Suppose $f:\R\mapsto \R$ is continuous. If $(x,z) \in \sB(y)$ and $f(k)>L^f_{x,z}(k)$ for all  $k\in(x,z)$, then
$$
X(y)\leq x\quad\textrm{and}\quad Z(y)\geq z.
$$
\end{lem}

\section{The shadow measure and left-curtain martingale coupling}\label{sec:shadow}
\subsection{The shadow measure}
Given two measures $\mu,\nu\in\sM$ with $\mu\leq_{E}\nu$ let $\sT(\mu,\nu)=\{\theta\in\sM:\mu\leq_{cx}\theta\leq \nu\}$. Then $\sT(\mu,\nu)$ represents the set of all possible target measures in $\nu$ to which we can transport $\mu$ using a martingale. We are interested in the smallest element of $\sT(\mu,\nu)$ with respect to convex order.
\begin{defn}[Shadow measure]\label{defn:shadow}
Let $\mu,\nu\in\sM$ and assume $\mu\leq_E\nu$. The shadow of $\mu$ in $\nu$, denoted by $S^\nu(\mu)$, has the following properties
\begin{enumerate}
\item\label{s1} $S^\nu(\mu)\leq\nu$.
\item\label{s2} $\mu\leq_{cx}S^\nu(\mu)$.
\item \label{s3}If $\eta$ is another measure satisfying $\mu \leq_{cx}\eta \leq \nu$, then $S^\nu(\mu)\leq_{cx}\eta$.
\end{enumerate}
\label{lem:shadowDef}
\end{defn}
Beiglb\"{o}ck and Juillet \cite[Proposition 4.4 and Lemma 4.6] {BeiglbockJuillet:16} proved the existence and uniqueness of the shadow measure $S^\nu(\mu)$. Furthermore, a result of Beiglb\"{o}ck et al.~\cite{BeiglbockHobsonNorgilas:20} says that the modified potential of the shadow measure, $P_{S^\nu(\mu)}$, can be explicitly constructed, and then $S^\nu(\mu)$ is identified as the second derivative of $P_{S^\nu(\mu)}$ in the sense of distributions.

\begin{thm}[Beiglb\"{o}ck et al. {\cite[Theorem 1]{BeiglbockHobsonNorgilas:20}}]\label{thm:shadow_potential}
Let $\mu,\nu\in\sM$ with $\mu\leq_{E}\nu$. Then
\begin{equation}\label{eq:shadow_potential}
P_{S^\nu(\mu)}=P_\nu-(P_\nu-P_{\mu})^c. 
\end{equation}
\end{thm}
\begin{cor}\label{cor:shadow_potential}
If $(P_\nu - P_\mu)^c$ is linear on $[a,b]$ then $\nu - S^\nu(\mu)$ does not charge $(a,b)$.
\end{cor}

\subsection{The left-curtain coupling $\pi_{lc}$ for continuous $\mu$}\label{sec:LC}

The left-curtain martingale coupling (introduced by Beiglb\"{o}ck and Juillet~\cite{BeiglbockJuillet:16}), and denoted by $\pi_{lc}$, is a martingale coupling that arises via the shadow measure, created working from left to right. More specifically (see Beiglb\"{o}ck and Juillet~\cite[Theorem 4.18]{BeiglbockJuillet:16}), $\pi_{lc}$ is the unique measure in $\Pi_M(\mu,\nu)$ which for each $x \in \R$ transports $\mu\lvert_{(-\infty,x]}$ to the shadow $S^\nu(\mu\lvert_{(-\infty,x]})$. In other words, the first and second marginals of $\pi_{lc}\lvert_{(-\infty,x]\times\R}$ are $\mu\lvert_{(-\infty,x]}$ and $S^\nu(\mu\lvert_{(-\infty,x]})$, respectively, for each $x$. Furthermore, as a consequence of the minimality with respect to convex order, $\pi_{lc}$ is also the unique martingale coupling  which is left-monotone in the sense of \cref{def:lmon} (see Beiglb\"{o}ck and Juillet ~\cite[Theorem 5.3]{BeiglbockJuillet:16}):

\begin{defn}\label{def:lmon}
A transport plan $\pi\in\Pi(\mu,\nu)$ is said to be \textit{left-monotone} if there exists $\Gamma\in\sB(\R^2)$ with $\pi(\Gamma)=1$ and such that, if $(x,y^-),(x,y^+),(x^\prime,y^\prime)\in\Gamma$ we cannot have $x<x^\prime$ and $y^-<y^\prime<y^+$.
\end{defn}

When the initial law $\mu$ is continuous, the left-curtain coupling has a rather simple representation. In particular, for $x\in\mathbb{R}$, the element $\pi^x_{lc}(\cdot)$ in the disintegration $\pi_{lc}(dx,dy) = \mu(dx) \pi_{lc}^x(dy)$ is a measure supported on a set of at most two points.

\begin{lem}[{Beiglb\"ock and Juillet \cite[Corollary 1.6]{BeiglbockJuillet:16}}]\label{lem:LC}
Let $\mu,\nu$ be probability measures in convex order and assume that $\mu$ is continuous. Then there exists a pair of measurable functions $T_d : \R \mapsto \R$ and $T_u : \R \mapsto \R$ such that $T_d(x) \leq x \leq T_u(x)$, such that for all $x<x'$ we have $T_u(x) \leq T_u(x')$ and $T_d(x') \notin (T_d(x),T_u(x))$, and such that, if we define $\bar{\pi}(dx,dy) = \mu(dx) \chi_{T_d(x),x,T_u(x)}(dy)$, then $\bar{\pi} \in {\Pi_M}(\mu,\nu)$ and $\bar{\pi}=\pi_{lc}$.
\end{lem}

Since $T_d(x) \leq x \leq T_u(x)$ we call $T_d$ a lower function and $T_u$ an upper function.

Lemma~\ref{lem:LC} is expressed in terms of elements of $\Pi_M$. We can give an equivalent expression in terms of a martingale. First we give an analogue of Definition~\ref{def:lmon} for functions.

\begin{defn}\label{def:lmonfns}
Given an interval $I$ and an increasing function $g:I \mapsto \R$, a pair of functions $f,h : I \mapsto \R$ is said to be left-monotone with respect to $g$ on $I$ if $f \leq g \leq h$ and if for $x < x'$ we have $h(x) \leq h(x')$ and $f(x') \notin (f(x), h(x))$.
\end{defn}

\begin{cor}\label{cor:BJ}
Let $(\Omega,\sF,\Prob) = (I \times (0,1), \sB(\Omega), \mu \times Leb((0,1)))$. 
Let $\omega = (x,v)$ and let the canonical random variable $(X,V)$ on $(\Omega,\sF,\Prob)$ be given by $(X(\omega),V(\omega))=(x,v)$.
Then $X$ has law $\mu$, $V$ is a $U(0,1)$ random variable and $X$ and $V$ are independent. Let $\F = (\sF_0 = \{ \emptyset, \Omega \}, \sF_1 = \sigma(X), \sF_2 = \sigma(X,V) \})$ and set $\mathbf{M} = (\Omega, \sF, \F, \Prob)$.

Suppose $\mu$ is continuous.
Then there exists $T_d,T_u:I\mapsto\R$ such that $(T_d,T_u)$ is left-monotone with respect to the identity function on $I$ and such that
if we define $Y(x,v) \in \{T_d(x),T_u(x) \}$ by $Y(x,v) = x$ on $T_d(x)=x=T_u(x)$ and
\begin{equation} Y(x,v) = T_d(x) I_{\{ v \leq \frac{T_u(x) - x}{T_u(x)-T_d((x)} \}} +  T_u(x) I_{ \{ v > \frac{T_u(x) - x}{T_u(x)-T_d(x)} \} }
\label{eq:YXVdef}
\end{equation}
otherwise, then
$M = (\bar{\mu},X,Y(X,V))$ is a $\mathbf{M}$-martingale for which $\sL(X) = \mu$ and $\sL(Y) = \nu$.
\end{cor}

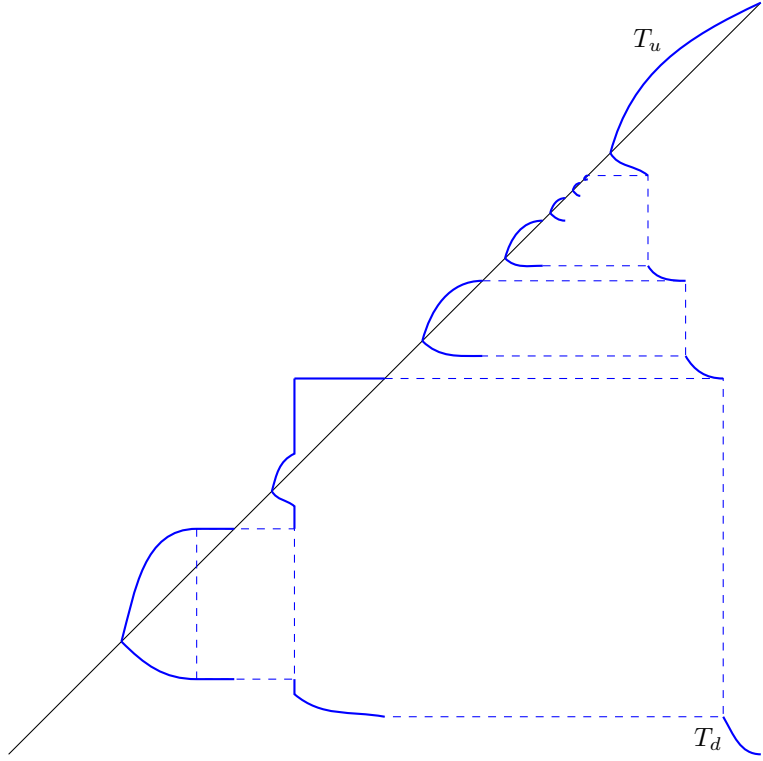
\begin{figure}[H]
	\centering
	\begin{tikzpicture}[dot/.style={circle,inner sep=1pt, fill, label={#1},name#1},
	extended line/.style={shorten >=-#1, shorten <=-#1},
	extended line/.default=3cm,
	declare function={	
		diag(\x)=\x;
		f=1;
		e1=1.5;
		g=3;
		e2=3.5;
		f2=4;
		f3=5;
		fx=5.5;
		e1x=5.5;
		gx=6.3;
		gx1=6.6;
		gx2=7.1;
		e2x=8;
		f2x=8;
		f3x=8.5;
		s=8.5;
		sx=8.5;
		k1=6.37;
		k2=4.5;
		a(\x)=(k1-\x)*(\x<k1)-8;
		b(\x)=(k2-\x)*(\x<k2)-8;
	}]

	\draw[name path=diag,black] (0,0) -- (10,10);
	
	
	\coordinate (g) at (g, {diag(g)});
	\coordinate (e1) at (e1, {diag(e1)});
	\coordinate (e2) at (e2, {diag(e2)});
	\coordinate (s) at (s, {diag(s)});
	
	\coordinate (a) at (g, {diag(f)});
	\coordinate (b) at (f2, {diag(g)});
	\coordinate (c) at (f2, {diag(f)});
	
	\draw[thick, blue, name path=tu1] (e1) to[out=75,in=180] (g-0.5, {diag(g)})--(g);
	\draw[thick, blue, name path=td1] (e1) to[out=315,in=180] (g-0.5, {diag(f)})--(a);
	
	\draw[blue, dashed] (g-0.5, {diag(g)}) -- (g-0.5, {diag(f)}) -- (e2+0.3,{diag(f)}) -- (e2+0.3,{diag(g)})--(g-0.5, {diag(g)});
	
	\draw[thick, blue, name path=tu2] (e2) to[out=75,in=205] (e2+0.3,{diag(f3)-1}) -- (e2+0.3,{diag(f3)});
	\draw[thick, blue, name path=td2] (e2) to[out=300,in=140] (e2+0.3,{diag(g)+0.3}) -- (e2+0.3,{diag(g)});
	\draw[thick, blue, name path=td3,] (e2+0.3,{diag(f)}) -- (e2+0.3,{diag(f)-0.2}) to[out=320,in=170] (f3,0.5);

	\coordinate (gx) at (gx, {diag(gx)});
		\coordinate (gx1) at (gx1, {diag(gx1)});
		\coordinate (gx2) at (gx2, {diag(gx2)});
	\coordinate (e1x) at (e1x, {diag(e1x)});
	\coordinate (e2x) at (e2x, {diag(e2x)});
	\coordinate (sx) at (sx, {diag(sx)});
	
	\coordinate (ax) at (gx, {diag(fx)});
	\coordinate (bx) at (f2x, {diag(gx)});
	\coordinate (cx) at (f2x, {diag(fx)});
	
	\draw[thick, blue, name path=tu1x] (e1x) to[out=75,in=180] (gx);
	\draw[thick, blue, name path=tu1x] (gx1) to[out=75,in=180] (gx2);
		\draw[thick, blue, name path=tu1x] (7.2,{diag(7.2)}) to[out=75,in=180] (7.4,{diag(7.4)});
		\draw[thick, blue, name path=tu1x] (7.5,{diag(7.5)}) to[out=75,in=180] (7.6,{diag(7.6)});
		\draw[thick, blue, name path=tu1x] (7.65,{diag(7.65)}) to[out=75,in=180] (7.7,{diag(7.7)});
	
	\draw[thick, blue, name path=td1x] (e1x) to[out=315,in=180] (gx, {diag(fx)-0.2});
	\draw[thick, blue, name path=td1x] (gx1) to[out=315,in=180] (gx2, {diag(gx1)-0.1});
	\draw[thick, blue, name path=td1x] (7.2,{diag(7.2)}) to[out=315,in=180] (7.4, {diag(gx2)});
	\draw[thick, blue, name path=td1x] (7.5,{diag(7.5)}) to[out=315,in=180] (7.6, {diag(7.43)});
	\draw[thick, blue, name path=td1x] (7.65,{diag(7.65)}) to[out=315,in=180] (7.7, {diag(7.65)});

	\draw[thick, blue, name path=tu2x] (e2x) to[out=75,in=205] (10,10);
	\draw[thick, blue, name path=td2x] (e2x) to[out=300,in=140] (8.5,{diag(7.7)});
	\draw[thick, blue, name path=td3x] (8.5, {diag(gx1)-0.1}) to[out=300,in=180] (9, {diag(gx)});
	\draw[thick, blue, name path=td3x] (9, {diag(fx)-0.2}) to[out=300,in=180] (9.5, {diag(f3)});
	\draw[thick, blue, name path=td3x] (9.5,0.5) to[out=300,in=180] (10,0);

	\draw [blue, dashed] (gx) -- (9,{diag(gx)}) -- (9,{diag(fx)-0.2}) -- (gx, {diag(fx)-0.2}) ;
	\draw[thick, blue] (e2+0.3,{diag(f3)})--(f3,{diag(f3)});
	\draw [blue, dashed] (f3,{diag(f3)}) -- (9.5, {diag(f3)})--(9.5,0.5) -- (f3,0.5);
	\draw[blue, dashed] (gx2, {diag(gx1)-0.1}) -- (8.5, {diag(gx1)-0.1})--(8.5,{diag(7.7)})-- (7.7,{diag(7.7)});
	
	\path[name path=base] (0,0) to (10,0);
	
	\node (Tu)[scale=1] at (8.5,9.5) {$T_u$};
	\node (Tu)[scale=1] at (9.3,0.2) {$T_d$};

	\end{tikzpicture}
	\caption{Stylized plot of the functions $T_d$ and $T_u$ in the general case. $T_d$ and $T_u$ are given by the solid lines in the figure. Note that on the set $\{x:T_u(x)=x\}$ we have $T_d(x)=x$. In the figure the set $\{x: T_u(x)>T_d(x) \}$ is a finite union of intervals whereas in general it may be a countable union of intervals (and the set of endpoints of these intervals may have accumulation points). Similarly, in the figure $T_d$ has finitely many downward jumps, whereas in general it may have countably many jumps. 
Atoms of $\nu$ lead to horizontal sections of $T_d$ and $T_u$. Atoms of $\mu$, which are excluded in \cref{lem:LC} and \cref{cor:BJ}, but included in \cref{thm:construction} below, lead to vertical (multi-valued) sections of $T_d$ and $T_u$.}
	\label{fig:generalTdTu}
\end{figure}


Suppose $\nu$ is also continuous and fix $x\in\mathbb{R}$. Under the left-curtain martingale coupling, $\mu\lvert_{(T_d(x),x)}$ is mapped to $\nu\lvert_{(T_d(x), T_u(x))}$. Thus $\{ T_d(x), T_u(x) \}$ with $T_d(x) \leq x \leq T_u(x)$ are solutions to
\begin{align}
\int_{T_d(x)}^x \mu(dz) & =  \int_{T_d(x)}^{T_u(x)} \nu(dz),  \label{eq:mass} \\
\int_{T_d(x)}^x z \mu(dz) & =  \int_{T_d(x)}^{T_u(x)} z \nu(dz). \label{eq:mean}
\end{align}
Essentially, \eqref{eq:mass} is preservation of mass condition and \eqref{eq:mean} is preservation of mean and the martingale property.

In general, there can be multiple solutions to \eqref{eq:mass} and \eqref{eq:mean} although under the additional left-monotonicity properties of Definition~\ref{def:lmonfns}, for almost all $x\in\mathbb{R}$ there is a unique solution. (However, even then there may be exceptional $x$ at which $T_d$ jumps and at which there are multiple solutions.)

As observed by Henry-Labord\`{e}re and Touzi~\cite{HenryLabordereTouzi:16}, when $\mu$ and $\nu$ admit continuous densities $\rho_\mu$ and $\rho_\nu$ respectively, and $T_u$ and $T_d$ are smooth, we find that they satisfy the pair of coupled differential equations
\begin{eqnarray*}  \rho_\mu(x) - T_d'(x) \rho_\mu(T_d(x)) & = & T'_u(x) \rho_\nu(T_u(x)) - T_d'(x) \rho_\nu(T_d(x)) \\
x \rho_\mu(x) - T_d'(x) T_d(x) \rho_\mu(T_d(x)) & = & T_u(x) T'_u(x) \rho_\nu(T_u(x)) - T_d'(x) T_d(x) \rho_\nu(T_d(x)) .
\end{eqnarray*}
However, it remains to specify the initial conditions of the differential equations and even in the case of smooth densities, $T_d$ may have downward jumps at locations which depend on the global properties of $\mu$ and $\nu$. In the special case where $\mu$ and $\nu$ satisfy the dispersion assumption, see Hobson and Klimmek~\cite{HobsonKlimmek:15} or Hobson and Norgilas~\cite{HobsonNorgilas:18} (see also \cite[Section 3.4]{HenryLabordereTouzi:16} for a closely related condition), in the sense that there exists $I_> = (e_-,e_+)$ such that $\rho_\mu > \rho_\nu>0$ on $I_>$ and $0 \leq \rho_\mu \leq \rho_\nu$ otherwise, then the situation simplifies. In particular, $T_d(x) = x = T_u(x)$ for $x \leq e_-$ and on $[e_-,\infty)$, $T_d$ is strictly decreasing, $T_u$ is strictly increasing and together they solve
\[ T_d'(x) =  \frac{T_u(x) - T_d(x)}{T_u(x)-x} \frac{(\rho_\nu-\rho_\mu)(T_d(x))}{\rho_\mu(x)}\quad\textrm{and}\quad
T_u'(x) = \frac{T_u(x) - T_d(x)}{x - T_d(x)} \frac{\rho_\nu(T_u(x))}{\rho_\mu(x)} \]
subject to $T_d(e_-)= e_- = T_u(e_-)$ (see also \cite[Equations (3.9) and (3.10)]{HenryLabordereTouzi:16}).

Henry-Labord\`{e}re and Touzi~\cite[Equations (3.15) and (3.16)]{HenryLabordereTouzi:16}) are able to go further and write $T_d(x)$ as the root of an integral equation, (and then $T_u(x)$ can be deduced from $T_d(x)$). But the integral equation depends on the curve $\{ T_d(z) ; e_- \leq z \leq x \}$.
They are also able to extend beyond the dispersion assumption case by defining $T_d$ on intervals $[m_i,n_i]$
where $m_i$ is an element of a certain set $\bM_0=\bM_0(\mu,\nu)$ which, in the case where the densities are well defined, continuously differentiable and not identically equal on an interval, is the set of points where $\rho_\mu=\rho_\nu$ and $\rho_\mu'>\rho_\nu'$. But, the assumption that $\mu$ and $\nu$ are atom free is essential, and Henry-Labord\`{e}re and Touzi also assume that $\bM_0$ is finite, and to move beyond the case where $\bM_0$ can be written as $\bM_0 = \{ m_j, j \in \N: i<j \Leftrightarrow m_i<m_j \}$ (for example, to allow for accumulation points) would require further arguments. The fundamental question of how to determine $T_d$ and $T_u$ remains, especially since Lemma~\ref{lem:LC} and Corollary~\ref{cor:BJ} are purely existence statements.

\subsection{The left-curtain coupling in the presence of atoms}

In the case with atoms previous work of the authors gives an existence result similar in form to Corollary~\ref{cor:BJ}.

\begin{thm}[Hobson and Norgilas~{\cite[Theorem 1]{HobsonNorgilas:18}}]
\label{thm:construction}
Let $(\Omega,\sF,\Prob) = ((0,1) \times (0,1), \sB(\Omega), \Leb(\Omega))$. Let $\omega = (u,v)$ and let $(U,V)$ be the canonical random variables on $(\Omega,\sF,\Prob)$ given by $(U(\omega),V(\omega))=(u,v)$ so that $U$ and $V$ are independent $U(0,1)$ random variables. Let $\F = (\sF_0 = \{ \emptyset, \Omega \}, \sF_1 = \sigma(U), \sF_2 = \sigma(U,V) \})$ and set $\mathbf{M} = (\Omega, \sF, \F, \Prob)$.

Fix $\mu\leq_{cx}\nu$ and let $G=G_\mu$ be a quantile function of $\mu$.

Then there exists $R,S:(0,1)\mapsto\R$ such that the pair $(R,S)$ is left monotone with respect to $G$ on $\sI=(0,1)$ and such that if we define $X(u,v)=X(u)=G(u)$ and $Y(u,v) \in \{R(u),S(u) \}$ by $Y(u,v) = G(u)$ on $R(u)=S(u)$ and
\begin{equation}
Y(u,v) = R(u) I_{\{ v \leq \frac{S(u) - G(u)}{S(u)-R(u)} \}} +  S(u) I_{ \{ v > \frac{S(u) - G(u)}{S(u)-R(u)} \} }
\label{eq:YUVdef}
\end{equation}
otherwise, then
$M = (\bar{\mu},X(U),Y(U,V))$ is a $\mathbf{M}$-martingale for which $\sL(X) = \mu$ and $\sL(Y) = \nu$.
\end{thm}

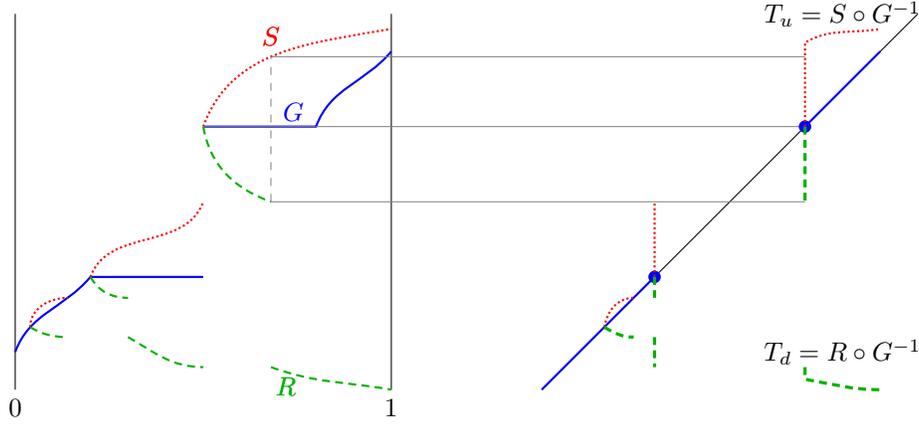
\begin{figure}[H]
\centering
\begin{tikzpicture}[scale=1,
declare function={	
	k1=2.1;
	k2=1;
	a(\x)=((k1-\x)*(\x<k1))-k1-1;
	b(\x)=((k2-\x)*(\x<k2))-k1-1;
	x1=-6.8;
	x2=-6.3;
	z1=-5.2;
	z2=-4.7;
}]

\draw[ black] (-7,0)--(-7,5);
\draw[black] (-2,0)--(-2,5);

\draw[name path=diag, black] (0,0) -- (5,5);

\draw[blue,thick, name path=g1] (-7,0.5) to[out=70, in=230] (-6,1.5) -- (-4.5,1.5) ;
\draw[blue,thick, name path=g2] (-4.5,3.5) -- (-3,3.5) to[out=70, in=230] (-2,4.5)  ;

\path [name path=lineA](x1,0) -- (x1,5);
\path [gray, very thin, name intersections={of=lineA and g1}] (x1,0) -- (intersection-1);
\coordinate (temp1) at (intersection-1);

\path [name path=lineA](x2,0) -- (x2,5);
\path [gray, very thin, name intersections={of=lineA and g1}] (x2,0) -- (intersection-1);
\coordinate (temp2) at (intersection-1);

\draw[red,thick, densely dotted,name path=s1] (temp1) to[out=90, in=180] (temp2) ;
\draw[red,thick, densely dotted,name path=s2] (-6,1.5) to[out=70, in=250] (-4.5,2.5) ;
\draw[red,thick, densely dotted,name path=s3] (-4.5,3.5) to[out=70, in=190] (-2,4.8) ;

\gettikzxy{(temp1)}{\l}{\k};
\gettikzxy{(temp2)}{\n}{\m};
\draw[black!30!green,thick, densely dashed, name path=r1] (temp1) to[out=330, in=180] (\n,0.7) ;
\draw[black!30!green,thick, densely dashed, name path=r2] (-6,1.5) to[out=300, in=180] (-5.5,\m) ;
\draw[black!30!green,thick, densely dashed, name path=r3] (-5.5,0.7) to[out=330, in=180] (-4.5,0.3) ;
\draw[black!30!green,thick, densely dashed, name path=r4] (-4.5,3.5) to[out=280, in=160] (-3.6,2.5) ;
\draw[black!30!green,thick, densely dashed, name path=r5] (-3.6,0.3) to[out=340, in=170] (-2,0) ;


\path[gray, very thin] (\n,0.7) -- (temp2);
\path[gray, very thin] (\n,0.7) -- (0.7,0.7);
\path[gray, very thin]  (temp2)--(\m,\m);
\path[gray, very thin]  (-2,4.5)--(4.5,4.5);

\path [name path=lineA](-5.5,0.7) -- (-5.5,5);
\path [gray, very thin, name intersections={of=lineA and s2}] (-5.5,0.7) -- (intersection-1);
\coordinate (vs2) at (intersection-1);
\gettikzxy{(vs2)}{\aa}{\bb}
\path[gray, very thin]  (vs2)--(\bb,\bb);

\path[gray, very thin] (-4.5,0.3) -- (-4.5,3.5);
\draw[gray, very thin] (-4.5,3.5) -- (3.5,3.5);
\path[gray, very thin] (-4.5,0.3) -- (0.3,0.3);

\path[gray, very thin] (-3.6,0.3) -- (-3.6,2.5) ;
\draw[gray, very thin] (-3.6,2.5) -- (3.5,2.5) ;

\path [name path=lineA](-3.6,2.5) -- (-3.6,5);
\draw [gray, very thin, dashed, name intersections={of=lineA and s3}] (-3.6,2.5) -- (intersection-1);
\coordinate (test1) at (intersection-1);
\gettikzxy{(test1)}{\testx}{\testy}
\draw[gray, very thin]  (test1)--(3.5,\testy);

\path [name path=lineA](-3,3.5) -- (-3,5);
\path [gray, very thin, name intersections={of=lineA and s3}] (-3,3.5) -- (intersection-1);
\coordinate (vs3) at (intersection-1);
\gettikzxy{(vs3)}{\aaa}{\bbb}
\path[gray, very thin]  (vs3)--(\bbb,\bbb);

\path [name path=lineA](-3,3.5) -- (-3,0);
\path [gray, very thin, name intersections={of=lineA and r5}] (-3,3.5) -- (intersection-1);
\coordinate (vr5) at (intersection-1);
\gettikzxy{(vr5)}{\aaaa}{\bbbb}
\path[gray, very thin]  (vr5)--(\bbbb,\bbbb);

\path[gray, very thin]  (vs2)--(\bb,\bb);
\path[gray, very thin]  (vs2)--(\bb,\bb);

\node [scale=0.5, shape=circle, fill, blue] at (3.5,3.5) {} ;
\node [scale=0.5, shape=circle, fill, blue] at (1.5,1.5) {} ;

\draw[blue,thick] (0,0) -- (1.5,1.5);
\draw[blue,thick] (3.5,3.5) -- (4.5,4.5);
\draw[red,densely dotted, thick] (\k,\k) to[out=80, in=180] (\m,\m);
\draw[red,densely dotted, thick] (1.5,1.5) -- (1.5,2.5);
\draw[red,densely dotted, thick] (3.5,3.5) -- (3.5,\bbb) to[out=30, in=190] (4.5,4.8);

\draw[black!30!green, densely dashed,very thick] (\k,\k) to[out=330, in=180] (\m,0.7);
\draw[black!30!green,densely dashed,very thick] (1.5,1.5) -- (1.5,\m);
\draw[black!30!green,densely dashed,very thick] (1.5,0.7) -- (1.5,0.3);
\draw[black!30!green,densely dashed,very thick] (3.5,3.5) -- (3.5,2.5);
\draw[black!30!green,densely dashed,very thick] (3.5,0.3) -- (3.5,\bbbb) to[out=350, in=180] (4.5,0);

\node[red] at (-3.6,4.7) {$S$};
\node[blue] at (-3.3,3.7) {$G$};
\node[black!30!green] at (-3.4,0.04) {$R$};

\node[red] at (-3.6,4.7) {$S$};
\node[black!30!green] at (-3.4,0.04) {$R$};

\node[below] at (-7,0) {$0$};
\node[below] at (-2,0) {$1$};
\node[black] at (4,0.5) {$T_d=R \circ G^{-1}$};
\node[black] at (4,5) {$T_u=S \circ G^{-1}$};
\end{tikzpicture}
\caption{Sketch of $R,G,S$ and the corresponding $T_u$ and $T_d$. On the atoms of $\mu$, $G$ is flat, and $T_d$ and $T_u$ are multi-valued, but $R$ and $S$ remain well-defined.}
\label{fig:RGSfg}
\end{figure}

Our goals in later sections are first to construct suitable candidate functions $R$ and $S$ satisfying left-monotonicity with respect to $G$, and second to show that they do indeed lead to a pair $(X,Y)$ with $X \sim \mu$ and $Y \sim \nu$.

\subsection{Lifted martingale transport plans}
Just as Corollary~\ref{cor:BJ} has an equivalent expression via Lemma~\ref{lem:LC}, Theorem~\ref{thm:construction} has an equivalent expression in terms of transport plans, provided we generalise the notion of a martingale transport plan.
Let $(\mu_u)_{0 \leq u \leq 1}$ be a family of measures with $\mu_u(\R)=u$, $\mu_1=\mu$ and $\mu_u \leq \mu_v$ for $0 \leq u \leq v \leq 1$, and let $\lambda$ denote Lebesgue measure on the unit interval. Then a \textit{lift} (Beiglb\"{o}ck and Juillet \cite{BeiglbockJuillet:16,BeiglbockJuillet:16s}) of $\mu$ with respect to $(\mu_u)_{0 \leq u \leq 1}$ is a probability measure $\hat\mu\in\Pi(\lambda,\mu)$ such that, for all $u\in[0,1]$ and Borel $A\subseteq\R$, $\hat\mu([0,u]\times A)=\mu_u(A)$. A {\it lifted martingale transport plan} is a probability measure $\hat{\pi} \in \Pi(\hat{\mu},\nu)$ such that $\int_{\R} y \hat{\pi}_{u,x}(dy)=x$, $\hat\mu\textrm{-a.e. }(u,x)$, where $\hat{\pi}_{u,x}$ denotes the disintegration of $\hat\pi\in\Pi(\hat\mu,\nu)$ with respect to $\hat\mu$: $\hat{\pi}(du,dx,dy) = \hat{\mu}(du,dx) \hat{\pi}_{u,x}(dy)$.

One of the insights of Beiglb\"{o}ck and Juillet \cite{BeiglbockJuillet:16,BeiglbockJuillet:16s} is that, for $(\mu_u)_{0 \leq u \leq 1}$ as above, the shadow measure induces a family of martingale couplings. In particular the idea is that for all $u\in[0,1]$, $\mu_u$ is mapped to $S^\nu(\mu_u)$. A crucial result making this possible is the fact that if $0<u<v<1$ and $\mu_u \leq \mu_v$ then $S^\nu(\mu_u) \leq S^\nu(\mu_v)$.

A natural choice for the lift $\hat{\mu}$ of $\mu$ is the {\it quantile lift} $\hat{\mu}^Q$ whose support is of the form $\{(v,G(v)): 0<v<1 \}$ where $G$ is a quantile function of $\mu$. Then $\hat{\mu}^Q(du,dx) = du \delta_{G(u)}(dx)$ and for a Borel set $A$, $\hat{\mu}^Q([0,w] \times A) = \int_{0}^w du I_{ \{G(u) \in A \}}$. Then, by Beiglb\"{o}ck and Juillet \cite[Theorem 1.1]{BeiglbockJuillet:16s}, there exists a unique lifted martingale transport plan $\hat{\pi}^Q$
such that for all $u\in[0,1]$ and Borel $A,B\subseteq\R$,  $\hat\pi^Q([0,u]\times A\times\R)=\mu_u(A)$ and $\hat\pi^Q([0,u]\times \R\times B)=S^\nu(\mu_u)(B)$. This is the left-curtain martingale coupling.

By analogy with the correspondence between Lemma~\ref{lem:LC} and Corollary~\ref{cor:BJ} we have the following equivalent restatement of Theorem~\ref{thm:construction}:
\begin{cor}\label{cor:construction}
Let $\mu,\nu$ be probability measures in convex order and let $\hat{\mu}^Q$ be the quantile lift of $\mu$. Then there exists a pair of measurable functions $R : \R \mapsto \R$ and $S : \R \mapsto \R$ such that $(R,S)$ is left-monotone with respect to $G=G_\mu$  and such that if
$\hat{\pi}^Q(du,dx,dy) = du \delta_{G(u)}(dx) \hat{\pi}^Q_{u,x}(dy)$ (recall $\hat{\mu}$ has support on $\{(u,G(u)) : 0 < u < 1 \}$) then
$\hat{\pi}^Q_{u,x}(dy) = \hat{\pi}^Q_{u,G(u)} (dy) = \chi_{R(u),G(u),S(u)}(dy)$ and $\hat{\pi}^Q$ is the lifted-left-curtain martingale transport plan which transports a second marginal $\mu$ to third marginal $\nu$.
\end{cor}

\section{The Geometric construction}
\label{sec:geometric}

Fix $\mu\leq_{cx}\nu$. The goal of this section is to construct candidates for the functions $R$ and $S$ of Theorem~\ref{thm:construction}. Then, in Section~\ref{sec:embed} we will prove that they can indeed be used to define a (left-monotone) martingale coupling of $\mu$ and $\nu$.

Recall the definition of $D(k)=P_\nu(k)-P_\mu(k)$, $k\in\R$. In what follows (and in the light of \cref{lem:chacon} and the subsequent discussion) we assume that $\{k\in\R:D(k)>0\}$ is an (open) interval, and
$\mu$ is supported on this set. Moreover, $\{k:D(k)>0\}=(\ell_\nu,r_\nu)$.

Recall also the definition of the sub-differential $\partial h(x)$ of a convex function $h:\R \mapsto \R$ at $x$:
\[ \partial h (x) = \{ \phi \in\R: h(y) \geq h(x) + \phi(y-x) \mbox{ for all } y \in \R \}. \]
We extend this definition to non-convex functions $f$ so that the subdifferential of $f$ at $x$ is given by
\[ \partial f (x) = \{ \phi\in\R : f(y) \geq f(x) + \phi(y-x) \mbox{ for all } y \in \R \}. \]
If $h$ is convex then $\partial h$ is non-empty everywhere, but this is not the case for non-convex functions. Instead we have that $\partial f(x)$ is non-empty if and only if $f(x)=f^c(x)$ and then $\partial f^c(x) = \partial f(x)$.

Let $G=G_\mu$ be a quantile function of $\mu$. (In Sections~\ref{sec:LeftCont} and \ref{sec:embed} we will take $G$ to be the left-continuous quantile function, but for now we let $G$ be {\em any} quantile function.)
For each $u\in(0,1)$, define $\mu_u\in\sM$ by
$$\mu_{u}(A)= \mu\Big(A \cap \big(-\infty,G(u)\big)\Big)+\Bigg(u-\mu\Big(\big(-\infty,G(u)\big)\Big)\Bigg)\delta_{G(u)}(A),\quad \textrm{for all Borel }A\subseteq\R.$$
Then for $u\in(0,1)$, $\mu_u\leq\mu$ and $\mu_u(\R)=u$. Note that $\mu_u$ does not depend on the choice of quantile function $G$.

 We have $P_{\mu_u}(k)=P_\mu(k)$ for $k\leq G(u)$, while $P_{\mu_u}(k)\leq P_\mu(k)$ for $k> G(u)$. In particular,
 \begin{equation*}
 P_{\mu_u}(k)=P_\mu(k\wedge G(u))+u(k-G(u))^+,\quad k\in\R,
 \end{equation*}
and thus, $P_{\mu_u}(\cdot)$ is linear on $[G(u),\infty)$ and $u\in\partial P_\mu(G(u))$, so that $P'_\mu(G(u)-)\leq u\leq P'_\mu(G(u)+)$.

For each $u \in (0,1)$ define $\sE_u:\R\mapsto\R_+$ by $\sE_u=P_\nu-P_{\mu_u}$, so that for $k \in \R$,
\[
 \sE_u(k)=P_\nu(k)-P_{\mu_u}(k)=D(k)+P_\mu(k)-P_{\mu_u}(k).
\]
Then, by \cref{thm:shadow_potential}, we have that
 \begin{equation*}
 P_{S^\nu(\mu_u)}(k)=P_\nu(k)-\sE^c_u(k),\quad k\in\R.
 \end{equation*}
The idea underlying this section is that we can hope to determine the functions characterising the left-curtain coupling by considering the properties of $\sE_u$ and $\sE_u^c$.

Note that $\sE_u(k)=D(k)$ for $k\leq G(u)$. Since $P_\mu-P_{\mu_u}$ is non-negative on $\R$, we have that $\sE_u(k)\geq D(k)$ for $k> G(u)$. Moreover, since $P_{\mu_u}$ is linear on $[G(u),+\infty)$, $\sE_u$ is convex on $(G(u),+\infty)$. It is also easy to see that $k\mapsto\sE_u(k)-D(k)$ is non-decreasing.

We now define candidate lower and upper functions. In fact we define two lower functions, which differ on a set of measure zero, either of which could be used in Theorem~\ref{thm:construction}. One of the lower functions is in keeping with the definition of $X$ in the study of convex hulls, but it turns out that the other is more convenient in the proof of Theorem~\ref{thm:construction}.

The idea is that typically $\sE_u$ is not convex, but we can define its convex hull. Moreover, commonly $\sE_u(G(u)) > \sE^c_u(G(u))$ in which case we can define $Q(u)$ to be the largest point to the left of $G(u)$ at which $\sE_u$ and its convex hull agree, and $S(u)$ to be the smallest point to the right of $G(u)$ where $\sE_u$ and $\sE^c_u$ agree. Then, when $Q(u)<S(u)$ we can define $\phi(u)$ to be the slope of $\sE^c_u$ over this interval. Typically, $Q(u)$ will be the only point $k$ below $G(u)$ such that $\sE_u(k) = \sE_u(S(u)) + (k-S(u))\phi(u)$, but in exceptional cases there may be other points with this property. In that case we let $R(u)$ be the smallest such point. See Figure~\ref{fig:RS}. Our first goal is to formulate the descriptions in this paragraph precisely, and in such a way that they apply to all situations, including pathological ones. The second goal is prove that the various quantities have certain properties, especially as $u$ varies.

Our motivation to study $R(u)$ and $S(u)$ stems from the fact that, at least in regular cases, by replacing $(T_d(x),x,T_u(x))$ with $(R(u),G(u),S(u))$, the mass and mean preservation conditions \eqref{eq:mass} and \eqref{eq:mean} hold. Indeed, suppose $\mu$ and $\nu$ are atomless with positive densities everywhere, so that, for each $u\in(0,1)$, $\sE_u$ is differentiable. Then if $R(u)<G(u)\leq S(u)$, by construction (see Figure \ref{fig:RS}) we have that
$$
\sE_u'(R(u))=\sE_u'(S(u))\quad\textrm{and}\quad \sE_u(R(u))+\sE_u'(R(u))(S(u)-R(u))=\sE_u(S(u)),
$$
which can be easily shown to be equivalent to \eqref{eq:mass} and \eqref{eq:mean}. Indeed, we have both \begin{align*}\sE_u'(R(u))&=P'_\nu(R(u))-P'_{\mu_u}(R(u))=P'_\nu(R(u))-P'_{\mu}(R(u))=\int^{R(u)}_{-\infty}\nu(dy)-\int^{R(u)}_{-\infty}\mu(dy),\\
	\sE_u'(S(u))&=P'_\nu(S(u))-P'_{\mu_u}(S(u))=P'_\nu(S(u))-P'_{\mu}(G(u))=\int^{S(u)}_{-\infty}\nu(dy)-\int^{G(u)}_{-\infty}\mu(dy),
	\end{align*}
	and therefore $
	\sE_u'(R(u))=\sE_u'(S(u))$ is equivalent to the mass preservation condition \eqref{eq:mass}. Similarly, by writing
	\begin{align*}\sE_u(R(u))=P_\nu(R(u))-P_{\mu_u}(R(u))&=P_\nu(R(u))-P_{\mu}(R(u))\\&=\int^{R(u)}_{-\infty}(R(u)-y)\nu(dy)-\int^{R(u)}_{-\infty}(R(u)-y)\mu(dy)
	\end{align*}
	and
	\begin{align*}
	\sE_u(S(u))&=P_\nu(S(u))-P_{\mu_u}(S(u))\\&=P_\nu(S(u))-\{P_{\mu}(G(u))+P'_{\mu_u}(G(u))(S(u)-G(u))\}\\
	&=\int^{S(u)}_{-\infty}(S(u)-y)\nu(dy)-\left\{\int^{G(u)}_{-\infty}(G(u)-y)\mu(dy)+(S(u)-G(u))\int^{G(u)}_{-\infty}\mu(dy)\right\}\\ &=\int^{S(u)}_{-\infty}(S(u)-y)\nu(dy)-\int^{G(u)}_{-\infty}(S(u)-y)\mu(dy),
	\end{align*}
	and using the mean preservation condition \eqref{eq:mass}, we have that that $\sE_u(R(u))+\sE_u'(R(u))(S(u)-R(u))=\sE_u(S(u))
	$ is equivalent to the mean preservation condition \eqref{eq:mean}.

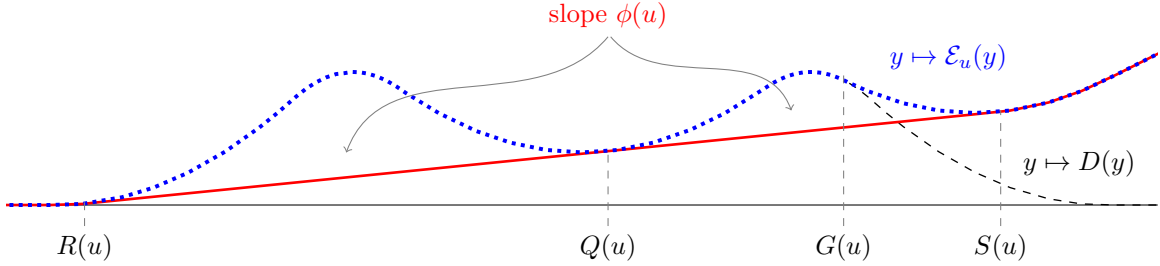
\begin{figure}[H]
	\centering
	\begin{tikzpicture}[scale=1]

\begin{axis}[%
width=6.028in,
height=2.754in,
at={(1.011in,0.642in)},
scale only axis,
xmin=-11,
xmax=11,
ymin=-0.5,
ymax=2,
axis line style={draw=none},
ticks=none
]
\draw[black, thin] (-11,0) -- (11,0);
\addplot [color=black, dashed, line width=0.5pt, forget plot]
table[row sep=crcr]{%
	5.10778631336004
	0.579120487313966\\
	5.28116588395275	0.550692356321193\\
	5.45454545454545	0.516281441424168\\
	5.79769279116164	0.441487336695587\\
	6.14084012777783	0.37232841678\\
	6.81818181818182	0.25309834481619\\
	7.50052689249177	0.156184065001039\\
	8.18181818181818	0.0826473956587641\\
	8.913369432737	0.0295195446611451\\
	9.54545454545455	0.00516182995117376\\
	10.1951138678739	-4.02019102274664e-06\\
	10.9090909090909	5.23608497005057e-06\\
	11.6245592768795	4.99949232768415e-06\\
	12.2727272727273	5.69424161334098e-06\\
	12.7873709632194	-0.000900998195628944\\
	13.6363636363636	2.55228360046544e-06\\
	14.3055157129462	-7.12227511812102e-06\\
	15	5.82376882718449e-06\\
};
\addplot [color=red, line width=1.0pt, forget plot]
table[row sep=crcr]{%
	-15	0\\
	-14.3743426864006	0\\
	-13.6363636363636	0\\
	-12.922520765984	0\\
	-12.2727272727273	0\\
	-11.6037511150776	0\\
	-10.9090909090909	0\\
	-10.1567767603786	0\\
	-9.54545454545455	0.00516529194875333\\
	-8.78555408111793	0.0241111129606366\\
	-8.18181818181818	0.0392044808885227\\
	-7.52404914247162	0.0556486746725015\\
	-6.81818181818182	0.0732953232255121\\
	-6.16274148089556	0.0896812995719799\\
	-5.45454545454546	0.107386165562501\\
	-4.69396922041941	0.12640053418326\\
	-4.09090909090909	0.141477007899491\\
	-3.35965597297158	0.159758300050986\\
	-2.72727272727273	0.17556785023648\\
	-2.00961780284731	0.193509188215844\\
	-1.36363636363636	0.20965869257347\\
	-0.688142581770708	0.226546004052749\\
	0	0.243749534910459\\
	0.681996956288809	0.260799348155818\\
	1.36363636363636	0.277840010546298\\
	2.07255602717534	0.295562666422883\\
	2.72727272727273	0.31193027388147\\
	3.46373839324672	0.330341566774406\\
	4.09090909090909	0.346020537216643\\
	4.76102717217463	0.362773171911045\\
	5.45454545454545	0.380110800551816\\
	6.14084012777783	0.397267842384896\\
	6.81818181818182	0.414201063886988\\
	7.50052689249177	0.431259367617357\\
	8.18181818181818	0.449124682527086\\
	8.913369432737	0.487440956412348\\
	9.54545454545455	0.542093630196611\\
	10.1951138678739	0.618135397881325\\
	10.9090909090909	0.707391398215165\\
	11.6245592768795	0.796824991527279\\
	12.2727272727273	0.877846509150189\\
	12.7873709632194	0.942167254247437\\
	13.6363636363636	1.04829872537446\\
	14.3055157129462	1.13389926036962\\
	15	1.21874929168467\\
};
\addplot [color=blue, dotted, line width=1.5pt, forget plot]
table[row sep=crcr]{%
	-15	0\\
	-14.3743426864006	0\\
	-13.6363636363636	0\\
	-12.922520765984	0\\
	-12.2727272727273	0\\
	-11.6037511150776	0\\
	-10.9090909090909	0\\
	-10.1567767603786	0\\
	-9.54545454545455	0.00516529194875333\\
	-8.78555408111793	0.0368719582992022\\
	-8.18181818181818	0.0826446406924073\\
	-7.52404914247162	0.153258338675332\\
	-6.81818181818182	0.253099065950084\\
	-6.16274148089556	0.368114076983243\\
	-5.80864346772051	0.43918676050732\\
	-5.45454545454546	0.516270463100197\\
	-5.26440139601394	0.553708788226936\\
	-5.07425733748243	0.58391669559351\\
	-4.88411327895092	0.606892956316372\\
	-4.69396922041941	0.622638584241668\\
	-4.54320418804183	0.626617195164236\\
	-4.39243915566425	0.632781558477449\\
	-4.24167412328667	0.631034673976273\\
	-4.09090909090909	0.624741555295822\\
	-3.90809581142471	0.611011938733503\\
	-3.72528253194034	0.590598341633919\\
	-3.54246925245596	0.563502600307279\\
	-3.35965597297158	0.532183251804731\\
	-2.72727272727273	0.435950274547483\\
	-2.00961780284731	0.350962188489822\\
	-1.36363636363636	0.296342723908122\\
	-0.688142581770708	0.261838896997191\\
	0	0.249996711366964\\
	0.681996956288809	0.261628059973472\\
	1.36363636363636	0.296487735376558\\
	2.07255602717534	0.357385147678734\\
	2.72727272727273	0.435946494150791\\
	3.46373839324672	0.549937949132153\\
	3.62053106766232	0.575890831516237\\
	3.77732374207791	0.597092931388165\\
	3.9341164164935	0.61337407422424\\
	4.09090909090909	0.624743965908923\\
	4.25843861122547	0.63190866150059\\
	4.42596813154186	0.632554779333081\\
	4.59349765185824	0.628031280250534\\
	4.76102717217463	0.617918105339684\\
	4.93440674276733	0.601515203917605\\
	5.10778631336004	0.580572723208292\\
	5.45454545454545	0.542107891512015\\
	6.14084012777783	0.483683429598855\\
	6.81818181818182	0.449121067486052\\
	7.50052689249177	0.437499922994881\\
	8.18181818181818	0.449124682527086\\
	8.913369432737	0.487440956412348\\
	9.54545454545455	0.542093630196611\\
	10.1951138678739	0.618135397881325\\
	10.9090909090909	0.707391398215165\\
	11.6245592768795	0.796824991527279\\
	12.2727272727273	0.877846509150189\\
	12.7873709632194	0.942167254247437\\
	13.6363636363636	1.04829872537446\\
	14.3055157129462	1.13389926036962\\
	15	1.21874929168467\\
};

\node (R)[scale=1] at (-9.5,-0.2) {$R(u)$};
\draw[gray,dashed,thin] (-9.5,-0.1) -- (-9.5,0.05);

\node (Q)[scale=1] at (0.5,-0.2) {$Q(u)$};
\draw[gray,dashed,thin] (0.5,-0.1) -- (0.5,0.27);

\node (G)[scale=1] at (5,-0.2) {$G(u)$};
\draw[gray,dashed,thin] (5,-0.1) -- (5,0.64);

\node (S)[scale=1] at (8,-0.2) {$S(u)$};
\draw[gray,dashed,thin] (8,-0.1) -- (8,0.45);

\node (E)[scale=1,blue] at (7,0.7) {$y\mapsto\sE_{u}(y)$};
\node (D)[scale=1,black] at (9.5,0.2) {$y\mapsto D(y)$};
\node (P)[scale=1,red] at (0.5,0.9) {slope $\phi(u)$};
\draw [gray,->] (0.5,0.8) to[out=230, in=70] (-4.5,0.25);
\draw [gray,->] (0.55,0.8) to[out=310, in=110] (4,0.45);
\end{axis}
	
	\end{tikzpicture}
	\caption{Plot of locations of $R(u)$, $Q(u)$, $G(u)$ and $S(u)$ in the case where $R(u)<Q(u)<G(u)<S(u)$. The dashed curve represents $D$. The dotted curve corresponds to the graph of $\sE_{u}$. Note that $D=\sE_u$ on $(-\infty, G(u)]$, while $\sE_u$ is convex and $D\leq\sE_u$ on $(G(u),\infty)$. The solid curve below $\sE_{u}$ represents $\sE^c_{u}$. The convex hull $\sE^c_{u}$ is linear on $[R(u),S(u)]$, and its slope is given by $\phi(u)$.}
	\label{fig:RS}
\end{figure}

Define $Q,S:(0,1)\mapsto\R$ by
\begin{align}
\label{eq:Q} Q(u) &:= X^{\sE_u}(G(u)) \\
\label{eq:S} S(u) &:= Z^{\sE_u}(G(u))
\end{align}

\begin{lem}
\label{lem:R=G=S}
$Q(u)=G(u)$ if and only if $S(u)=G(u)$.
\end{lem}

\begin{proof}
Fix $u\in(0,1)$.
By continuity of $\sE_u(\cdot)$ we have $\sE_u(Q(u)) = \sE^c_u(Q(u))$, and $\sE_u(S(u)) = \sE^c_u(S(u))$.

Suppose $Q(u)=G(u)$. Then $G(u) \in \{z: z \geq G(u), \sE^c_u(z) = \sE_u(z) \}$ and hence
$S(u) = \inf \{z: z \geq G(u), \sE^c_u(z) = \sE_u(z) \} = G(u)$.

The reverse implication follows by symmetry.
\end{proof}

We want to introduce a function $\phi:(0,1)\mapsto\R$ which represents the slope of $\sE^c_u(\cdot)$ at $G(u)$. If $Q(u)< G(u) < S(u)$, then this slope is well defined.
If $Q(u)=G(u)$ or $G(u)=S(u)$ then the slope of $\sE^c_u$ may not be well defined at $G(u)$. To cover all cases we define:
\begin{defn}
\label{def:slope}
$\phi:(0,1) \mapsto \R$ is given by $\phi(u) = \inf \{ \psi: \psi \in \partial \sE_u^c (G(u)) \}$.
\end{defn}

\begin{lem}
	\label{lem:phialt}
	$\phi(u) = (\sE^c_u)'(G(u)-) = (\sE^c_u)'(S(u)-)$.
\end{lem}

\begin{proof}
	The first equality is immediate from the definition of $\phi$, as is the second one provided $G(u)=S(u)$. On the other hand, if $G(u)<S(u)$, then $\sE^c_u=L^{\sE_u}_{Q(u),S(u)}$ on $[Q(u),S(u)]\supset\{G(u)\}$ and the second equality follows.
\end{proof}

Now we can introduce our second candidate lower function.

Recall the definition of $ L^f_{a,b}$ for any $f:\R\mapsto\R$ (see \eqref{eq:L1}), so that (in the case $a<b$) $L^f_{a,b}$ is the line passing through $(a,f(a))$ and $(b, f(b))$. Define also $L^{f,\psi}_a$ by $L^{f,\psi}_a(y) = f(a) + \psi(y-a)$ so that $L^{f,\psi}_a$ is the line passing through $(a,f(a))$ with slope $\psi$. (Note that, in the case $a=b$, $L^f_{a,a}=L^{f,0}_a$.)
Define $R:(0,1)\mapsto\R$ by
\begin{equation}    \label{eq:R}
R(u):=  \inf\{k:k \leq G(u), D(k) = {L}^{\sE^c_u ,\phi(u)}_{G(u)}(k) \}.
\end{equation}
If $Q(u)<G(u)$ then the definition of $R$ can be rewritten as $R(u)=   \inf\{k:k \leq G(u), D(k) = L^{\sE_u}_{Q(u),S(u)}(k) \}$.
Note that 
$Q(u) \in \{k:k \leq G(u), D(k) = {L}^{\sE^c_u,\phi(u)}_{G(u)}(k) \}$ so that $R(u)$ exists in all cases and satisfies $R(u) \leq Q(u)$. See Figure~\ref{fig:RS}.

Recall that $\ell_\nu$ and $r_\nu$ are the left- and right-hand endpoints of the interval $\{ k: D(k)>0 \}$.  The next lemma, the proof of which is postponed until Appendix \ref{ssec:proofsGC}, shows that $R$ and $S$ are finite on $(0,1)$.
\begin{lem}\label{lem:RSfinite}
	Fix $u\in(0,1)$. Either $-\infty<\ell_\nu\leq R(u)$  or $-\infty=\ell_\nu<R(u)$. Similarly, either $S(u) \leq r_\nu<-\infty$ or $S(u) < r_\nu=\infty$.
\end{lem}

If $Q(u)<S(u)$, then by construction, $\sE^c_u<\sE_u$ on $(Q(u),S(u))$ and $\sE^c_u \leq \sE_u$ on $[R(u),S(u)]$. In particular, $\sE^c_u$ is linear on $(R(u),S(u))$, whilst $\sE^c_u(S(u))=\sE_u(S(u))$, $\sE^c_u(Q(u))=\sE_u(Q(u))=D(Q(u))$ and $\sE^c_u(R(u))=\sE_u(R(u))=D(R(u))$. It follows that if $Q(u)<S(u)$ then
\begin{equation*}
\phi(u) = \frac{\sE_u(S(u))-D(Q(u))}{S(u)-Q(u)} = \frac{\sE_u(S(u))-D(R(u))}{S(u)-R(u)}.
\end{equation*}
Further, $\phi(u)$ is an element of each of $\partial \sE_u(R(u))$, $\partial \sE_u(Q(u))$ and $\partial\sE_u(S(u))$ together with $\partial \sE^c_u(R(u))$, $\partial \sE^c_u(Q(u))$ and $\partial\sE^c_u(S(u))$.

Our goal is to prove first that $(R,S)$ is left-monotone with respect to $G$ on $(0,1)$ in the sense of Definition~\ref{def:lmonfns} (Theorem~\ref{thm:monotone} below) and second that they define a martingale coupling of $\mu$ and $\nu$ (Theorem~\ref{thm:construction2} below). Together, these results give an explicit construction of a pair $(R,S)$ which solve the problem in Theorem~\ref{thm:construction} above.

We begin with some preliminary lemmas and other results. We are interested in properties of $u\mapsto\sE_u(k)$, for fixed $k\in\R$. Let $0<u<v<1$. Then, since $\mu_u \leq \mu_v$, $P_{\mu_u} \leq P_{\mu_v}$ and
\begin{equation}\label{eq:Eu}
\sE_u(k)\geq\sE_{v}(k),\quad k\in\R.
\end{equation}
Indeed, for $k\in\R$,
\begin{align}
&\sE_u(k)-\sE_{v}(k)=P_{\mu_{v}}(k)- P_{\mu_u}(k)\nonumber\\
&=\begin{cases}\label{eq:Ediff}
0,&\textrm{if }k\leq G(u);\\
P_\mu(k)-P_\mu(G(u))-u(k-G(u)),&\textrm{if }G(u)<k\leq G(v);\\
P_\mu(G(v)) - P_\mu(G(u))-v(G(v)-G(u))+(v-u)(k-G(u)),&\textrm{if }k > G(v).
\end{cases}
\end{align}
Note that, for $0<u<v<1$ and $k\geq G(v)$, \eqref{eq:Ediff} can be written as
\begin{equation}\label{eq:diffE}
\sE_u(k)-\sE_{v}(k)=(v-u)\left(k-R(u)\right)+\Gamma_{u,v},
\end{equation}
where
\[ \Gamma_{u,v} = P_\mu(G(v)) - \{ P_\mu(G(u)) + v ( G(v) - G(u) ) \} - (v-u)(G(u) - R(u)) \]
does not depend on $k$.
Convexity of $P_\mu$ ensures that $\Gamma_{u,v}\leq0$, and  if $R(u)<G(u)$ then $\Gamma_{u,v}<0$.
Further, from \eqref{eq:Eu} we also have that $\sE^c_{v} \leq \sE_{v} \leq \sE_u$, so that $\sE^c_{v}$ is a convex minorant of $\sE_u$ and
\begin{equation}\label{eq:hullEu}
\sE^c_u(k)\geq\sE^c_{v}(k),\quad k\in\R.
\end{equation}
Finally, for any $u\in(0,1)$, $\sE_u$ is defined as a difference of two convex functions, and thus its left and right derivatives exist. It follows that
for $v>u$ and $k \geq G(v)$,
\begin{equation}\label{eq:EslopeDiff+}
\sE_u^\prime(k+)-\sE_{v}^\prime(k+) = v-u.
\end{equation}
Again for $v>u$,
\begin{equation}\label{eq:EslopeDiff-}
 \sE_u^\prime(k-)-\sE_{v}^\prime(k-) = \left\{ \begin{array}{ll} v-u,  &\textrm{if } k > G(v); \\
                                                            P'_{\mu}(k-)-u,   &\textrm{if } k \in (G(u),G(v)]; \\
                                                            0, &\textrm{if } k\leq G(u).
                                                             \end{array} \right.
\end{equation}

\begin{lem}\label{lem:easyderivcomparison}
	Suppose $k \in [G(v-),G(v+)]$, $u<v$ and $G(u+)<k$. Suppose $\sE_v'(k-) \leq \sE'_v(k+)$ and $\psi \in [\sE_v'(k-) ,\sE'_v(k+)]$. Then $\sE_u'(k-) \leq \psi + (v-u) \leq \sE_u'(k+)$.
\end{lem}
\begin{proof}
	The result follows easily from the fact that $\sE_v'(k-) = P'_\nu(k-) - P'_{{ \mu_v}}(k-) \geq P'_\nu(k-) - v$ together with $\sE_v'(k+) = P'_\nu(k+) - v$ and
	$\sE_u'(k \pm) = P'_\nu(k\pm) - u$.
\end{proof}

The proof of the following lemma and its corollary are deferred until \cref{ssec:proofsGC}.
\begin{lem}Fix $v>u$. Suppose $\sE^c_u(z) = \sE^c_v(z) = \sE_u(z)$ for some $z$. Then $\sE^c_u = \sE^c_v$ on $(-\infty,z]$.
	\label{lem:chagree}
\end{lem}
\begin{cor}\label{cor:R=R}
	Suppose $v>u$ and $S(u) = S(v) \leq G(u+)$. Then $\sE^c_u = \sE^c_v$ on $(-\infty, S(v)]$. Further $\phi(u)=\phi(v)$ and $R(u)=R(v)$.
\end{cor}

For $0 < u < v <1$ let $\xi_{u,v}$ be given by
\[ (v - u) \xi_{u,v} = \int_{(G(u),G(v))} x \mu(dx) + G(u) \{ \mu(-\infty, G(u)] - u \} + G(v) \{ v - \mu(-\infty, G(v)) \}. \]
Then $\xi_{u,v}$ is the conditional mean of $\mu$ between the quantiles at $u$ and $v$. It is easily checked that $\xi_{u,v}$ does not depend on the choice of quantile function $G$.

\begin{lem}
\label{lem:tangentsmeet}
We have
\begin{enumerate}
\item[(i)] Suppose $k \in [G(v-), G(v+)]$ and $u<v$. Then the line passing through $(k,\sE_v(k))$ with slope $P_\nu'(k-)-v$ meets the line passing through $(k,\sE_u(k))$
with slope $P_\nu'(k-)-u$ at a point with $x$-coordinate $\xi_{u,v}$, which does not depend on $k$.

\item[(ii)] Take $k > G(v+)$ and suppose $u<v$. Then the line passing through $(k,\sE_v(k))$ with slope $\sE_v'(k-)$ meets the line passing through $(k,\sE_u(k))$
with slope $\sE_u'(k-)$ at a point with $x$-coordinate $\xi_{u,v}$, which does not depend on $k$.

\item[(iii)] Suppose $k \in [G(w-),G(w+)]$, $u<w$, $G(u+)<k$, $\sE_w'(k-) \leq \sE'_w(k+)$ and $\psi \in [\sE_w'(k-) ,\sE'_w(k+)]$. Then the line passing through $(k,\sE_w(k))$ with slope $\psi$ meets the line passing through $(k,\sE_u(k))$
with slope $\psi + (w-u)$ at a point with $x$-coordinate $\xi_{u,v}$, which does not depend on $k$.
\end{enumerate}
\end{lem}

\begin{proof}
(i) The line $L^{\sE_v,P_\nu'(k-) - v}_k$ is given by
\[ L^{\sE_v,P_\nu'(k-) - v}_k(z) = \sE_v(k) + (P_\nu'(k-) - v)(z-k) = P_\nu(k) + P_\nu'(k-)(z-k) - ( P_{\mu_v}(k) + v(z-k) ). \]
Similarly the line $L^{\sE_u,P_\nu'(k-) - u}_k$ is given by
\[ L^{\sE_u,P_\nu'(k-) - u}_k(z) = \sE_u(k) + (P_\nu'(k-) - u)(z-k) = P_\nu(k) + P_\nu'(k-)(z-k) - ( P_{\mu_u}(k) + u(z-k) ). \]
These lines intersect at the point where the lines $P_{\mu_v}(k) + v(z-k)$ and $P_{\mu_u}(k) + u(z-k)$ meet.
But for $k \geq G(v-)$, $P_{\mu_v}(k) + v(z-k) = P_{\mu_v}(G(v)) + v(z-G(v))$ and $P_{\mu_u}(k) + u(z-k) = P_{\mu_u}(G(u)) + u(z-G(u))$.

It is easily checked that the lines $P_{\mu_v}(G(v)) + v(z-G(v))$ and $P_{\mu_u}(G(v)) + u(z-G(u))$ meet at
$z =\xi_{u,v}$.

(ii) If $k>G(v+)$ then from the definition of $P_{\mu_v}$ it follows that $P'_{\mu_v}(k-)=v$ and therefore $\sE'_v(k-) = P_\nu'(k-)-v$. Since $k> G(v+)\geq G(u)$, the same argument shows that $\sE'_u(k-) = P_\nu'(k-)-u$. Then (ii) follows exactly as in the proof of (i).

(iii) This follows similarly to (i).

\end{proof}

 \begin{thm}\label{thm:monotone}
The pair $R,S:(0,1)\mapsto\R$, defined by \eqref{eq:R} and \eqref{eq:S} 
is left-monotone with respect to $G$ on $(0,1)$ in the sense of Definition~\ref{def:lmonfns}.
\end{thm}
\begin{proof}
That $R(u)\leq Q(u) \leq G(u)\leq S(u)$, $u\in(0,1)$, follows by definition.

Fix $0<u<v<1$.

If $R(u)=G(u)$ then necessarily $S(u)=G(u)$ (see Lemma~\ref{lem:R=G=S}). Then $S(v) \geq G(v) \geq G(u) = S(u)$. Further, $(R(u),S(u)) = \emptyset$ so that $R(v) \notin (R(u),S(u))$ by default.

There are two remaining cases, when $R(u) < G(u) \leq G(v) < S(u)$ and $R(u) < G(u) \leq S(u) \leq G(v)$.

\textit{Case 1:}  Suppose $R(u) < G(u) \leq G(v) < S(u)$.  We show that
\[ \sE_{v}(k)> L^{\sE_{v}}_{R(u),S(u)}(k) \]
for $k\in(R(u),S(u))$. Then it follows from Lemma~\ref{lem:convh2} that $R(v) \leq Q(v) =X^{\sE_{v}}(G(v)) \leq R(u) < S(u) \leq Z^{\sE_{v}}(G(v)) = S(v)$ as required.

First, for $k\in(R(u),G(u)]$ and since $\sE_u \geq \sE_{v}$ everywhere (with equality to the left of $G(u)$),
\begin{equation}
\label{eq:RneqQ1}  \sE_{v}(k)= \sE_u(k) \geq L^{\sE_u}_{R(u),S(u)}(k) > L^{\sE_{v}}_{R(u),S(u)}(k).
\end{equation}
with the strict inequality in \eqref{eq:RneqQ1} following from the fact that $S(u)>G(v)$ and $\sE_{v} < \sE_u$ on $(G(v),\infty)$.

Second, from \eqref{eq:EslopeDiff-} we have that
\begin{eqnarray}
\sE_{v}^\prime(S(u)-)&=& \sE_{ u}^\prime(S(u)-)-(v-u) \nonumber \\
&\leq & \frac{\sE_u(S(u))-D(R(u))}{S(u)-R(u)}-(v-u) \nonumber \\
&= & \frac{\Gamma_{u,v}+(v-u)(S(u)-R(u))+\sE_{v}(S(u))-D(R(u))}{S(u)-R(u)}-(v-u) \nonumber \\
& < & \frac{\sE_{v}(S(u))-D(R(u))}{S(u)-R(u)} 
\label{eq:EuSlope}
\end{eqnarray}
where we use the fact that $\Gamma_{u,v}<0$ for $R(u)<G(u)$.
We conclude that $\sE'_{v}(S(u-)) < \frac{\sE_{v}(S(u))-D(R(u))}{S(u)-R(u)}$ which is the slope of $L^{\sE_{v}}_{R(u),S(u)}$.
Then from the convexity of $\sE_{v}$ on $[G(v),S(u)]$ we have that $\sE_{v}(k) > L^{\sE_{v}}_{R(u),S(u)}(k)$ for $k \in [G(v),S(u))$.

It remains to show that $\sE_{v}(k) > L^{\sE_{v}}_{R(u),S(u)}(k)$ for $k \in (G(u),G(v))$. This will follow from the following pair of inequalities which are valid on $(G(u),G(v))$:
\begin{equation}
\label{eq:pair<}
D = \sE_{v} \geq L^{\sE_{v}}_{R(u),G(v)} > L^{\sE_{v}}_{R(u),S(u)}.
\end{equation}
The second of these inequalities is valid on $(R(u),S(u))$ and follows from the fact that $\sE_{v}$ is convex on $(G(v),S(u))$ and \eqref{eq:EuSlope}.

Consider, therefore, the first inequality in \eqref{eq:pair<}. 
We consider two subcases, namely when $D(G(v)) \geq L^{\sE_u}_{R(u),S(u)}(G(v))$ and when $D(G(v)) < L^{\sE_u}_{R(u),S(u)}(G(v))$.

\textit{Case 1a:} $D(G(v)) \geq L^{\sE_u}_{R(u),S(u)}(G(v))$. Either $D \geq L^{\sE_v}_{R(u),G(v)}$ on $(G(u),G(v))$ in which case we are done, or there exists $y$ with $G(u)<y \leq G(v)$ with $D(y) = L^{\sE_v}_{R(u),G(v)}(y)$ and $D'(y-) \geq (L^{\sE_v}_{R(u),G(v)})'$ whence
\[ D'(y-) \geq (L^{\sE_v}_{R(u),G(v)})' \geq (L^{\sE_{u}}_{R(u),S(u)})' \geq \sE_u'(S(u)-) \geq \sE_u'(y+) \geq \sE'_u(y-). \]
Indeed by minimality of $S(u)$, $(L^{\sE_{u}}_{R(u),S(u)})' >  \sE_u'(y+)$.
Then, $D'(y-)>\sE'_{u}(y-)$. But, this would mean that $P'_{\nu}(y-) - P'_{\mu}(y-) > P'_{\nu}(y-) - P'_{\mu_u}(y-)$ or equivalently $u > P'_\mu(y-)$. Since $y > G(u)$ we have $P'_\mu(y-) \geq P'_\mu(G(u)+) \geq u$, yielding a contradiction. Hence in \textit{Case 1a:} we have $D\geq L^{\sE_v}_{R(u),G(v)}$ on $(G(u),G(v))$ as required.

\textit{Case 1b:} $D(G(v)) < L^{\sE_u}_{R(u),S(u)}(G(v))$.
Again, either $D \geq L^{\sE_v}_{R(u),G(v)}$ on $(G(u),G(v))$ in which case we are done, or there exists $y$ with $G(u)<y \leq G(v)$ with $D(y) = L^{\sE_v}_{R(u),G(v)}(y)$ and $D'(y-) \geq \frac{\sE_v(G(v)) - D(R(u))}{G(v)-R(u)}$ which is the slope of $L^{\sE_v}_{R(u),G(v)}$. Let $w \in (u,v]$ be such that $G(w-) \leq y \leq G(w+)$ and $P_\mu'(y-) = w$. Then $\sE_w'(y-) = P'_{\nu}(y-) - w$. Further, since $G(u)<y$, we have $P'_{\mu_u}(y-)=u$ and therefore $\sE'_u(y-)=P_\nu'(y-)-u$.

Now consider the lines $L^{D,P'_{\nu}(y-) - w}_y \equiv L^{\sE_w,P'_{\nu}(y-) - w}_y$ and $L^{\sE_u,P'_{\nu}(y-) - u}_y$, and note that $D=\sE_w$ on $(-\infty,y]$. Since $D'(y-) \geq (L^{\sE_v}_{R(u),G(v)})'$, $L^{D,P'_{\nu}(y-) - w}_y$ lies on or below $L^{\sE_v}_{R(u),G(v)}$ to the left of $y$ and $L^{\sE_v}_{R(u),G(v)}$ lies on or below $L^{\sE_{u}}_{R(u),S(u)}$ on $(R(u),G(v))$. In contrast, $L^{\sE_u,P'_{\nu}(y-) - u}_y$ lies above $L^{\sE_{u}}_{R(u),S(u)}$ to the left of $y$ by convexity of $\sE_u$ on $[G(u),\infty)$. Hence, if $L^{\sE_w,P'_{\nu}(y-) - w}_y$ and $L^{\sE_u,P'_{\nu}(y-) - u}_y$ meet then they must do so outside $(R(u),y]$. But, by Lemma~\ref{lem:tangentsmeet}(i) they meet at $\xi_{u,w} \in [G(u),G({w})]$, and since $\xi_{u,w}$ does not depend on the choice of quantile function $G$ we must have that $\xi_{u,w} \in [G(u),G({w-})]$; a contradiction. Hence
$D \geq L^{\sE_v}_{R(u),G(v)}$ on $(G(u),G(v))$ as required.

\textit{Case 2:} Now suppose $R(u) < G(u) \leq S(u)\leq G(v)$.
Then $S(v) \geq G(v) \geq S(u)$ so that all that remains to be shown is that $R(v) \notin (R(u),S(u))$.

Suppose to the contrary that $R(u) < R(v) < S(u)$. Consider the lines $L^{\sE_w}_{R(w),S(w)}$ for $w \in \{u,v\}$. Recall that $\phi(w)$ is the slope of $L^{\sE_w}_{R(w),S(w)}$. Note that, since $\sE^c_w =L^{\sE_w}_{R(w),S(w)}$ on $[R(w),S(w)]$, for $w\in\{u,v\}$, and $\sE^c_u\geq\sE^c_v$ everywhere, we have that $L^{\sE_u}_{R(u),S(u)}\geq L^{\sE_v}_{R(v),S(v)}$ on $[R(v),S(u)]$.

Suppose first that $D(R(v)) \geq L^{\sE_u}_{R(u),S(u)}(R(v))$. (This will follow if, for instance, $R(v) \leq G(u+)$ for then since $R(v) > R(u)$ by hypothesis,
\( D(R(v)) = \sE_u(R(v)) \geq L^{\sE_u}_{R(u),S(u)}(R(v)). \))
Then $D(R(v)) = \sE_v(R(v)) = L^{\sE_v}_{R(v),S(v)}(R(v))\leq L^{\sE_u}_{R(u),S(u)}(R(v))\leq D(R(v))$.
Therefore $L^{\sE_u}_{R(u),S(u)}(R(v)) = L^{\sE_v}_{R(v),S(v)}(R(v))$, and since $L^{\sE_u}_{R(u),S(u)}\geq L^{\sE_v}_{R(v),S(v)}$ to the right of $R(v)$, we must have that $\phi(v)\leq \phi(u)$. If $\phi(v)<\phi(u)$ then $D(R(u))=L^{\sE_u}_{R(u),S(u)}(R(u))< L^{\sE_v}_{R(v),S(v)}(R(u))\leq\sE^c_v(R(u))\leq \sE_v(R(u))=D(R(u))$, a contradiction. On the other hand, if $\phi(v)=\phi(v)$ then the lines $L^{\sE_u}_{R(u),S(u)}$ and $L^{\sE_v}_{R(v),S(v)}$ are identical. But then $R(u)=R(v)$, a contradiction.

Now suppose $D(R(v)) < L^{\sE_u}_{R(u),S(u)}(R(v))$. Since $L^{\sE_u}_{R(u),S(u)}(R(v))\leq\sE_u(R(v))$ we must have that $G(u+)<R(v)$, and then by the minimality of $S(u)$ it follows that $L^{\sE_u}_{R(u),S(u)}(R(v))<\sE_u(R(v))$. Since $v>u$ and $R(v)<S(u)$ there exists $w \in (u,v]$ with $G(w-)\leq R(v) \leq G(w+)$ and then we have from the twin facts that $\sE_w(R(v))=D(R(v))$ and $\sE_{w}^c \geq \sE^c_v \geq L^{D,\phi(v)}_{R(v)}$ that $\sE'_w(R(v)-) \leq \phi(v) \leq \sE'_w(R(v)+)$. It follows from Lemma~\ref{lem:tangentsmeet}(iii) with $k=R(v)$ and $\psi = \phi(v)$ that the line passing through $(R(v),\sE_w(R(v)))$ with slope $\phi(v)$ meets the line passing through $(R(v),\sE_u(R(v)))$
with slope $\phi(v) + (w-u)$ at a point with $x$-coordinate $\xi_{u,w}\in[G(u),G(w)]$. But since $\xi_{u,w}$ does not depend on the choice of quantile function $G$ we have that $\xi_{u,w}\in[G(u),G(w-)]$, and then by noting that $\xi_{u,w}\neq R(v)$ it follows that $\xi_{u,w}\in[G(u),R(v))$.
But, from the convexity of $\sE_u$ on $[G(u),\infty)$ and the fact that by Lemma~\ref{lem:easyderivcomparison} $\sE_u'(R(v)-) \leq \phi(v) + w-u \leq \sE'_u(R(v)+)$, we have that $L^{\sE_u, \phi(v)+w-u}_{R(v)} > L^{\sE_u}_{R(u),S(u)}$ on $(-\infty,R(v)]$ and hence $L^{D,\phi(v)}_{R(v)} \equiv L^{\sE_v}_{R(v),S(v)} $ crosses $L^{\sE_u}_{R(u),S(u)}$ in $[G(u), R(v))$. Then
\[ \sE_u(R(u)) = \sE_v(R(u)) \geq L^{\sE_v}_{R(v),S(v)}(R(u)) >  L^{\sE_u}_{R(u),S(u)}(R(u)) =  \sE_u(R(u)), \]
a contradiction.

We conclude that $R(v) \notin (R(u),S(u))$.
\end{proof}

\begin{figure}[H]
		\centering
	\begin{tikzpicture}
	
	\begin{axis}[%
	width=6.028in,
	height=2.754in,
	at={(1.011in,0.642in)},
	scale only axis,
	xmin=-11,
	xmax=11,
	ymin=-0.5,
	ymax=2,
	axis line style={draw=none},
	ticks=none
	]
	\draw[black, thin] (-11,0) -- (11,0);
	\draw[black] (-11,0) to[out=0, in=180] (-5,1) to[out=0, in=180] (0,0.5) to[out=0, in=180] (5,1) to[out=0, in=180] (11,0) ;
	\draw[red,thick] (0.61,0.51)--(9,0.95);
	\draw[gray,thin, dashed] (9,0.95)--(9,-0.1);
	
	\node[black] at (9,-0.2) {$S(v)$};
		\draw[gray,thin, dashed] (0.61,0.51)--(0.61,-0.1);
		\node[black] at (0.61,-0.2) {$R(v)$};
	\draw[blue,dotted, very thick] (6,0.95) to[out=345, in=230] (12,1.5);
	\draw[gray,thin, dashed] (6,0.95)--(6,-0.1);
\node[black] at (6,-0.2) {$G(v)$};
	\node[blue] at (9.5,1.3) {$k\mapsto \sE_v(k)$};
	\node[black] at (-9,0.7) {$k\mapsto D(k)$};
	\draw[red,thick] (-10,0.047)--(3,1.2);
	\draw[gray,thin, dashed] (-10,0.047)--(-10,-0.1);
	\draw[gray,thin, dashed] (3,1.2)--(3,-0.1);
	\node[black] at (3,-0.2) {$S(u)$};
	\node[black] at (-10,-0.2) {$R(u)$};
		\draw[blue,dotted, very thick] (-4,0.96) to[out=355, in=205] (7,1.65);
		\node[blue] at (4,1.6) {$k\mapsto \sE_u(k)$};
		\draw[gray,thin,dashed] (-4,0.96) -- (-4,-0.1);
		\node[black] at (-4,-0.2) {$G(u)$};
	\end{axis}
	\end{tikzpicture}
\caption{A potential counterexample to the left-monotonicity of $R$ and $S$. In the figure, for a pair of arbitrary curves $(\sE_u,\sE_v)$, which are convex beyond $G(u)$ and $G(v)$ respectively, we have $R(v)\in (R(u),S(u))$. But this is not a feasible pair since $\sE_u$ and $\sE_v$ are not generated from a pair of distributions $\mu\leq_{cx}\nu$, and, in particular, $\sE_u-\sE_v$ is not convex.}
\label{fig:whatgoeswrong}
\end{figure}
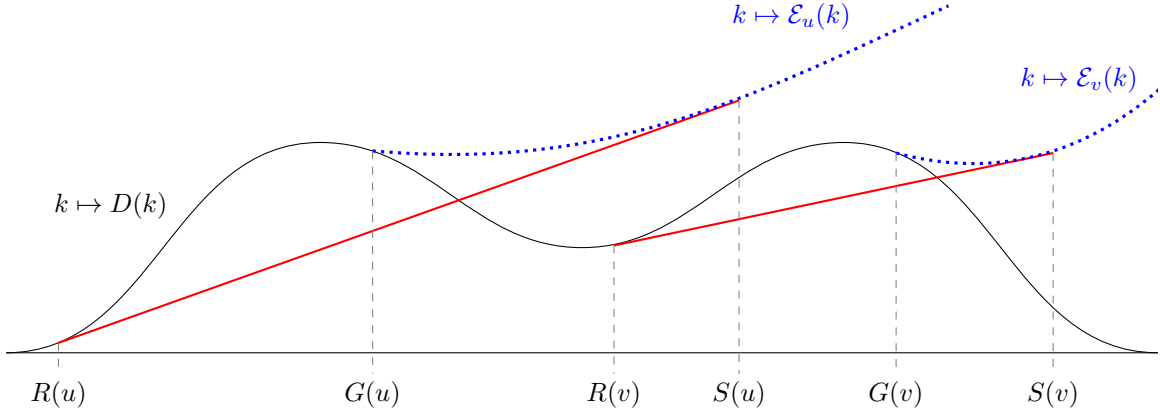

It is possible to draw pairs of curves $\sE_u$ and $\sE_v$ which have most of the correct properties (for example, $\sE_v \leq \sE_u$ with equality to the left of $G(u+)$) and for which $R(v) \in (R(u),S(u))$. See Figure~\ref{fig:whatgoeswrong}. But crucially $\sE_v$ and $\sE_u$ in Figure~\ref{fig:whatgoeswrong} do not satisfy $\sE_u - \sE_v$ is convex. The extra structure described in \eqref{eq:EslopeDiff-} and Lemma~\ref{lem:tangentsmeet} makes counterexamples such as the one in the figure infeasible.


By definition $G$ is increasing and by Theorem~\ref{thm:monotone} the same property holds for $S$. Further, for $v>u$, $R(v) \notin (R(u),S(u))$, and so, except at places where $R=G=S$, and except at points where $R$ jumps upwards, we expect $R$ to be decreasing. We now argue that the set where $G<S$ can be divided into a union of disjoint intervals on which $R$ is decreasing.

The functions $G$ and $S$ are monotonic, so we can define left and right limits. If $G(u+) < S(u-)$ then there exists $v$ with $v>u$ such that $G(v+) < S(u-) \leq S(v-)$. Conversely, there exists $w$ with $w<u$ such that
$S(w-) > G(u+) \geq G(w+)$. Then $\{ u : G(u+) < S(u-) \}$ is open, and since each such interval contains a rational we can write $A_< := \{u:G(u+)<S(u-)\}$ as a countable union of disjoint open sets:
\begin{equation}\label{eq:S<U}
A_< := \{u\in(0,1):G(u+)<S(u-)\}=\bigcup_{n\geq1}A^n_<.
\end{equation}
As the next lemma shows, $R$ is decreasing on each of these sets.  Moreover, for $u,v \in A_<^n$ with $u<v$, $R(v)<G(u)$. The proof of Lemma \ref{lem:RdecreasingA<} is given in Appendix \ref{ssec:proofsGC}.
\begin{lem}\label{lem:RdecreasingA<}
	For each $n\geq 1$, $R(\cdot)$ is decreasing on $A^n_<$.
	\end{lem}

The results of this section (especially Theorem~\ref{thm:monotone}) give functions $R,S$ which are left-monotone with respect to $G$. The remaining task is to show that they define a martingale transport for $\mu$ to $\nu$, see Theorem~\ref{thm:construction}. The next two sections give further characterisations and regularity results on the functions $\phi$, $R$ and $S$.

\begin{rem}\label{rem:measurability}
The monotonicity of $S$ (see Theorem \ref{thm:monotone}) implies that $S$ is (Borel) measurable. On other hand, by Lemma \ref{lem:RdecreasingA<}, the restriction of $R$ to $A_<$ is also measurable. It turns out that, in order to prove Theorem \ref{thm:construction}, global measurability of $R$ is not necessary (see Section \ref{sec:embed}).
\end{rem}

\section{Properties of $\phi$}
\label{sec:phiprop}
The goal in this section is to give some further representations and properties of $\phi(\cdot)$. In particular, although $\phi$ can jump upwards on $(0,1)$ it is decreasing and Lipschitz continuous on $A_<$ and has a derivative almost everywhere on $A_<$ which we can identify in terms of a rational function of $R$, $G$ and $S$.

\begin{lem}\label{lem:phiRepresentation}
We have
\begin{equation}
\label{eq:phiR}
\phi(u)=   \sup_{k < G(u)} 
\frac{\sE_u(S(u))-D(k)}{S(u)-k},\quad u\in(0,1),
\end{equation}
and
\begin{equation}
\label{eq:phiS}
\phi(v) \leq \inf_{k > G(v)} \frac{\sE_v(k)-D(R(v))}{k-R(v)},\quad v\in(0,1),
\end{equation}
with equality in \eqref{eq:phiS} if $R(v)<G(v)$. Moreover, if $u,v \in A_<^n$ for some $n$, with $u<v$, then in \eqref{eq:phiS} the infimum over $k>G(v)$ can be extended to an infimum over $k>G(u)$ and we have
\begin{equation}
\label{eq:phiS2}
\phi(v) = \inf_{k > G(v)} \frac{\sE_v(k)-D(R(v))}{k-R(v)} = \inf_{k > G(u)} \frac{\sE_v(k)-D(R(v))}{k-R(v)}.
\end{equation}

\end{lem}

\begin{proof}

First consider \eqref{eq:phiR}.

If $R(u)<G(u)$ then $\sE_u \geq \sE^c_u \geq L^{\sE_u}_{R(u),S(u)}$ everywhere and $\sE^c_u = L^{\sE_u}_{R(u),S(u)}$ on $[R(u),S(u)]$.
Then, for $k<G(u)$,
\begin{equation}
 \frac{\sE_u(S(u))-\sE_u(k)}{S(u)-k} 
 \leq \frac{L^{\sE_u}_{R(u),S(u)}(S(u))-L^{\sE_u}_{R(u),S(u)}(k)}{S(u)-k} = \frac{\sE_u(S(u))-\sE_u(R(u))}{S(u)-R(u)} = \phi(u)
 \label{eq:phiineq}
 \end{equation}
so that \eqref{eq:phiR} holds.

Now suppose that $R(u)=G(u)$, and hence that $\sE^c_u(G(u))=\sE_u(G(u))$ and $G(u)=S(u)$. Let $\phi^* = \sE'_u(G(u)-)$. We show $\phi(u) = \phi^* =\sE'_u(G(u)-)$ and that \eqref{eq:phiR} holds. Clearly if $\psi \in \partial \sE^c_u(G(u))$ then $\psi \geq \frac{\sE^c_u(S(u)) - \sE_u(k)}{S(u)-k}$ for $k<G(u)$ and then
\[ \psi \geq  \sup_{k < G(u)} 
\frac{\sE_u(G(u))-D(k)}{G(u)-k}  \geq   \lim_{k \uparrow G(u)} \frac{\sE_u(G(u))-\sE_u(k)}{G(u)-k} = \sE'_u(G(u)-) = \phi^*. \]
Hence, since $\phi(u)\in\partial\sE^c_u(G(u))$ by definition, to show that $\phi(u)=\phi^*$ it  is sufficient to show that $\phi^* =\sE'_u(G(u)-) \in \partial \sE^c_u(G(u))$. We show
\begin{equation}
\label{eq:candslope}
 \sE_u(y) \geq \sE_u(G(u))+ \phi^*(y-G(u)),
\end{equation}
for $y<G(u)$ and $y>G(u)$ separately.

Suppose \eqref{eq:candslope} fails for some $r_0<G(u)$. Then there exists $\epsilon>0$ such that
\[ \sE_u(r_0) < \sE_u(G(u))+ (\phi^*+\epsilon)(r_0-G(u)). \]
Let $r_1$ be the largest solution (with $r<G(u)$) of $\sE_u(r) = \sE_u(G(u))+ (\phi^*+\epsilon)(r-G(u))$. Then from consideration of the slope of $\sE_u$ near $G(u)$, $r_1<G(u)$; further, from the property of $r_1$ as the largest solution we have $\sE_u > L^{\sE_u}_{r_1,G(u)}$ on $(r_1,G(u))$. Then
$Q(u) = Y^{\sE_u}(G(u)) \leq r_1$ by Lemma~\ref{lem:convh2} and {\em a fortiori} $R(u)<G(u)$. But, this contradicts our assumption that $R(u)=G(u)$.

Now we show \eqref{eq:candslope} for $y>G(u)$ in the case $R(u)=G(u)=S(u)$. For $y>G(u)$, since $\sE_u = \sE_u^c$ on $[G(u),\infty)$,
\begin{eqnarray*}
\sE_u(y)= \sE_u^c(y) & \geq & \sE_u(G(u)) + (\sE^c_u)'(G(u)+)(y-G(u)) \\
 & \geq & \sE_u(G(u)) + (\sE^c_u)'(G(u)-)(y-G(u)) \\
& \geq & \sE_u(G(u)) + \phi^*(y-G(u))
\end{eqnarray*}
and we are done.

Hence $\phi(u)=\phi^*$. It remains to show \eqref{eq:phiR}. But
\[ \sE'_u(G(u)-) = \lim_{k \uparrow G(u)} \frac{\sE_u(G(u))-\sE_u(k)}{G(u)-k} \leq \sup_{k<G(u)} \frac{\sE_u(S(u))-\sE_u(k)}{S(u)-k} \leq \phi(u) = \phi^* \] and \eqref{eq:phiR} holds.

Now consider \eqref{eq:phiS}. If $R(v)<G(v)$ then
interchanging the roles of $R$ and $S$, by analogy with \eqref{eq:phiineq} we obtain for $k \geq G(v)$
\[ \frac{\sE_v(k) - \sE_v(R(v))}{k-R(v)} \geq \phi(v) \]
and \eqref{eq:phiS} follows. By considering $k=S(v)$ we see there is equality. If $R(v) = G(v)=S(v)$, then $\sE_v(G(v)) = \sE^c(G(v))$ and then if $\psi \in \partial \sE_v(G(v))$ we have that, for all $k\in\R$, 
\[ \sE_v(k) \geq \sE_v(G(v)) + \psi(k-G(v)) = \sE_{ v}(R(v)) + \psi(k-R(v)).\] Hence, for $k > G(v)=R(v)$, we have that $\phi(v) \leq \psi \leq \frac{\sE_v(k) - \sE_v(R(v))}{(k-R(v))}$ and \eqref{eq:phiS} follows.

For the final part, fix $u<v$ with $u,v \in A_<^n$. Then $R(v)\leq Q(v)<G(v)<S(v)$ and we have equality in \eqref{eq:phiS} so that $\phi(v) = \inf_{k>G(v)}  \frac{\sE_v(k) - \sE_v(R(v))}{k-R(v)}$. Moreover, $\sE_v \geq \sE^c_v$ on $[R(v),S(v)]$ and $R(v) \leq R(u) < G(u)$ so that $\sE_v \geq \sE^c_v$ on $[G(u),G(v)]$ and
$\inf_{k \in [G(u),G(v)]} \frac{\sE_v(k) - \sE_v(R(v))}{k-R(v)} \geq \frac{\sE_v(S(v)) - \sE_v(R(v))}{S(v)-R(v)} = \phi(v)$. Then the second equality in \eqref{eq:phiS2} follows.
\end{proof}

\begin{lem}\label{lem:phi}
\begin{enumerate}
\item[(i)] $\lim_{u \uparrow 1} \phi(u)=0$.
\item[(ii)] For $u,v \in (0,1)$ with $u<v$ we have $\phi(v) \geq  \phi(u) - (v - u)$.
\item[(iii)]
On each $A_<^n$, $n\geq1$, for $u,v\in A_<^n$ with $u<v$, we have $\phi(v)\leq\phi(u)$ so that $\phi$ is non-increasing.
\end{enumerate}
\end{lem}

In light of (i) above we may extend the domain of $\phi$ to $(0,1]$ by setting $\phi(1)=0$, and (ii) still holds, even with $v=1$.

\begin{cor}\label{cor:phiabs}
On each $A_<^n$ we have that $\phi$ is Lipschitz continuous. Moreover, $\phi$ is absolutely continuous on $A_<^n$ and there exists a function $\phi': A^n_< \rightarrow \R$ such that for $u,v \in A_<^n$,
\[
\phi(v) - \phi(u) = \int_u^{v} \phi'(w) dw.
\]
\end{cor}

\begin{proof}[Proof of Lemma~\ref{lem:phi}]
(i) First consider the limit on $\phi$. We have $0 \leq \phi(u) \leq \sE'_u(S(u)+) = P'_\nu(S(u)+)-u \leq 1-u$, and the result follows.

(ii)
Suppose $G(u)=S(v)$ and then $G(u)=G(v)=S(u)=S(v)$. Then ${L}^{\sE_v,\phi(v)}_{G(v)}$ lies below $\sE_v$, which in turn lies below $\sE_u$ (with equality at $G(u)=G(v)$) and so $\phi(v) \in \partial \sE_u(G(u))$. Then $\phi(v) \geq \phi(u) \geq \phi(u) - (v-u)$.

Suppose $G(u)<S(v)$ and $R(u) < G(v)$. Then from Lemma~\ref{lem:phiRepresentation},
\begin{align*}
\phi(v)&=\sup_{k < G(v)}\frac{\sE_{v}(S(v))-D(k)}{S(v)-k}\\
&\geq\frac{\sE_{v}(S(v))-D(R( u))}{S(v)-R( u)}\\
&=\frac{\sE_{u}(S(v))-D(R( u))}{S(v)-R( u)} - \frac{\sE_{u}(S(v))-\sE_{v}(S(v))}{S(v)-R( u)}\\
&\geq\inf_{k > G(u)}\frac{\sE_u(k)-D(R(u))}{k-R(u)}-(v- u)-\frac{\Gamma_{u,v}}{S(v )-R(u)}\\
&\geq \phi(u)-(v-u),
\end{align*}
where we use
$R(u)< G(v)$ for the first inequality, \eqref{eq:diffE} and $G(u)< S(v)$ for the second and \eqref{eq:phiS} and $\Gamma_{u,v} \leq 0$ for the third.

Finally, suppose $G(u)<S(v)$ and $R(u)=G(v)$. Then, $R(u)=Q(u)=G(u)=G(v)=S(u)<S(v)$. Since ${L}^{\sE_u,\phi(u)}_{G(u)}$ lies below $\sE_u$ everywhere, while ${L}^{\sE_v,\phi(v)}_{S(v)}$ lies below $\sE_v$ everywhere, using \eqref{eq:EslopeDiff-} together with the fact that $S(v)>G(v)=G(u)$ we have that $\phi(v)\geq\sE'_v(S(v)-)=\sE'_u(S(v)-)-(v-u)\geq \phi(u)-(v-u)$.

(iii) Fix $n\geq1$ and let $u,v\in A_<^n$ with $u<v$. Then $R(v)<G(u)<S(u)$ and we have
\begin{align*}
\phi(u)&=\sup_{k <  G(u)}\frac{\sE_u(S(u))-D(k)}{S(u)-k}\\
&\geq\frac{\sE_u(S(u))-D(R(v))}{S(u)-R(v)}\\
&\geq\inf_{k > G(u)}\frac{\sE_v(k)-D(R(v))}{k-R(v)} = \phi(v),
\end{align*}
where we used that $\sE_u\geq\sE_v$ everywhere for the second inequality and \eqref{eq:phiS2} for the last equality.
\end{proof}

Fix $n\geq1$ and let $u\in A_<^n$. Then
\begin{align}\label{eq:tangentCond}
D(R(u))-R(u)\phi(u)&=\sE_u(S(u))-\phi(u)S(u)\nonumber\\
&=\{P_\nu(S(u))-S(u)(\phi(u)+u)\}-\{P_\mu(G(u))-G(u)u\}.
\end{align}
Recall that $\phi(u)$ is an element of both $\partial \sE_u(R(u))$ and $\partial\sE_u(S(u))$. Moreover, $u\in\partial P_{\mu_u}(G(u))$, but since $P_{\mu_u}(\cdot)$ is linear on $[G(u),+\infty)$, we also have that $u\in\partial P_{\mu_u}(S(u))$. Since the subdifferential of the sum of two functions is equal to the sum of individual subdifferentials (at least provided both subdifferentials are non-empty, as in our case) and since $P_\nu(S(u))=\sE_u(S(u))+P_{\mu_u}(S(u))$, we have that $(\phi(u)+u)\in\partial P_\nu(S(u))$.
\begin{lem}\label{lem:conjugate}
	For each $n\geq1$ and $u,v\in A_<^n$ with $u<v$, we have that
	\begin{enumerate}
		\item $D(R(u))-R(u)\phi(u)-\{D(R(v))-R(v)\phi(v)\}=\int^{v}_u\phi^\prime(w)R(w)dw$,
		\item $P_\mu(G(u))-G(u)u-\{P_\mu(G(v))-G(v)v\}=\int^{v}_uG(w)dw$,
		\item $P_\nu(S(u))-(\phi(u)+u)S(u)-\{P_\nu(S(v))-(\phi(v)+v)S(v)\}=\int^{v}_u(\phi^\prime(w)+1)S(w)dw$.
	\end{enumerate}
\end{lem}
\begin{proof}
	Let $A_<$ be as in \eqref{eq:S<U} and for each $n\geq1$ write $A^n_<:=(u^n_-,u^n_+)$. Recall that $R$ is decreasing on each $A^n_<$. Furthermore, since $R\leq G$ everywhere, $R(u)\leq G(u^n_-+)$ for all $u\in A^n_<$.
	
	Define, for each $n\geq1$, conjugate functions $C^n_D,C^n_{P_\nu},C^n_{P_\mu}:(-1,1)\mapsto\mathbb{R}$ by
	\begin{align*}
	C^n_D(\theta)&=\inf_{-\infty< x\leq G(u^n_-+)}\{D(x)-\theta x\},\\
	C^n_{P_\nu}(\theta)&=\inf_{x\in\R}\{P_\nu(x)-\theta x\},\\
	C^n_{P_\mu}(\theta)&=\inf_{x\in\R}\{P_\mu(x)-\theta x\}.
	\end{align*}
	Note that all three conjugates are concave and thus differentiable almost everywhere.
	
	We prove the first statement, the other two being similar. Fix $u,v\in A_<^n$ with $u<v$. Since $\phi(u)\in\partial \sE_u(R(u))$ and $\phi(v)\in\partial \sE_v(R(v))$, and $\sE_u(R(u))=D(R(u))$ and $\sE_v(R(v))=D(R(v))$,
	\begin{align*}
	D(R(u))-R(u)\phi(u)- \left\{ D(R(v)) - R(v)\phi(v) \right\} &=C^n_D(\phi (u))-C^n_D(\phi (v))\\
	&=-\int^{\phi(v)}_{\phi( u)}(C^n_D)^\prime(\chi) d\chi\\
	&=-\int^{v}_u\phi^\prime(w)(C^n_D)^\prime(\phi(w))dw,
	\end{align*}
	where we used that, by \cref{cor:phiabs}, $\phi(\cdot)$ is differentiable almost everywhere on $A_<^n$.
	
	Moreover, $(C^n_D)^\prime(\phi(w))=-(D^\prime)^{-1}(\phi(w))=-R(w)$ Lebesgue almost everywhere, and hence
	\begin{equation*}
	D(R(u))-R(u)\phi(u)-\{D(R(v))-R(v)\phi(v)\}=\int^{v}_u\phi^\prime(w)R(w)dw,
	\end{equation*}
	as required.
\end{proof}

Using \eqref{eq:tangentCond} and \cref{lem:conjugate} we have that for $u<v$ with $u,v \in A_<^n$,
$$
\int^{v}_u\phi^\prime(w)R(w)dw=\int^{v}_u(\phi^\prime(w)+1)S(w)dw-\int^{v}_uG(w)dw.
$$
Since $n\geq 1$ and $u,v\in A_<^n$ were arbitrary, $\phi^\prime(w)R(w)=(\phi^\prime(w)+1)S(w)-G(w)$ Lebesgue almost everywhere on each $A_<^n$. Hence,
$$
\phi^\prime(w)=-\frac{S(w)-G(w)}{S(w)-R(w)}\quad\textrm{almost everywhere on } A_<.
$$

\section{Left Continuity}\label{sec:LeftCont}
In Sections~\ref{sec:geometric} and \ref{sec:phiprop} we allowed $G$ to be any quantile function. To help simplify the analysis going forward, from now on we assume that $G$ is the left-continuous quantile function associated with $\mu$.
\begin{prop}\label{prop:Rlowersemi}
	Suppose $G$ is left-continuous.
	\begin{enumerate}
		\item[(i)] $S$ is left-continuous.
        \item[(ii)] $\phi$ is 
		left-continuous: $\phi(v-) := \lim_{u \uparrow v} \phi(u) = \phi(v)$. Moreover, for each $v \in (0,1)$ the right limit $\phi(v+)$ exists and
		$\phi(v+) := \lim_{w \downarrow v} \phi(w) = \lim_{w \downarrow v} (\sE^c_v)'(S(w)-)$.
		\item[(iii)] $R$ satisfies $R(v) \leq \liminf_{u \uparrow v} R(u)$. 
	\end{enumerate}
\end{prop}
Before proving \cref{prop:Rlowersemi} we record a couple of lemmas whose proofs are given in \cref{ssec:proofsLC}.

\begin{lem} \label{lem:Elim} For each $v\in(0,1)$,
\begin{equation}\label{E}
\lim_{u\uparrow v}\sE_u(k)=\lim_{u\downarrow v}\sE_u(k)=\sE_v(k)\quad\textrm{and}\quad \lim_{u\uparrow v}\sE^c_u(k) =\lim_{u\downarrow v}\sE^c_u(k)=\sE^c_v(k)\quad\textrm{for all } k\in\R.
\end{equation}
\end{lem}

\begin{lem}\label{lem:hullConv}
	Suppose that $f:(0,1)\mapsto\R$ is non-decreasing.
	Then
	$$
	\lim_{u\uparrow v}\sE^c_u(f(u))=\sE^c_v(f(v-)).
	$$
	If in addition, for each $u\in(0,1)$, $\sE_u$ is non-decreasing on $(f(u),\infty)$, then
	$$\lim_{u\uparrow v}\sE_u(f(u))=\sE_v(f(v-)).
	$$\end{lem}

We are now ready to prove \cref{prop:Rlowersemi}.
\begin{proof}[Proof of \cref{prop:Rlowersemi}(i)]
By definition, $G\leq S$, while by \cref{thm:monotone}, $S$ is non-decreasing. Furthermore, $\sE_u$ is non-decreasing on $(S(u),\infty)$, and therefore \cref{lem:hullConv} applies with $f=S$.
	
	Fix $u,v\in(0,1)$ with $u<v$. Then $G(u)\leq S(u)\leq S(v)$ and therefore
	$$
	G(v)\leq S(v-)\leq S(v).
	$$
	
	We will now show that $S(v)\leq S(v-)$. Since $\sE^c_u(S(u))=\sE_u(S(u))$, letting $u\uparrow v$ on both sides and using \cref{lem:hullConv}, we have that
	$$
	\sE^c_v(S(v-))=\lim_{u\uparrow v}\sE^c_u(S(u))=\lim_{u\uparrow v}\sE_u(S(u))=\sE_v(S(v-)),
	$$
	which implies that
	$$
	S(v-)\in\{k:k\geq G(v), \sE^c_v(k)=\sE_v(k)\}.
	$$
	It follows that
	$$
	S(v-)\geq \inf\{k:k\geq G(v), \sE^c_v(k)=\sE_v(k)\}=S(v).
	$$
\end{proof}

We now turn to the continuity properties of $\phi$.
The proof of \cref{prop:Rlowersemi}(ii) is based on the following lemma.

\begin{lem}
\label{lem:sEc'}
For $u,v\in(0,1)$ with $u<v$, and all $k\in\R$,\\
\noindent{(i)} $(\sE^c_u)'(k -) - (v-u) \leq (\sE^c_v)'(k-) \leq (\sE^c_u)'(k-)$, \\
(ii)
$(\sE^c_u)'(k +) - (v-u) \leq (\sE^c_v)'(k+) \leq (\sE^c_u)'(k+)$.
\end{lem}

\begin{proof}
	
(i) First we show that for $u,v\in(0,1)$ with $u<v$, $(\sE^c_v)'(k-) \leq (\sE^c_u)'(k-)$, for all $k\in\R$.

Fix $k\in\R$ and suppose $\theta \in \partial \sE^c_u(k)$. We show that $(\sE_v^c)'(k-) \leq \theta$ and then since $\theta \in \partial \sE^c_u(k)$ is arbitrary,
\[ (\sE^c_v)'(k-) \leq \inf \{ \theta : \theta \in \partial \sE^c_u(k) \} =  (\sE^c_u)'(k-) \]
as required.

Note that $\theta\in(0,1-u)$. If $\theta \in [1-v, 1-u)$ then we have $(\sE^c_v)'(k-) \leq 1-v \leq \theta$. Hence in what follows we may assume $k$ and $\theta$ are such that $\theta \in \partial \sE^c_u(k)$ and $\theta < 1-v$.

Consider the case $k>R(u)$. Given $\theta \in \partial \sE^c_u(k)$ there exists $j \geq k$ such that $\sE_u(j) = \sE^c_u(j)$ and $\theta \in \partial \sE^c_u(j)$. (If $R(u)<k \leq S(u)$ then $j=S(u)$ is one choice and since $j>R(u)$ we may assume $j \geq S(u)$). Thus
\begin{equation}
\label{eq:ellj} \sE_u(\ell) \geq \sE_u(j) + \theta(\ell - j),  \hspace{10mm}  \mbox{ for all $\ell  \in \R$.} \end{equation}
By the remarks in the previous paragraph we may assume $\theta\in(0,1-v)$. Then we can find $m$ such that both $\sE^c_v(m) = \sE_v(m)$ and $\theta \in \partial \sE^c_v(m)$, and then
\begin{equation}
\label{eq:ellm}
 \sE_v(\ell) \geq \sE_v(m) + \theta(\ell - m),  \hspace{10mm}  \mbox{ for all $\ell \in \R$. }
\end{equation}
Then $\sE_u(m) \geq \sE_u(j) + \theta(m - j)$ and $\sE_v(j) \geq \sE_v(m) + \theta(j - m)$ and adding
\[ \sE_v(j) - \sE_u(j) \geq \sE_v(m) - \sE_u(m). \]
But $\sE_v(\cdot) - \sE_u(\cdot)$ is decreasing, and strictly decreasing on $(G(u+),\infty)$.
Therefore, provided $j>G(u+)$, $m \geq j$.
Then $(\sE^c_v)'(k-)\leq (\sE^c_v)'(j-) \leq (\sE^c_v)'(m-) \leq \theta$ 
as required.

If $S(u) \leq j \leq G(u+)$ we can find $m$ as above with $\sE^c_v(m)= \sE_v(m)$ and $\theta \in \partial \sE^c_v(m)$. If $m\geq j$ then the above argument still works and $(\sE^c_v)'(k-) \leq (\sE^c_u)'(k-)$. So, suppose $m<j$. We begin by arguing that in this case we must have $\sE_v(j) = \sE_v(m) + \theta(j-m)$. If not then by \eqref{eq:ellm} $\sE_v(j) > \sE_v(m) + \theta(j-m)$. But $m<j \leq G(u+)$ and on $(-\infty, G(u+)]$ we have $\sE_u=\sE_v$. Hence $\sE_u(j) > \sE_u(m) + \theta(j-m)$, or equivalently $\sE_u(m) <\sE_u(j) + \theta(m-j)$ contradicting \eqref{eq:ellj}. It follows that we must have $\sE_v(j) = \sE_v(m) + \theta(j-m)$ and hence
\[ \sE_v(\ell) \geq \sE_v(m) + \theta(\ell - m) = \sE_v(j) + \theta(\ell - j), \hspace{10mm}  \mbox{ for all $\ell\in\R$ } \]
and $\sE^c_v(j) = \sE_v(j)$ and $\theta \in \partial \sE^c_v (j)$. In particular we could take $m=j$, and for this choice of $m$ we have
$(\sE^c_v)'(k-) \leq (\sE^c_v)'(j-) \leq \theta$ and since $\theta \in \partial \sE^c_u(k)$ is arbitrary, $(\sE^c_v)'(j-) \leq (\sE^c_u)'(k-)$.

Now suppose $k \leq R(u) \wedge R(v)$. Since $\sE^c_u = \sE^c_v$ on $(-\infty, R(u) \wedge R(v)]$ we have $(\sE^c_v)'(k-) = (\sE^c_u)'(k-)$.

Finally, suppose $R(v)<k \leq R(u) \leq S(v)$. If $\psi = (\sE^c_v)'(k-)$ then $\psi \in \partial \sE^c_v(R(v))$ and since $\sE^c_v(R(v))= \sE^c_u(R(v))$ it follows that $\psi \in \partial \sE^c_u(R(v))$. Then, since $k> R(v)$, if $\theta \in \partial \sE^c_u(k)$ we must have $\theta \geq \psi$. Then $(\sE^c_u)'(k-) \geq (\sE^c_v)'(k-)$.

The proof that $(\sE^c_u)'(k-) - (v-u) \leq (\sE^c_v)'(k-)$ is similar, but based on the fact that $\sE_u(k) - \sE_v(k) - k(v-u)$ is decreasing in $k$.

(ii) Given that $(\sE^c_w)'(k+) = \lim_{j \downarrow k} (\sE^c_w)'(j-)$ the result follows from (i) by taking limits from above.
\end{proof}

\begin{proof}[Proof of Proposition~\ref{prop:Rlowersemi}(ii)]
Fix $v\in(0,1)$.
Consider the left limit $\phi(v-) = \lim_{u \uparrow v} \phi(u)$.
 We have, using Lemma~\ref{lem:sEc'} and the last representation of Lemma~\ref{lem:phialt},
\begin{eqnarray*}
\phi(v) = (\sE^c_v)'(S(v)-) & = & \lim_{u \uparrow v} (\sE^c_v)'(S(u)-) \\
& \leq &  \lim_{u \uparrow v} (\sE^c_u)'(S(u)-) =   \lim_{u \uparrow v} \phi(u) = \phi(v-) \\
& \leq & \lim_{u \uparrow v} (\sE^c_u)'(S(v)-) \\
& \leq & \lim_{u \uparrow v} \left[ (\sE^c_v)'(S(v)-) + (v-u) \right] \\
&=& (\sE^c_v)'(S(v)-) = \phi(v).
\end{eqnarray*}

The proof of the result for the right limit is similar. 
We have
\begin{eqnarray*}
	\lim_{v \downarrow u} (\sE^c_u)'(S(v)-)
	& \leq &  [\lim_{v \downarrow u} (\sE^c_v)'(S(v)-) + (v-u)] \\
	& = &  \lim_{v \downarrow u} (\sE^c_v)'(S(v)-) = \lim_{v \downarrow u} \phi(v) = \phi(u+) \\
	& \leq & \lim_{v \downarrow u} (\sE^c_u)'(S(v)-).
\end{eqnarray*}
\end{proof}

\begin{proof}[Proof of \cref{prop:Rlowersemi}(iii)]
We are left to prove the lower semi-continuity (from the left) of $R$. Fix $v\in(0,1)$. Since $R\leq G$ everywhere, $\liminf_{u\uparrow v}R(u)\leq G(v)$. Moreover, for $u\in(0,1)$ with $u<v$, by the left-monotonicity we have that $R(v)\notin(R(u),S(u))$, and therefore
$$
R(v)\notin(\inf_{u<v}R(u),\sup_{u<v}S(u)=S(v)).
$$
Then, if $R(v)<G(v)\leq S(v)$, $R(v)\leq\inf_{u<v}R(u)\leq\liminf_{u\uparrow v}R(u)$, and we are done. So, suppose that $R(v)=G(v)=S(v)$. Let $L={L}^{\sE_v,\phi(v)}_{G(v)}$. Then $D>L$ to the left of $G(v)$ and $\{k:k < G(v),\textrm{ } D(k) = L(k)\}=\emptyset$. Furthermore, if $G(w)=G(v)$ for some $w\in(0,v)$, then, for all $u\in(w,v)$, $G(w)=G(u)=S(u)=G(v)=S(v)$, and $\sE^c_u=\sE_u=\sE_v=\sE^c_v$ at $G(v)$. Then, by \cref{lem:chagree}, $\sE^c_u=\sE^c_v$ on $(-\infty, G(v)]$. Then $R(u)=G(u)=S(u)=G(v)=R(v)$ and it follows that $\liminf_{u\uparrow v}R(u)=R(v)$.

Hence suppose that $R(v)=G(v)=S(v)$ and $G(u)<G(v)$ for all $u\in(0,v)$. There are two cases to consider.

First, suppose that $D=\sE_v>\sE^c_v$ on an interval to left of $G(v)$, and such that this interval cannot be made larger without violating $\sE_v>\sE^c_v$. Then $\sE^c_v$ is linear on this interval, and this interval must be finite since $\sE_v\geq\sE^c_v\geq0$ and $\sE^c_v$ can only change slope at points where $\sE_v = \sE_v^c$. But, if the interval is finite, $\{k:k \leq G(u), D(k) = {L}^{\sE_u,\phi(u)}_{G(u)}(k) \}\neq\emptyset$, a contradiction.

Second, suppose that, for each $u\in(0,v)$ there exists $k_u\in(G(u),G(v))$ such that $D(k_u)=\sE^c_v(k_u)$. We first argue that, for all $u\in(0,v)$, $\sE^c_v$ is not linear on $(G(u),G(v))$. Suppose there exists $u_0\in(0,v)$ such that $\sE^c_v$ is linear on $(G(u_0),G(v))$. Let $(k_0,G(v)]\supseteq (G(u_0),G(v)]$ be the largest interval of the form $(k, G(v)]$ on which $\sE^c_v$ is linear. As in the first case we must have that $k_0$ is finite, but then $R(v)=k_0<G(v)$, a contradiction.
Hence $\sE^c_v$ is not linear to the left of $G(v)$. Now, let $\bar{k}_u:=\sup\{k\geq k_u: \sE^c_v(k)=L_{k_u}^{\sE^c_v,(\sE^c_v)^\prime(k_u-)}(k)\}$. Then $D\geq\sE^c_v>L_{k_u}^{\sE^c_v,(\sE^c_v)^\prime(k_u-)}$ on $(\bar{k}_u,G(v))$ and $\bar{k}_u\in(G(u),\bar{y}_u)$, where $\bar{y}_u\in(G(u),G(v))$ is the point where $L_{k_u}^{\sE^c_v,(\sE^c_v)^\prime(k_u-)}$ crosses $L$. (Note that neither $\bar{k}_u=\bar{y}_u<G(v)$ nor $\bar{k}_u<\bar{y}_u=G(v)$ can happen, since then either $\sE^c_v=L$ on $(\bar{y}_u,G(v))$ or $\sE^c_v=L_{k_u}^{\sE^c_v,(\sE^c_v)^\prime(k_u-)}$ on $(k_u,G(v))$, respectively. But this contradicts the fact that $\sE^c_v$ is not linear to the left of $G(v)$.) Then, $G(u)<\bar{k}_u\leq R(w)\leq R(v)=G(v)$ for all $w\in(F_\mu(\bar{k}_u),v)$, and therefore $G(u)\leq\liminf_{z\uparrow v}R(z)\leq R(v)=G(v)$. Using the left-continuity of $G$ we conclude that $\liminf_{u\uparrow v}R(u)=R(v)$.
\end{proof}

In the next section we will need further two results.
\begin{cor}\label{cor:phisemi}
If $S(w)>S(v+)$ for all $w>v$ then $\phi(v+) = (\sE^c_v)'(S(v+)+)$. Otherwise, if $S(w)=S(v+)$ for some $w>v$ then $\phi(v+) = (\sE^c_v)'(S(v+)-)$. In either case $\phi(v+) \geq (\sE^c_v)'(S(v+)-)$.
\end{cor}
\begin{proof}
	If $S(w)>S(v+)$ for all $w>v$ then $\lim_{w \downarrow v}(\sE^c_v)'(S(w)-) = (\sE^c_v)'(S(v+)+)$ where we use the fact that, for a convex function $f$, $\lim_{y\downarrow x}f'(y-)=f'(x+)$. Otherwise, if $S(w)=S(v+)$ for some $w>v$ then $\lim_{w \downarrow v}(\sE^c_v)'(S(w)-) = (\sE^c_v)'(S(v+)-)$.
\end{proof}

Let $S^{-1}$ be the right-continuous inverse to the increasing function $S$.
By our conventions, for $y\in\R$, $S(S^{-1}(y)) \leq y$ with equality whenever $S$ is continuous at $S^{-1}(y)$.
Moreover, if $w>S^{-1}(y)$ then $S(w)>y$.

\begin{cor}
	\label{cor:phi*} For $y\in\R$,
	$\phi(S^{-1}(y)) \leq (\sE^c_{S^{-1}(y)})'(y-) \leq (\sE^c_{S^{-1}(y)})'(y+) \leq \phi(S^{-1}(y)+)$.
\end{cor}

\begin{proof}
	Since $S(S^{-1}(y)) \leq y  \leq S(S^{-1}(y)+)$,
	\[
	\phi(S^{-1}(y))=(\sE^c_{S^{-1}(y)})'(S(S^{-1}(y))-)
	\leq  (\sE^c_{S^{-1}(y)})'(y-)
	\leq  (\sE^c_{S^{-1}(y)})'(y+),
	\]
	where we use Lemma~\ref{lem:phialt} for the equality.
	
	Moreover, if $y = S(S^{-1}(y+))$ then $(\sE^c_{S^{-1}(y)})'(y+) = (\sE^c_{S^{-1}(y)})'(S(S^{-1}(y)+)+)$ and $S(w) > y = S(S^{-1}(y)+)$ for all $w > S^{-1}(y)$, so that by Corollary~\ref{cor:phisemi}, $\phi(S^{-1}(y)+) = (\sE^c_{S^{-1}(y)})'(S(S^{-1}(y)+)+)$ and $(\sE^c_{S^{-1}(y)})'(y+) = \phi(S^{-1}(y)+)$. Otherwise, if $y < S(S^{-1}(y)+)$ then, by Corollary~\ref{cor:phisemi}, $(\sE^c_{S^{-1}(y)})'(y+) \leq (\sE^c_{S^{-1}(y)})'(S(S^{-1}(y)+)-) \leq \phi(S^{-1}(y)+)$.
\end{proof}

\section{The candidate coupling is an embedding}
\label{sec:embed}

We are now almost ready to prove \cref{thm:construction}.
Let $U,V\sim U(0,1)$ be two independent uniform random variables. Then $X=G(U)\sim\mu$. On the other hand, $Y(U,V)$, if defined as in \eqref{eq:YUVdef}, may not be a random variable since $R$ may not be measurable (see Remark \ref{rem:measurability}). In order to deal with this, we introduce $T:(0,1)\to\R$ given by
\begin{equation}\label{eq:T}
T(u)=\begin{cases}
G(u),\quad \textrm{if }S(u)=G(u)\\
R(u),\quad \textrm{if }S(u)>G(u)
\end{cases}\quad u\in(0,1).
\end{equation}
The proof of the following lemma is postponed until Appendix \ref{ssec:proofsE}.
 \begin{lem}\label{lem:monotone} Let $T,S:(0,1)\mapsto\R$ be defined by \eqref{eq:T} and \eqref{eq:S}. Then
	$T$ is (Borel) measurable and the pair $(T,S)$ is left-monotone with respect to $G$ on $(0,1)$ in the sense of Definition~\ref{def:lmonfns}.
\end{lem}
Define $Y:(0,1)^2\to\R$ by $Y(u,v) = G(u)$ on $\{(u,v)\in(0,1)^2:T(u)=S(u)\}$ and
\begin{equation}
Y(u,v) = T(u) I_{\left\{ v \leq \frac{S(u) - G(u)}{S(u)-T(u)} \right\}} +  S(u) I_{ \left\{ v > \frac{S(u) - G(u)}{S(u)-T(u)} \right\} }
\label{eq:YUVdef1}
\end{equation}
otherwise. Then $Y(U,V)$ is a random variable and the martingale property (see the statement of Theorem \ref{thm:construction}) is a direct consequence of the definition of $Y$, see \eqref{eq:YUVdef1}. We are left to show that $Y(U,V)$ has law $\nu$.

Note that with our choice of left-continuous quantile function the definition of $A_<$ becomes $A_< = \{ u \in (0,1): G(u+) < S(u) \} = \cup_{n \geq 1} A^n_<$. Also recall that $S^{-1}$ denotes the right-continuous inverse to the increasing function $S$.

Let $\chi = \sL(Y(U,V))$. We want to show that $\chi = \nu$. We begin by describing the strategy of our proof. Fix $y\in \R$.
Then, since $\Prob(Y(u,V) \leq y) = \frac{S(u)-G(u)}{S(u)-{T}(u)}$ whenever ${T}(u)\leq y < S(u)$,
\begin{eqnarray}
\chi((-\infty,y]) = \Prob[Y(U,V)\leq y]&=&\Prob[U\leq S^{-1}(y)]+\Prob[Y(U,V)\leq y, U>S^{-1}(y)] \nonumber \\
&=&S^{-1}(y)+\int^1_{S^{-1}(y)}\frac{S(u)-G(u)}{S(u)-T(u)}I_{\{{T}(u)\leq y < S(u)\}}du\nonumber\\
 &=&S^{-1}(y)+\int^1_{S^{-1}(y)}\frac{S(u)-G(u)}{S(u)-{T}(u)}I_{\{{T}(u)\leq y \}}du,\label{eq:intrepresentation}
\end{eqnarray}
where for the last equality we used that, since $S^{-1}$ is right-continuous, $S(u)>y$ for all $u>S^{-1}(y)$.

We are free to include a multiplicative term $I_{ \{ S(u) > G(u) \} }$ in the integrand since off this set $\frac{S(u)-G(u)}{S(u)-{T}(u)}=0$. (Recall that $T=R$ on $\{ S(u) > G(u) \}$.) But $\{ S(u) > G(u) \}$ and $A_{<}$ differ only by a countable (and thus null) set and on $A_<$ we have $\frac{S(u)-G(u)}{S(u)-R(u)} = -\phi'(u)$. Hence, we are interested in integrals of the form $-\int^1_a \phi'(u) I_{\{R(u)\leq y \}}du$. If $I=(u_-,u_+)$ is an interval over which $R(\cdot) > y$ (and $I$ cannot be made any larger without violating $R(\cdot)>y$) then we expect that $\phi(u_-)=\phi(u_+)$. Adding in such intervals would allow us to replace $-\int^1_a \phi'(u) I_{\{R(u)\leq y\}}du$ with $-\int^1_a \phi'(u) du = \phi(a) - \phi(1)=\phi(a)$ since $\phi(1)=0$. Then
\begin{equation}
\label{eq:Proby}
\Prob[Y(U,V)\leq y] = S^{-1}(y) + \phi(S^{-1}(y)).
\end{equation}
Further, for each $v\in(0,1)$, since $\phi(v) \in \partial \sE_v(S(v))$ and $v \in \partial P_{\mu_v} (S(v))$ we have $(v+\phi(v)) \in \partial P_\nu (S(v))$, and then, if $S(v)$ is a continuity point of $\nu$, $v+\phi(v)=P^\prime_\nu(S(v))=\nu((-\infty,S(v)])$. Hence, provided $y$ is a continuity point of $\nu$ and $S(S^{-1}(y))=y$, we have
\[ \chi((-\infty,y]) = \Prob[Y(U,V)\leq y] =S^{-1}(y)+\phi(S^{-1}(y)) = P_\nu^\prime(S(S^{-1}(y)))=P_\nu^\prime(y)= \nu((-\infty,y]) \]
as desired.

There are at least three issues we must overcome to complete this analysis. First, the derivative $\phi'$ need not exist everywhere. Second, there may be a countably infinite number of intervals $A_<^m=(u^m_-,u^m_+)$ which we must add, on each of which $R(\cdot) > y$.
Third, we need a refined argument to cover the case where $S(S^{-1}(y))<y$.

To deal with the issues about $\phi$ we introduce a family of modified functions $\psi_{v,x}$, each member of which is monotonically decreasing and Lipschitz, and therefore has a derivative almost everywhere. The introduction of the monotonic function $\psi$ also allows us to easily add the missing intervals since $\psi$ is constant on those intervals by construction; moreover the intervals where $\psi$ is decreasing are precisely the intervals where $R(\cdot)\leq y$. The case where $S(S^{-1}(y))<y$ requires a careful definition of the initial value of $\psi$ and an additional argument.

Note that it is sufficient to prove $\chi((-\infty,y])=\nu((-\infty,y])$ at continuity points of $\nu$ only. Indeed it is sufficient to prove $\chi((-\infty,y])=\nu((-\infty,y])$ on a dense set of values of $y$, so we may also restrict attention to $y$ which are continuity points of $\mu$ also.



\begin{lem}
\label{lem:claims3phi}Suppose $y$ is a continuity point of $\nu$.
Then $\phi(S^{-1}(y)) \leq P_{\nu}'(y)- S^{-1}(y) \leq \phi(S^{-1}(y)+)$. 
\end{lem}

\begin{proof}
We have $\sE^c_{S^{-1}(y)} = \sE_{S^{-1}(y)}$ on $[S(S^{-1}(y)),\infty)$ and $y \geq S(S^{-1}(y))$ so that $\sE^c_{S^{-1}(y)}(y) = \sE_{S^{-1}(y)}(y)$ and
$(\sE^c_{S^{-1}(y)})'(y+) = \sE_{S^{-1}(y)}'(y+)$. Also, since $y$ is a continuity point of $\nu$, $\sE'_{S^{-1}(y)}(y+) = P_\nu'(y+) - S^{-1}(y) = P_\nu'(y) - S^{-1}(y)$. Then the result follows by Corollary~\ref{cor:phi*}.
\end{proof}
\begin{defn}\label{def:psi}
	For $v\in(0,1)$ and $x\in[0,1)$, define $\psi_{v,x} : [v,1] \mapsto [0,1]$ by
	$\psi_{v,x}(w) = x \wedge \inf_{v < u \leq w} \{ \phi(u) \}$.
\end{defn}

Here we use the fact that we have extended the domain of $\phi$ to $(0,1]$ whence also $\psi_{v,x}(1)=0$. Note that the only relevant cases are when $x \in [\phi(v), \phi(v+)]$. See Figure~\ref{fig:PhiPsi}.

The proof of the following lemma is quite straightforward but is deferred to \cref{ssec:proofsE}.

\begin{lem}\label{lem:psi}
	For $v\in(0,1)$ and $x \in [\phi(v), \phi(v+)]$, $\psi_{v,x}$ is decreasing and absolutely continuous.
\end{lem}

It follows from a combination of \cref{lem:claims3phi} and \cref{lem:psi} that $\psi_{S^{-1}(y),P_{\nu}'(y) - S^{-1}(y)}$ is decreasing and absolutely continuous on $(S^{-1}(y),1)$, and hence $\psi'_{S^{-1}(y),P_{\nu}'(y) - S^{-1}(y)}$ is defined almost everywhere on $(S^{-1}(y),1)$.

	\begin{figure}[H]
		\centering
		\begin{tikzpicture}
		\draw (0,-0.2) -- (0,6.2);
		\draw (10,-0.2)--(10,6.2);
		\draw[loosely dashed, gray] (-0.2,0)--(10,0);
		\draw[loosely dashed, gray] (-0.2,6)--(10,6);
		\node at (-0.5,0) {0};
		\node at (-0.5,6) {1};
		\node at (0,-0.5) {0};
		\node at (10,-0.5) {1};
		\node at (2,-0.5) {$v$};
		\draw [blue,densely dotted] (0,0) to[out=0, in=240] (2,2);
		\draw[fill=blue] (2,2) circle (2pt);
		\draw[blue] (2,3) circle (2pt);
		\draw[loosely dashed, gray] (-0.2,2)--(2,2);
		\node at (-0.7,2) {$\phi(v)$};
		\draw[loosely dashed, gray] (-0.2,3)--(2,3);
		\node at (-0.8,3) {$\phi(v+)$};
		\draw[loosely dashed, gray] (-0.2,2.5)--(2,2.5);
		\node at (-0.5,2.5) {$x$};
		\draw[densely dotted, blue] (2,3) to[out=0, in=180] (3,4);
		\draw[densely dotted, blue] (3,4) to[out=0, in=160] (4,2.5);
		\draw[fill=red] (2,2.5) circle (2pt);
		\draw[red] (2,2.5)--(4,2.5);
		\draw[red] (4,2.5) to[out=340, in=180] (5,1.5);
		\draw[blue, densely dotted] (5,1.5) to[out=0, in=180] (6,3) to[out=0, in=150] (7,1.5);
		\draw[red] (5,1.5)--(7,1.5);
		\draw[red] (7,1.5) to[out=330, in=180] (7.5,1);
		\draw[densely dotted,blue] (7.5,1) to[out=0, in=220] (8,1.3);
		\draw[fill=blue] (8,1.3) circle (2pt);
		\draw[blue,densely dotted] (8,2) to[out=40, in=150] (9,1);
		\draw[blue] (8,2) circle (2pt);
		\draw[red] (7.5,1)-- (9,1) to[out=330, in=180] (10,0);
		\node[blue] at (5,4) {$w\mapsto\phi(w)$};
		\node[red] at (5,0.7) {$w\mapsto\psi_{v,x}(w)$};
		\end{tikzpicture}
		\caption{Construction of $\psi$ (solid curve) from $\phi$ (dotted curve). $\phi$ can have jumps, but only upwards, and for $x \in [\phi(v),\phi(v+)]$, $w\mapsto\psi_{v,x}(w) = x \wedge \inf_{v < u \leq w} \{ \phi(u) \}$ is monotonically decreasing and Lipschitz with Lipschitz constant 1.}
		\label{fig:PhiPsi}
	\end{figure}
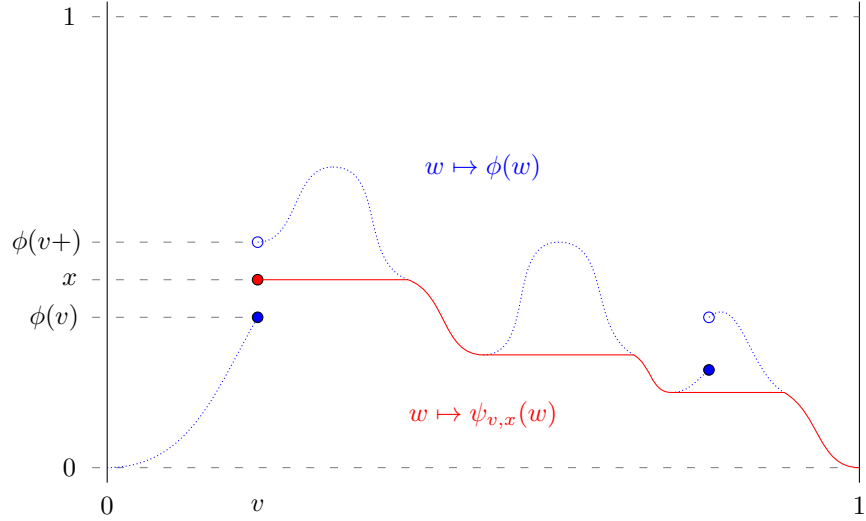

The next lemma says that the places $\{u : u \in (S^{-1}(y),1) \}$ where $\psi_{S^{-1}(y),P_{\nu}'(y) - S^{-1}(y)}$ decreases are essentially the places where $\psi_{S^{-1}(y),P_{\nu}'(y) - S^{-1}(y)}(u)=\phi(u)$ and $S(u) > G(u+)$.
\begin{lem}
\label{lem:claims12}
For $u \in (0,1)$ and $x \in [\phi(u),\phi(u+)]$ the following set inclusion holds:
\begin{equation}
\{ w : w>u, \psi_{u,x}'(w) < 0  \} \subseteq
\left( \{w: w>u, S(w) > G(w+) \} \cap \{w:w>u, \psi_{u,x}(w) = \phi(w) \} \right) \cup  \sN_u  \label{eq:claim2b}
\end{equation}
where $\sN_u$ is a set of measure zero.
\end{lem}

\begin{proof}
Let $\sN_u = \{ w \in [u,1) \mbox{ such that $S(w+) > S(w)$ or $G(w+)>G(w)$ or $\psi'_{u,x}(w)$ does not exist} \}$.

To prove the lemma we need to show that if $w \in \sN_u^c$ and $S(w)=G(w)$ then $\psi_{u,x}'(w)=0$ and also that if  $w \in \sN_u^c$ and $\phi(w)>\psi_{u,x}(w)$ then $\psi_{u,x}'(w)=0$. The second of these is immediate: if $\phi(w) - \psi_{u,x}(w)> \epsilon >0$, then for $0< \delta \leq \epsilon$ we have $\phi(w+\delta) \geq \phi(w) - \delta > \psi_{u,x}(w)$, and hence $\psi_{u,x}(w+\epsilon) = \psi_{u,x}(w)$. Hence $\psi_{u,x}'(w+)=0$ and if $\psi'_{u,x}(w)$ is defined --- as it must be since $w \in \sN^c_u$ --- it must take the value 0.

Hence, to complete the proof of the lemma it is sufficient to show that if $w$ is such that $S$ and $G$ are continuous at $w$ and $\psi_{u,x}'(w)$ exists, then $S(w)=G(w)$ implies $\psi_{u,x}'(w)=0$.

We split into two cases: $R(w)<G(w)=S(w)$ and $R(w)=G(w)=S(w)$.

{\em Case 1: $R(w)<G(w)=S(w)$.}
Using \eqref{eq:phiR} and \eqref{eq:phiS} for the inequalities we have that
\begin{eqnarray*}
\phi(w+h) & \geq & \frac{\sE_{w+h}( S(w+h)) - D(R(w))}{S(w+h) - R(w)} \\
& = &  \frac{\sE_{w}( S(w+h)) - D(R(w))}{S(w+h) - R(w)} + \frac{\sE_{w+h} (S(w+h)) - \sE_w(S(w+h))}{S(w+h) - R(w)} \\
& \geq & \phi(w) + \frac{\sE_{w+h} (S(w+h)) - \sE_w(S(w+h))}{S(w+h) - R(w)}.
\end{eqnarray*}
It is sufficient to show that $\liminf_{h \downarrow 0} \frac{\phi(w+h) - \phi(w)}{h} \geq 0$.
This will follow if
\[ \lim_{h \downarrow 0} \frac{1}{h}\left( \sE_{w}( S(w+h)) - \sE_{w+h}(S(w+h)) \right)   \leq 0. \]
But, by \eqref{eq:Ediff}, $\sE_{w} (S(w+h)) - \sE_{w+h}(S(w+h)) = h(S(w+h) - G(w)) + P_\mu(G(w+h)) - P_\mu(G(w)) - (w+h) (G(w+h)-G(w))$.
 Moreover, $P_\mu(G(w+h)) - P_\mu(G(w)) = \tilde{w}(G(w+h)-G(w))$ for some $\tilde{w} \in [w,w+h]$ and then
\( |P_\mu(G(w+h)) - P_\mu(G(w)) - (w+h) (G(w+h)-G(w))| \leq h (G(w+h) - G({w})) .\)
We conclude $|\frac{1}{h}\left( \sE_{w}( S(w+h)) - \sE_{w+h}(S(w+h)) \right)| \leq S(w+h) - G(w) + G(w+h) - G(w) \rightarrow 0$,
by the continuity of $S$ and $G$ at $w$ and the fact that $S(w)=G(w)$.

{\em Case 2: $R(w)=G(w)=S(w)$.}
Then $\phi(w) = D'(G(w)-)$ and $L^{\sE_w,\phi(w)}_{G(w)} \leq \sE^c_w \leq \sE_w$. Note that $D=\sE_w$ to the left of $G(w)$.

Pick $\theta \in (0,\phi(w))$. Define $z_\theta = \sup \{ z : \theta \in \partial \sE^c_w(z) \}$. Then $z_\theta<G(w)$. Note that $ \theta \in \partial \sE^c_w(z_\theta)$ and $D(z_\theta) = \sE_w^c(z_\theta)$. If $G(u+)=G(w)$ then $\phi(v)=\phi(w)$ for all $v\in(u,w]$, and thus $\psi_{u,x}=x$ on $[u,w]$. It follows that $\psi_{u,x}^\prime(w)=\psi_{u,x}^\prime(w-)=0$. Otherwise $G(u+)<G(w)$, and by choosing $\theta$ large enough we may assume $z_\theta>G(u+)$.

Let $u_\theta=F_\mu(z_\theta)$. Then $u<u_\theta<w$, and either $G(u_\theta)=z_\theta$ or $G(u_\theta)< z_\theta< G(u_\theta+)$. We claim that $\phi(u_\theta)\leq\theta$. If $G(u_\theta)=z_\theta$ then by \cref{lem:chagree} we have that $\sE^c_{u_\theta}=\sE^c_{w}$ on $(-\infty,G(u_\theta)=z_\theta]$, and therefore $\phi(u_\theta)=\inf\{z:z\in\partial\sE^c_{w}(z_\theta)\}\leq\theta$. Suppose $G(u_\theta)< z_\theta< G(u_\theta+)$. If $\sE_{u_\theta}(G(u_\theta))=D(G(u_\theta))=\sE^c_w(G(u_\theta))$ then using the same argument (i.e., by \cref{lem:chagree}) we have that $\phi(u_\theta)=\inf\{z:z\in\partial\sE^c_{w}(G(u_\theta))\}$. But, by the monotonicity of the subdifferential of $\sE_w$, $\inf\{z:z\in\partial\sE^c_{w}(G(u_\theta))\} \leq\inf\{z:z\in\partial\sE^c_{w}(z_\theta)\}\leq\theta$. Hence $\phi(u_\theta) \leq \theta$. Finally, if $\sE_{u_\theta}(G(u_\theta))=D(G(u_\theta))>\sE^c_w(G(u_\theta))$, then since $\sE_{u_\theta}$ is convex and coincides with $D$ on $[G(u_\theta),G(u_\theta+)]$, it follows that $S(u_\theta)\leq z_\theta$ and $\sE_{u_\theta}(S(u_\theta))=D(S(u_\theta))=\sE^c_w(S(u_\theta))$. Then $\sE^c_{u_\theta}=\sE^c_{w}$ on $(-\infty,S(u_\theta)]$ and, by \cref{lem:phialt}, we again have that $\phi(u_\theta) = (\sE_u^c)'(S(u_\theta)-) = (\sE_w^c)'( S(u_\theta)-) \leq  (\sE_w^c)'(z_\theta-)\leq\theta$.
	
Therefore in all cases $\psi_{u,x}(w) \leq\phi(u_\theta)\leq \theta < \phi(w)$. By \cref{lem:phi}(ii) we than have that $\phi(v)>\psi_{u,x}(w)$ for $v \in (w, w+ \phi(w)-\psi_{u,x}(w))$ and thus $\psi$ is constant on this interval.

It follows that if $\psi_{u,x}'(w)$ exists, then $\psi_{u,x}'(w)=0$.
\end{proof}

The final result we need identifies the set where $R(v) \leq y$ with the set where $\psi = \phi$. The proof of this result has to cover several cases and consequently is quite long and technical. For this reason the proof is postponed to Appendix~\ref{ssec:setequality}.

\begin{prop}
\label{prop:setequality}
Suppose $y$ is a continuity point of $\mu$ and $\nu$.
Then
\begin{equation} \{ v : v>S^{-1}(y), R(v) \leq y  \} = \{ v : v> S^{-1}(y), \psi_{S^{-1}(y), P_\nu(y)-S^{-1}(y)}(v) = \phi(v) \}. \label{eq:claim3D}
\end{equation}
\end{prop}

We are now ready to prove the main results.

\begin{prop}
\label{prop:intrewrite}
Suppose $y$ is a continuity point of both $\mu$ and $\nu$. 
Then
\begin{equation} \label{eq:int}
\int^1_{S^{-1}(y)}\frac{S(v)-G(v)}{S(v)-{T}(v)}I_{\{T(v)\leq y  \}}dv =  P'_\nu(y) - S^{-1}(y) .
\end{equation}
\end{prop}


\begin{proof}

From \cref{prop:setequality} we know that under the hypotheses of the proposition
\begin{equation}
\{u: u > S^{-1}(y), R(u) \leq y \}  =  \{u:  u > S^{-1}(y), \psi_{<y>}(u) = \phi(u) \} \label{eq:claim3}.
\end{equation}
where $\psi_{<y>}$ is shorthand for $\psi_{S^{-1}(y), P_\nu'(y) - S^{-1}(y)}$.

Then, using the fact that $A_< \subseteq \{v: G(v)< S(v) \}$ but the difference is a set of measure zero,
\begin{eqnarray*}
\int^1_{S^{-1}(y)}\frac{S(v)-G(v)}{S(v)-{T}(v)}I_{\{{T}(v)\leq y\}}dv
& = & \int^1_{S^{-1}(y)}\frac{S(v)-G(v)}{S(v)-R(v)}I_{\{R(v)\leq y \}} I_{ \{ G(v)< S(v) \} }dv \\
& = & \int^1_{S^{-1}(y)}\frac{S(v)-G(v)}{S(v)-R(v)}I_{\{R(v)\leq y \}} I_{ A_< } dv \\
& = & - \int^1_{S^{-1}(y)}\phi'(v) I_{\{R(v)\leq y \}} I_{ A_< } dv \\
& = & - \int^1_{S^{-1}(y)}\phi'(v) I_{\{ \psi_{<y>}(v) = \phi(v) \}} I_{ A_< } dv \\
& = & - \int^1_{S^{-1}(y)}\psi_{<y>}'(v) I_{\{ \psi_{<y>}(v) = \phi(v) \}} I_{A_<} dv \\
& = & - \int^1_{S^{-1}(y)}\psi_{<y>}'(v) \left[ I_{\{ \psi_{<y>}(v) = \phi(v) \}} I_{ A_< } + I_{ \sN } \right] dv,
\end{eqnarray*}
where $\sN$ is some set of measure zero.

For any absolutely continuous, decreasing function $g$ we have $\int_a^b g'(z) dz = \int_a^b g'(z) I_{\{ g'(z) < 0 \} } dz + \int_a^b g'(z) I_{\{ g'(z) = 0 \} } dz = \int_a^b g'(z) I_{A} dz $ for any set $A \supseteq \{ g'(z) < 0 \}$.
We saw in Lemma~\ref{lem:claims12} that 
\begin{equation}
\label{eq:claim2}
\{ u : u > S^{-1}(y), \psi_{<y>}'(u) < 0  \} \subseteq
\left( \{u: u > S^{-1}(y), S(u) > G(u+) \} \cap \{ \psi_{<y>}(u) = \phi(u) \} \right)  \cup \sN_{S^{-1}(y)}
\end{equation}
for a well chosen set $\sN_{S^{-1}(y)}$ of measure zero.
Then, 
\begin{eqnarray*}
\lefteqn{\int^1_{S^{-1}(y)}\psi_{<y>}'(v) \left( I_{\{ \psi_{<y>}(v) = \phi(v) \}} I_{ \{ G(v+)< S(v) \} } +  I_{\sN_{S^{-1}(y)}} \right) dv} \\
& = & \int_{S^{-1}(y)}^1 \psi_{<y>}'(v) dv
 =  \psi_{<y>}(1) - \psi_{<y>}(S^{-1}(y))
= - (P'_{\nu}(y) - S^{-1}(y))
\end{eqnarray*}
with the last equality following from Lemma~\ref{lem:phi} and the fact that $\psi_{<y>}(1) = \phi(1) = 0$.
\end{proof}

\begin{prop}
\label{prop:resultempty}
Suppose that $y$ is a continuity point of both $\mu$ and $\nu$. 
Then $\chi((-\infty,y]) = \nu((-\infty,y])$.
\end{prop}

\begin{proof}
Combining \eqref{eq:intrepresentation} and \eqref{eq:int} we find $\Prob(Y(U,V) \leq y) = P'_{\nu}(y)$.
\end{proof}

\begin{thm}\label{thm:construction2}
Define $T$ and $S$ as in \eqref{eq:T} and \eqref{eq:S}.
Then they are left-monotone with respect to $G$ and define a construction as in Theorem~\ref{thm:construction} such that $\sL(Y)=\nu$, where $Y$ is as in \ref{eq:YUVdef1}. In particular, $T$ and $S$ define the (lifted) left-curtain coupling.
\end{thm}

\begin{proof}
This follows immediately from Lemma~\ref{lem:monotone} and Proposition~\ref{prop:resultempty}, where we note that it is sufficient to show that $\chi((-\infty,y]) = \nu((-\infty,y])$ on a dense subset of the reals and we may exclude non-continuity points of $\mu$ and $\nu$.
\end{proof}

\bibliographystyle{plainnat}

\appendix
\section{Proofs}\label{sec:proofs}
\subsection{Convex hull}
\label{ssec:proofsCH}
\begin{proof}[Proof of \cref{lem:linear}]
Without loss of generality we assume that $f^c$ is not equal to $-\infty$ everywhere. Since $f>f^c$, we also cannot have that $f^c=\infty$ everywhere.
		
		 We first deal with the case when $a,b\in\R$ with $a<b$. Suppose $f^c$ is not a straight line on $(a,b)$. Then, by the convexity of $f^c$, for all $x\in(a,b)$ we have
	\begin{equation}\label{eq:f<L}
	f^c(x) <\frac{b-x}{b-a} f^c(a) + \frac{x-a}{b-a} f^c(b) = L^{f^c}_{a,b}(x) .
	\end{equation}
	
	Let $\eta = \inf_{y \in (a,b)} \{ f(y) -  L^{f^c}_{a,b}(y) \}$. If $\eta \geq 0$ then $f \geq L^{f^c}_{a,b}$ on $(a,b)$ and $f \geq f^c \vee L^{f^c}_{a,b}$, contradicting the maximality of $f^c$ as a convex minorant of $f$.
	
	Now suppose that $\eta<0$. Since $f$ is lower semi-continuous, $f-L^{f^c}_{a,b}$ is also lower semi-continuous, and therefore attains its infimum on $[a,b]$. Fix $z\in\arginf_{y\in[a,b]}\{f(y) -  L^{f^c}_{a,b}(y)\}$. Since $f(k)-L^{f^c}_{a,b}(k)=f(k)-f^c(k)\geq0$ for $k\in\{a,b\}$, $a,b\notin\arginf_{y\in[a,b]}\{f(y) -  L^{f^c}_{a,b}(y)\}$, and thus $z\in(a,b)$. Then since $f > f^{c}$ on $(a,b)$ we have $0 > \eta = f(z)-L^{f^c}_{a,b}(z)> f^c(z) -  L^{f^c}_{a,b}(z)$.
	Then $f^c \vee (L^{f^c}_{a,b} + \eta)$ is convex, is a minorant of $f$ and is strictly larger than $f^c$ (in particular at $z$) again contradicting the maximality of $f^c$ as a convex minorant of $f$.
	The case when one (or both) of the endpoints of $(a,b)$ are infinite can be reduced to the previous finite case. Indeed, suppose $f>f^c$ on $(a,\infty)$ (resp. $(-\infty,b)$) with $a\in\R$ (resp. $b\in\R$). If $f^c$ is not a straight line on $(a,\infty)$ (resp. $(-\infty,b)$), then there exists $b\in \R$ (resp. $a\in\R$) with $a<b$ such that $f^c$ is not a straight line on $(a,b)$ and \eqref{eq:f<L} holds for all $x\in(a,b)$. The rest of the argument remains the same. Finally, if $f>f^c$ on $\R$ but $f^c$ is not linear, then we can find $a,b\in\R$ with $a<b$ such that $f^c$ is not linear on $(a,b)$ and \eqref{eq:f<L} holds. We conclude as before.	
\end{proof}

\begin{proof}[Proof of \cref{lem:convh1}]
If $f=\infty$ on $\R$ then $f^c = \infty$, $X(y)=y=Z(y)$, and $L^f_{X(y),Z(y)}(y)=\infty=f^c(y)$. Henceforth we exclude this degenerate case.

By continuity of $f$ and necessarily of $f^c$, if $X(y)>-\infty$ then $f^c(X(y)) = f(X(y))$; similarly, if $Z(y)<\infty$ then $f^c(Z(y)) = f(Z(y))$.

Fix $y \in \R$. If $X(y)=y=Z(y)$ then $f^c(y) = f(y) = L^f_{y,y}(y) = L^f_{X(y),Z(y)}(y)$, as required.

So we may suppose $X(y)<Z(y)$. Suppose for now that $-\infty < X(y) < Z(y) < \infty$. 
By definition $f> f^c$ on $(X(y),Z(y))$. Then, by Lemma~\ref{lem:linear}, $f^c$ is linear on $(X(y),Z(y))$ and, by continuity of $f$, $f^c(X(y)) = f(X(y))$ and $f^c(Z(y)) = f(Z(y))$. Then $f^c(k) =  L^{f^c}_{X(y),Z(y)}(k)=L^{f}_{X(y),Z(y)}(k)$ on $[X(y),Z(y)]$. Applying this result at $k=y$ we have $f^c(y) =  L^{f}_{X(y),Z(y)}(y)$ as required.

We want to extend this result to the case where one, or both, of $\{X(y),Z(y)\}$ is infinite. Suppose $-\infty<X(y)<Z(y)=\infty$ (the case of $-\infty = X(y)< Z(y)< \infty$ can be treated symmetrically). By Lemma~\ref{lem:linear}, $f^c$ is a straight line on $(X(y),\infty)$ and by continuity of $f$, $f(X(y)) = f^c(X(y))$. Then $f^c(k) = f(X(y)) + \theta (k-X(y))$ on $[X(y),\infty)$, for some $\theta$ to be determined. Let $\phi = \liminf_{k \uparrow \infty} f(k)/k$; we show that $\phi = \theta$ for then $f^c(y) = L^f_{X(y),\infty}(y)$. We have
\[ \theta = \lim_{ k\uparrow\infty} \left[ \frac{f^c(k) - f(X(y))}{k-{ X(y)}} \right] \leq \liminf_{ k\uparrow\infty} \left[ \frac{f(k) - f(X(y))}{k-{ X(y)}} \right] = \liminf_{ k\uparrow\infty} \frac{f(k)}{k}  = \phi. \]
If $\theta < \phi$ then let $J = \inf_{ k } (f(k) - \phi k)$. Then $f^c \vee (J + \phi k)$ is convex, is a minorant of $f$ and is greater than $f^c$ for large enough $k$, contradicting the maximality of $f^c$. Hence $\theta = \phi$, and then $f^c(y)=L^f_{X(y),Z(y)}(y)$.

Finally, suppose $-\infty = X(y)<Z(y)=\infty$. Then, by Lemma~\ref{lem:linear}, $f^c$ is a straight line, and either $f^c(y) = \alpha + \theta y$ with $\alpha,\theta \in \R$, or $f^c = - \infty$ or $f^c=\infty$. The latter is ruled out since $f^c \leq f$ and at least for some $x$, $f < \infty$ by non-degeneracy.

Suppose $\psi + \phi >0$. Then both $\psi,\phi>-\infty$. We first show that $f^c\neq-\infty$. Suppose that $\psi,\phi\in\R$. Let $\epsilon$ be given by $2 \epsilon = \psi + \phi > 0$. Then $f(y) - (\phi - \epsilon) y$ is bounded below on $\R_+$ by $\beta_+ \in \R$ say and $f(y) - (\psi - \epsilon) |y|$ is bounded below on $\R_-$ by $\beta_- \in \R$. Then, if $\beta = \beta_+ \wedge \beta_-$, $f(y) \geq f^c\geq \beta + (\phi - \epsilon) y> -\infty$, for all $y\in\R$. If $\psi=\phi=\infty$, then by the continuity of $f$ and since $f<\infty$ for some $x\in\R$, $f$ is bounded below on $\R$ by some $\alpha\in\R$, and then $f^c\geq\alpha$ everywhere.  Finally, suppose that $\psi=\infty$ and $\phi\in\R$ (the case of $\psi\in\R$ and $\phi=\infty$ follows by symmetry). Then, if $\epsilon>0$, $f(y)-(\phi-\epsilon)y$ is bounded below on $\R$, and it follows that we cannot have $f^c=-\infty$.

So suppose $\psi + \phi>0$ and $f^c(y) = \alpha + \theta y \leq f(y)$, $y\in\R$. We show that this leads to a contradiction (to the fact that $-\infty = X(y)<Z(y)=\infty$). If $\psi + \phi>0$ then either $\psi > - \theta$ or $\phi>\theta$ (or both). Suppose $\phi > \theta$ (the case of $\psi > -\theta$ can be treated similarly). Then, there exists $y_0\in\R$ such that for $y \geq y_0$, $f(y) \geq f^c(y_0) + \phi (y-y_0) > f^c(y_0) + \theta (y-y_0)=f^c(y)$. But then $(\alpha + \theta y) \vee (f^c(y_0)+ \phi(y-y_0))$ is a convex minorant of $f$ which is greater than $f^c$.

Now suppose $\psi + \phi<0$. Choose $z_n \uparrow \infty$ such that $f(z_n)/z_n \rightarrow \phi$ and $x_m \downarrow - \infty$ such that $f(x_m)/|x_m| \rightarrow \psi$. By (5) (i.e., the result of Rockafellar that $f^c(y) = \inf_{x \leq y \leq z} L^f_{x,z}(y)$), for big enough $m$ and $n$, we have $f^c(y) \leq L^f_{x_m,z_n}(y)$. Hence, if $-\infty < \phi + \psi < 0$, then $-\infty\leq f^c(y) \leq \lim_m( \lim_n L^f_{x_m,z_n}(y) ) = \lim_m L^f_{x_m,\infty}(y) = L^f_{-\infty,\infty}(y) = -\infty$.

Similarly, if $\phi = - \infty$, then $L^f_{x,\infty}(y) = -\infty$ for $x<y$ and $\lim_m L^f_{x_m,\infty}(y) = -\infty$, irrespective of the value of $\psi$. If on the other hand, $\psi = -\infty$, then we reverse the order of taking limits and use $f^c(y) \leq \lim_n( \lim_m L^f_{x_m,z_n}(y) ) = \lim_n L^f_{-\infty,z_n}(y) = -\infty$. Thus, if $\phi+\psi = -\infty$ we have both $L^f_{-\infty,\infty} = -\infty$ and $f^c=-\infty$ and then $f^c= L^f_{X(y),Z(y)}$.

Finally, suppose $\phi + \psi=0$ and $L^f_{-\infty,\infty}(y) =\gamma+ \phi y$. For any $k \in \R$ we have that $\gamma \leq f(k) - \phi k$ and hence $L^f_{-\infty,\infty}(k) = \gamma + \phi k \leq f(k)$. Hence $L^f_{-\infty,\infty}$ is a convex minorant of $f$.

Suppose $\gamma = -\infty$, so that $L^f_{-\infty,\infty}=-\infty$. Let $g$ be any convex function with $g \leq f$. We show that $g$ is identically equal to $-\infty$. Suppose not. Then there exists $k$ such that $g(k)>-\infty$ and by the convexity of $g$, $g(y) \geq g(k) + (y-k)g'(k-)$. Dividing by $y$ and letting $y \rightarrow \infty$ we conclude $\phi \geq g'(k-)$. Dividing by $|y|$ and letting $y \rightarrow -\infty$ we conclude $\psi \geq - g'(k-)$. Then $-\psi = \phi \geq g'(k-) \geq - \psi$ and $f(y) \geq g(y) \geq g(k) + \phi(y-k)$. Then $\inf_w (f(w) - \phi w) \geq g(k)-\phi(k)>-\infty$, contradicting the fact that $\gamma=-\infty$. Finally, since $g =-\infty$ is the only convex minorant of $f$ we have that $f^c = -\infty = L^f_{-\infty,\infty}$.

Now suppose $\gamma\in\R$. We first show that $f^c$ must be linear. Suppose not. Then, since $f^c\geq L^f_{-\infty,\infty}$, either there exists $k_+\in\R$ such that $f^c>L^f_{-\infty,\infty}$ on $(k_+,\infty)$ or there exists $k_-\in\R$ such that $f^c>L^f_{-\infty,\infty}$ on $(-\infty,k_-)$ (or both). In either case $f^c(k)>f(k)$ for some (large or small enough) $k\in\R$, a contradiction.

Hence suppose $f^c(y) = \alpha + \theta y$, $y\in\R$, for some $\alpha,\theta\in\R$. We aim to show that $\theta = \phi$ and $\alpha = \gamma$ so that $f^c = L^f_{-\infty,\infty}$. Suppose $\theta < \phi$. Then for large enough $y$, $f^c(y)<L^f_{-\infty,\infty}(y)$, so that $f^c \vee L^f_{-\infty,\infty}$ is a convex minorant of $f$ which is bigger than $f^c$, thus contradicting the maximality of $f^c$.  $\theta>\phi$ can be ruled out similarly by considering large negative $y$; hence $\theta = \phi$.
Finally, choose $k_n$ such that $f(k_n)-\phi k_n \downarrow \gamma$. Then $0 \leq f(k_n) - f^c(k_n) = f(k_n) - \phi k_n - \alpha \rightarrow \gamma - \alpha$. Hence $\gamma \geq \alpha$ and $f^c \leq L^f_{-\infty,\infty} \leq f$. Since $f^c$ is the largest convex minorant we conclude $\gamma = \alpha$ and $f^c = L^f_{-\infty,\infty}$.
\end{proof}

\begin{proof}[Proof of \cref{lem:convh2}]
Fix $y\in\R$ and suppose $(x,z)\in\sB(y)$ are such that $f(k)>L^f_{x,z}(k)$ for all $k\in(x,z)$. If $X(y)=-\infty$, then we trivially have that $X(y)=-\infty\leq x$. Now suppose that $X(y)$ is finite. Using again that $f\geq f^c$ and $f^c$ is convex we have that
$$
f(k)>L^f_{x,z}(k)\geq L^{f^c}_{x,z}(k)\geq f^c(k),\quad k\in(x,z).
$$
Therefore, if $X(y)\in(x,z)$, we have a contradiction since $f(X(y))=f^c(X(y))$. A symmetric argument shows that $z\leq Z(y)$.

\end{proof}
\subsection{The Geometric construction}
\label{ssec:proofsGC}

\begin{proof}[Proof of Lemma \ref{lem:RSfinite}]
First note that, since $\{k:D(k)>0\}=(\ell_\nu,r_\nu)\supseteq(\ell_\mu,r_\mu)$ and by hypothesis $\mu(\{\ell_\nu\})=0$ (resp. $\mu(\{r_\mu\})=0$) in the case $\ell_\mu=\ell_\nu$ (resp. $r_\mu=r_\nu$), we have that $G(u)\in(\ell_\nu,r_\mu)$ for all $u\in(0,1)$. Hence $S(u)<r_\nu$ (resp. $\ell_\nu<R(u)$) in the case $G(u)=S(u)$ (resp. $R(u)=G(u)$). Recall that if $R(u)=G(u)$, then $R(u)=Q(u)=G(u)=S(u)$ (see Lemma \ref{lem:R=G=S}).

If $-\infty\leq R(u)<\ell_\nu$ (or $-\infty=R(u)=\ell_\nu$), then $\sE^c_u(R(u))=\sE_u(R(u))=D(R(u))= 0$ (or $\lim_{k\downarrow R(u)}\sE_u(R(u))=\lim_{k\downarrow R(u)}\sE^c_u(R(u))=\lim_{k\downarrow R(u)}D(k)=0$). In both cases we have that $L^{\sE^c_u}_{R(u),S(u)}$ has zero slope. It follows that either $S(u)<-\infty$ and $\sE^c_u(S(u))=\sE_u(S(u))=D(S(u))=0$ or $S(u)=\infty$ and then $\sE^c_u\equiv 0$. In the former case we must have that $r_\nu\leq S(u)<\infty$. But then $\sE^c_u$ is linear on $(R(u),S(u))$, and thus (see Corollary \ref{cor:shadow_potential}) $\nu-S^\nu(\mu_u)$ does not charge $[\ell_\nu,r_\nu)$. It follows that we have $(\mu-\mu_u)\leq_{cx}(\nu-S^\nu(\mu_u))\leq \nu(\{r_\nu\})$, and then $(\mu - \mu_u)$ must be a point mass at $r_\nu$ which is excluded by hypothesis, since $\mu$ places no mass there. In the latter case, since $\sE^c_u\equiv 0$, Corollary \ref{cor:shadow_potential} implies that $\nu=S^\nu(\mu_u)$. But $\nu(\R)=1>u=\mu_u(\R)=S^\nu(\mu_u)(\R)$, a contradiction. We conclude that either $-\infty<\ell_\nu\leq R(u)$ or $-\infty=\ell_\nu<R(u)$.

Now we deal with $S(u)$. Suppose that either $r_\nu<S(u)\leq\infty$ or $r_\nu=S(u)=\infty$. Then, by Theorem \ref{thm:shadow_potential}, linearity of $\sE^c_u$ on $(R(u),S(u))$ implies that $\nu-S^\nu(\mu_u)$ does not charge $(R(u),r_\nu]$, so that $\textrm{supp}(\nu-S^\nu(\mu_u))\subseteq (-\infty,R(u)]$. On the other hand, $\textrm{supp}(\mu-\mu_u)\subseteq[G(u),\infty)$. But, $S(u)>G(u)$, and therefore, by Lemma \ref{lem:R=G=S}, $R(u)\leq Q(u)<G(u)$. This contradicts the fact that $(\mu-\mu_u)\leq_{cx}(\nu-S^\nu(\mu_u))$. It follows that either $S(u)\leq r_\nu<\infty$ or $S(u)<r_\nu=\infty$.
\end{proof}

\begin{proof}[Proof of \cref{lem:chagree}]
	We have $\sE_v(z) \leq \sE_u(z) = \sE^c_v(z) \leq \sE_v(z)$ so that $\sE_v(z) = \sE_u(z)$.
	
	Consider ${\tE} : \R \mapsto \R$ given by $\tE = \sE^c_u$ on $(-\infty,z]$ and $\tE = \sE^c_v$ on $(z,\infty)$. If we can show that $\tE \leq \sE_v$ and $\tE$ is convex then $\tE \leq \sE^c_v$ everywhere and therefore $\sE^c_u \leq \sE^c_v$ on $(-\infty,z]$. But $\sE^c_v \leq \sE^c_u$ everywhere, and in particular $\sE^c_u = \sE^c_v$ on $(-\infty,z]$.
	
	To show that $\tE \leq \sE_v$ note that $\sE_v < \sE_u$ on $(G(u+),\infty)$ and so we must have $z \leq G(u+)$. Then, for $x \in (-\infty,z]$, $\tE(x) = \sE^c_u(x) \leq \sE_u(x)=D(x) = \sE_v(x)$ whereas for $x>z$, $\tE(x) = \sE^c_v(x) \leq \sE_v(x)$.
	
	To show that $\tE$ is convex note that if $x<y \leq z$ then for $\lambda \in (0,1)$, by convexity of $\sE^c_u$,
	\[ \tE(\lambda x + (1-\lambda)y) =  \sE^c_u(\lambda x + (1-\lambda)y) \leq  \lambda \sE^c_u(x) + (1-\lambda)\sE^c_u(y) =  \lambda \tE(x) + (1-\lambda)\tE(y) . \]
	We obtain a similar inequality for $z \leq x < y$ using the convexity of $\sE^c_v$.
	
	So suppose $x<z<y$. Suppose $\lambda$ is such that $\lambda x + (1-\lambda)y \geq z$. Then
	\begin{eqnarray*}
		\lefteqn{ \tE(\lambda x + (1-\lambda)y) =  \sE^c_v(\lambda x + (1-\lambda)y) \leq  \lambda \sE^c_v(x) + (1-\lambda)\sE^c_v(y) } \\
		&& \leq  \lambda \sE^c_u(x) + (1-\lambda)\sE^c_v(y) =  \lambda \tE(x) + (1-\lambda)\tE(y).
	\end{eqnarray*}
	
	Finally suppose that $x<z<y$ and $\lambda x + (1-\lambda)y < z$. For $x<r<k$ consider $f=f(k,r)$ given by
	\[ f(k,r) = \sE^c_u(x) + \frac{r-x}{k-x} ( \sE^c_v(k) - \sE^c_u(x)). \]
	Since $\sE^c_u(x) \geq \sE^c_v(x)$ and $\sE^c_v$ is convex it is easily seen that $f(k,r)$ is increasing in $k$.
	Then
	\begin{eqnarray*}
		\lambda \tE(x) + (1-\lambda) \tE(y) & = & \sE^c_u(x) + (1-\lambda) (\sE^c_v(y) - \sE^c_u(x)) \\
		& = & f(y,\lambda x + (1-\lambda)y) \\
		& \geq & f(z,\lambda x + (1-\lambda)y) \\
		& = & \sE^c_u(x) +  \frac{(1-\lambda)(y-x)}{z-x} (\sE^c_v(z) - \sE^c_u(x)) \\
		&=& \left( 1 - \frac{(1-\lambda)(y-x)}{z-x} \right) \sE^c_u(x) + \frac{(1-\lambda)(y-x)}{z-x} \sE^c_u(z) \\
		& \geq & \sE^c_u \left( \left( 1 - \frac{(1-\lambda)(y-x)}{z-x} \right)x +  \frac{(1-\lambda)(y-x)}{z-x} z     \right) \\
		& = & \sE^c_u(\lambda x + (1 - \lambda y)) = \tE(\lambda x + (1 - \lambda y)).
	\end{eqnarray*}
	\end{proof}
\begin{proof}[Proof of \cref{cor:R=R}]
	Set $s = S(u)=S(v)$. Then $G(u+) \leq G(v) \leq S(v)=s \leq G(u+)$ so that $G(u+)=G(v)=s$ also.
	
	By \cref{lem:chagree}, in order to conclude that  $\sE^c_u = \sE^c_v$ on $(-\infty, S(v)]$ it is sufficient to show that $\sE^c_u(s)= \sE^c_v(s)=\sE_v(s)$. But $\sE^c_u(s) = \sE_u(s)$ since $s=S(u)$, $\sE_u(s)=\sE_v(s)$ since $s \leq G(u+)$ and $\sE_v(s)= \sE^c_v(s)$ since $s=S(v)$, and we are done.
	
	Inspection of the proof of \cref{lem:phiRepresentation} shows that provided we replace $D(r)$ with $\sE_u(r)$ in the expression for $\phi(u)$ in \eqref{eq:phiR}, the supremum over $k<G(u)$ can be replaced by a supremum over $k<S(u)$. (The two cases can be considered separately, and if $G(w)=S(w)$ there is nothing to prove.) Then, since $S(u)=s=S(v)$ and $\sE_u = \sE_v$ on $(-\infty,s]$,
	\[ \phi(u) = \sup_{r<S(u)} \frac{\sE_u(S(u)) - \sE_u(r)}{S(u)- r} = \sup_{r<S(v)} \frac{\sE_v(S(v)) - \sE_v(r)}{S(v)- r} = \phi(v). \]
	Further, since $\sE^c_u = \sE^c_v$ on $(-\infty,S(u)=s=S(v)]$ and $R(\cdot)$ depends on $\sE_\cdot$ only through $\sE^c_{\cdot}$ on $(-\infty, S(\cdot)]$, we have $R(u)=R(v)$.
\end{proof}

\begin{proof}[Proof of Lemma \ref{lem:RdecreasingA<}]
	Fix $n\geq1$ and $u,v\in A^n_<$ with $u<v$. We have that $R(u)<G(u)\leq G(u+)<S(u-)\leq S(u)$. Furthermore, by Theorem \ref{thm:monotone}, $R(v)\notin(R(u),S(u))$. We will show that $R(v)\leq R(u)$.

Suppose not, so that $S(u)\leq R(v)$. We aim to find a contradiction. Define $\bar u:=\inf \{w\in(u,v] : S(u)\leq R(w)\}$ and note that, since $v \in \{w\in(u,v] : S(u)\leq R(w)\}$, $\bar u$ is well-defined in $[u,v]$.
	
	First note that $\bar u>u$. Indeed, since $G(u+)<S(u)$, there exists $\bar\epsilon>0$ such that $G(u+)\leq G(u+\epsilon)<S(u)$ for all $0\leq\epsilon\leq\bar\epsilon$. But then $R(u+\epsilon)\leq G(u+\epsilon)<S(u)$ for all $0\leq\epsilon\leq\bar\epsilon$. It follows that $\bar u > u$.
	
	We claim that $R(\cdot)$ is decreasing on $[u,\bar u)$. Let $u_1,u_2\in[u,\bar u)$ with $u_1<u_2$. For $w\in\{u_1,u_2\}$ $R(w)<S(u)$ and then by the results of the previous paragraph and left-monotonicity, $R(w)\leq R(u)<S(u)\leq S(u_1)\leq S(u_2)$. Since $R(u_2)\notin(R(u_1),S(u_1))$ it follows that $R(u_2)\leq R(u_1)$, which proves the claim.
	
	Now observe that $R(\bar u)\notin(R(w),S(w))$ for all $w\in [u,\bar u)$, and therefore, by the monotonicity of $R(\cdot)$ and $S(\cdot)$ on $[u,\bar u)$, we have that $R(\bar u)\notin (R(\bar u-),S(\bar u -))$. But $\bar u \in A^n_<$, so that $G(\bar u)\leq G(\bar u+)<S(\bar u-)$. Hence, since $R(\bar u)\leq G(\bar u)$, we must have that $R(\bar u)\leq R(\bar u-)$ and therefore $R(\bar u)<S(u)$. In particular, $\bar u<v$.
	
	Finally, we have that $R(\bar u)<G(\bar u)\leq G(\bar u+)<S(\bar u-)\leq S(\bar u)$ and we can pick $\tilde\epsilon>0$ such that $G(\bar u+\epsilon)<S(\bar u)$ for all $0\leq \epsilon\leq\tilde\epsilon$. But then $R(\bar u +\epsilon)\leq G(\bar u+\epsilon)<S(\bar u)$ and therefore $R(\bar u+\epsilon)\leq R(\bar u)<S(u)$ for all $0\leq \epsilon\leq \tilde\epsilon$, contradicting the definition of $\bar u$ as the infimum of points in $\{w\in(u,v]:S(u)\leq R(w)\}$. This is the desired contradiction and we conclude that $R(v)\leq R(u)$.
	\end{proof}
\subsection{Left Continuity}
\label{ssec:proofsLC}
\begin{proof}[Proof of Lemma~\ref{lem:Elim}]
	By the left-continuity of $G$ we have
	\begin{align*}
	\lim_{u\uparrow v}\sE_u(k)=\lim_{u\uparrow v}\{P_\nu(k)-P_\mu(k\wedge G(u))+u(k-G(u))^+\}
	=P_\nu(k)-P_\mu(k\wedge G(v))+v(k-G(v))^+=\sE_v(k).
	\end{align*}
	Also, for all $u\in(0,1)$ with $v<u$, by \eqref{eq:Ediff} we have that
	$$
	0\leq \sE_v(k)-\sE_u(k)\leq P_\mu(G(u)) - P_\mu(G(v))-u(G(u)-G(v))+(u-v)(k-G(v))^+,\quad k\in\R.
	$$
	But, since $\mu$ does not charge $(G(v),G(v+))$ (provided $G(v)<G(v+)$), $P_\mu$ is linear on $(G(v),G(v+))$ with slope $v$, and we have that
	$$
	\lim_{u\downarrow v}\{P_\mu(G(u)) - P_\mu(G(v))-u(G(u)-G(v))\}=P_\mu(G(v+)) - P_\mu(G(v))-v(G(v+)-G(v))=0.
	$$
	It follows that, for each $k\in\R$, $\sE_u(k)\uparrow \sE_v(k)$ as $u\downarrow v$.
	
	For convergence of convex hulls note that, for $u\in(0,1)$ with $u<v$, $\sE_u(k)\geq\sE_v(k)\geq\sE^c_v(k)$ and therefore $\sE_u(k)\geq\sE^c_u(k)\geq\sE^c_v(k)$, $k\in \R$. Then
	$$
	\sE_v(k)=\lim_{u\uparrow v}\sE_u(k)\geq\lim_{u\uparrow v}\sE^c_u(k)\geq\sE^c_v(k),\quad k\in\R.
	$$
	Since a point-wise limit of convex functions is convex, $\lim_{u\uparrow v}\sE^c_u$ is a convex minorant of $\sE_v$. Hence $\sE^c_v\geq \lim_{u\uparrow v}\sE^c_u$ and equality follows. Similarly, for $u\in(0,1)$ with $v<u$, we have that $\sE_v\geq\sE_u\geq\sE_u^c$, and therefore $\sE^c_v\geq\lim_{u\downarrow v}\sE^c_u$.
	
	It remains to show that $\sE^c_v\leq\lim_{u\downarrow v}\sE^c_u$. First note that if $s_n \rightarrow s$ and $v_n \downarrow v$ then $\sE_{v_n}(s_n) \rightarrow \sE_v(s)$. To see this note that $\sE_{v_n}(s_n) \leq \sE_v(s_n)$ and so $\lim_n \sE_{v_n}(s_n) \leq \lim_n \sE_v(s_n) = \sE_v(s)$ by the continuity of $\sE_v$. Conversely, since $|\frac{\sE_{u}(k) - \sE_u(j)}{j-k}| \leq 1$ for all $j \neq k$ and $u \in (0,1)$, $\sE_{v_n}(s_n) \geq \sE_{v_n}(s) - |s_n - s|$ and $\lim_n \sE_{v_n}(s_n) \geq \lim_n \sE_{v_n}(s) = \sE_v(s)$.
	
	Fix $k$ such that $D(k)>0$. Either there exists a sequence $v_n \downarrow v$ such that $\sE_{v_n}^c(k) = \sE_{v_n}(k)$ or there exists a sequence $v_n \downarrow v$ such that for each $n$ there exists $r_n = r_n(k)$ and $s_n=s_n(k)$ with $r_n<k < s_n$ with
	\begin{equation}
	\label{eq:sEvn}
	\sE^c_{v_n}(k) = \sE_{v_n}(r_n) \frac{(s_n-k)}{s_n-r_n} + \sE_{v_n}(s_n) \frac{(k-r_n)}{s_n-r_n} = \sE_{v_n}(s_n) - \frac{\sE_{v_n}(s_n) - \sE_{v_n}(r_n)}{s_n-r_n}(s_n-k).
	\end{equation}
	In the former case we have
	\[ \sE^c_v(k) \leq \sE_v(k) = \lim_n \sE_{v_n}(k) = \lim_{v_n \downarrow v} \sE^c_{v_n}(k). \]
	In the latter case we can choose a subsequence such that $s_n \rightarrow s \geq k$ and $r_n \rightarrow r \in [-\infty, k]$. If $s=k$ then from \eqref{eq:sEvn} and using
	$|\frac{\sE_{v_n}(s_n) - \sE_{v_n}(r_n)}{s_n-r_n}| \leq 1$,
	\[ \lim_n \sE^c_{v_n}(k) = \lim_n \sE_{v_n}(s_n) = \sE_{v}(k) \geq \sE^c_v(k). \]
	Otherwise, if $s>k$, then taking limits in \eqref{eq:sEvn},
	\[ \lim_n \sE^c_{v_n}(k) = \sE_v(s) - \frac{\sE_v(s) - \sE_v(r)}{s-r} (s-k) =
	\sE_v(r) \frac{s-k}{s-r} + \sE_v(s)\frac{k-r}{s-r} \geq \sE^c_v(k) . \]
	
\end{proof}

\begin{proof}[Proof of Lemma~\ref{lem:hullConv}]
	Fix $u,v\in(0,1)$ with $u<v$. Since $\sE_u\geq\sE_v$, $\sE_u(f(u))\leq\sE_u(k)$ for all $k\geq f(u)$, and $f$ is non-decreasing, using \eqref{E} we have that
	\begin{equation}\label{eq:convE}
	\sE_v(f(v-))=\lim_{u\uparrow v}\sE_v(f(u))\leq \lim_{u\uparrow v}\sE_u(f(u))\leq \lim_{u\uparrow v}\sE_u(f(v-))=\sE_v(f(v-)).
	\end{equation}
	Equation \eqref{eq:convE} still holds when $\sE$ is replaced by $\sE^c$ (note that $\sE^c_u$ is non-decreasing everywhere for each $u\in(0,1)$), which concludes the proof.
\end{proof}

\subsection{The candidate coupling is an embedding}
\label{ssec:proofsE}
\begin{proof}[Proof of Lemma \ref{lem:monotone}]
First we argue that $(T,S)$ is left-monotone with respect to $G$. We need to show that if $v>u$ then $T(v) \notin (T(u),S(u))$.

If $S(u)=G(u)$ then $(T(u),S(u)) = \emptyset$ so there is nothing to prove. So take $T(u)=R(u)$.
Suppose first that $G(v) = S(v)$. Then $T(v)=G(v) =S(v) \geq S(u)$.
Alternatively suppose $G(v)<S(v)$. Then $T(v)=R(v)$. Finally, properties of $R$ (see Theorem \ref{thm:monotone}) imply that either $T(v) = R(v) \geq S(u)$ and $T(v) \geq S(u)$ or $T(v) = R(v) \leq R(u) = T(u)$.

Measurability of $T$ now follows easily. We have $(0,1) = (\cup_{n \geq 1} A^n_<) \cup (u : G(u) < G(u+), G(u) \leq S(u) \leq G(u+) ) \cup (u : G(u)=S(u)=G(u+))$. Then $(T(u) \leq t) = (\cup_{n \geq 1} A^n_< \cap \{ u: T(u) \leq t \} ) \cup \{ u : G(u) < G(u+), G(u) \leq S(u) \leq G(u+), T(t) \leq t \}  \cup \{ u : G(u)=S(u)=G(u+) , T(u) \leq t \}$.

But $T$ is decreasing on $A_<^n$ and so $A^n_< \cap \{ u: T(u) \leq t \}$ is a Borel subset of $(0,1)$; $\{ u : G(u)=S(u)=G(u+) , T(u) \leq t \} = \{ u : G(u)=S(u)=G(u+) , G(u) \leq t \}$ is Borel from the measurability of $G$ and $S$, and $\{ u : G(u) < G(u+), G(u) \leq S(u) \leq G(u+), T(t) \leq t \}$ is countable. Hence $T$ is measurable.

\end{proof}

\begin{proof}[Proof of Lemma~\ref{lem:psi}]
From \cref{lem:phi}(ii) we have $\phi(v+) \geq \phi(v)$. Further, for $v \leq w \leq u \leq 1$ we have $\phi(w)-\phi(u) \leq u-w$. Introduce $\tp:[v,1] \rightarrow [0,1]$ by $\tp(v)=x$ and $\tp(w)=\phi(w)$ for $w \in (v,1]$. Then, for $v < u \leq 1$, $\tp(v)-\tp(u) \leq \lim_{n \uparrow \infty} \phi(v + 1/n) - \phi(u) \leq \lim_{n \uparrow \infty} \{ u - (v+1/n) \} = u-v$. By checking the easy cases $w=v=u$ and $v < w \leq u \leq 1$ separately we conclude $\tp(w) - \tp(u) \leq u-w$ for all $v \leq w \leq u \leq 1$.

By construction, $\psi_{v,x}(w) = \inf_{v \leq u \leq w} \{ \tp(u) \}$. Then
$\psi_{v,x}$ is decreasing, $\psi_{v,x}(v) = \tp(v)=x$ and $\psi_{v,x} \leq \tp$ on $[v,1]$. Fix $w$ and $u$ with $v \leq w<u$. Let $(u_m)_{m \geq 1}$ with $v \leq u_m \leq u$ be such that $\lim_{m \uparrow \infty} \tp(u_m) \downarrow \psi_{v,x}(u)$. Taking a convergent subsequence if necessary we may assume $u_m \rightarrow \tilde{u} \in [v,u]$. Then, if $w<\tilde{u}$,
\[ 0 \leq  \psi_{v,x}(w) - \psi_{v,x}(u) \leq \tp(w) - \lim_m \tp(u_m) = \lim_m \left\{ \tp(w) - \tp(u_m) \right\}\leq \lim_m \left\{ u_m - w \right\} = \tilde{u} - w \leq u-w. \]
On the other hand, if $\tilde{u}<w$, $\psi_{v,x}(w)\leq \lim_m \tp(u_m) = \psi_{v,x}(u)$ so that $\psi_{v,x}(w) = \psi_{v,x}(u)$. Finally, if $\tilde{u}=w$ then either there exists a sequence $u_m \rightarrow \tilde{u}$ with $\lim_m \tp(u_m) = \psi_{v,x}(u)$ and $u_m \geq \tilde{u}$, or there exists a sequence $u_m \rightarrow \tilde{u}$ with $\lim_m \tp(u_m) = \psi_{v,x}(u)$ and $u_m \leq \tilde{u}$ (or both). In either case the corresponding proof shows that $\psi_{v,x}(w) = \psi_{v,x}(u)$. It follows that $0 \leq \psi_{v,x}(w) - \psi_{v,x}(u) \leq u-w$ so that $\psi_{v,x}$ is absolutely continuous on $[v,1]$ (and not just on $[v,1] \cap A_<$) and has a derivative $\psi'$ such that $\int_a^b \psi_{v,x}'(u) du = \psi_{v,x}(b) - \psi_{v,x}(a)$, for all $v\leq a\leq b\leq 1$.
\end{proof}

\subsection{Proof of Proposition~\ref{prop:setequality}}
\label{ssec:setequality}
Our goal is to show that $R(\cdot ) \leq y$ is equivalent to $\psi(\cdot) = \phi(\cdot)$ for a well-chosen element $\psi=\psi_{v,x}$. In particular, we want to show
	\begin{equation} \{ v : v>S^{-1}(y), R(v) \leq y  \} = \{ v : v> S^{-1}(y), \psi_{S^{-1}(y), (\sE^c_{S^{-1}(y)})'(y-)}(v) = \phi(v) \} \label{eq:claim3C0}
	\end{equation}
in a sufficiently rich set of circumstances, with the ultimate aim of proving Proposition~\ref{prop:setequality}.
	We begin with a partial result, valid in the case where $w=S^{-1}(y)$ is such that $R(w)<S(w)$. Later the issue will be to show that \eqref{eq:claim3C0} holds also for both $y$ such that $R(S^{-1}(y)) = S(S^{-1}(y))$ and for $y$ which are not of the form $y=S(w)$ for some $w$.

	\begin{lem}
		\label{lem:Rolz<Solz}
		Suppose $w$ is such that $R(w)<S(w)$. Then
		\[ \{ v : v>w, R(v) \leq R(w)  \} = \{ v : v> w, \psi_{w, \phi(w)}(v) = \phi(v) \}. \]
	\end{lem}

\begin{proof}
		Let $I^R_w = \{ v : v>w, R(v) \leq R(w)  \}$ and $I^\phi_w = \{ v : v> w, \psi_{w, \phi(w)}(v) = \phi(v) \}$.
		
		First we argue that, for all $u,v\in(0,1)$ with $u < v$, $\phi(u)< \phi(v)$ implies $R(v)\geq S(u)$. Suppose $u<v$ and $\phi(u)<\phi(v)$. Suppose $(S(v),\sE_v(S(v)))$ lies on or below the line $L^{\sE_u, \phi(u)}_{S(u)}$.
		Then $L^{\sE_v, \phi(v)}_{S(v)} < L^{\sE_u, \phi(u)}_{S(u)} \leq \sE_u$ on $(-\infty,S(v))\supseteq (-\infty,S(u))$ and $L^{\sE_v, \phi(v)}_{S(v)}<\sE_u=D$ on $(-\infty,G(u+))$. Hence $R(v) \geq G(u+)$ and by left-monotonicity $R(v) \geq S(u)$. Conversely, if $(S(v),\sE_v(S(v))$ lies above $L^{\sE_u, \phi(u)}_{S(u)}(S(v))$ then $L^{\sE_v, \phi(v)}_{S(v)}(S(u)) \leq \sE_v(S(u)) \leq \sE_u(S(u)) = L^{\sE_u, \phi(u)}_{S(u)}(S(u))$.
		Then $L^{\sE_v, \phi(v)}_{S(v)}(k) < L^{\sE_u, \phi(u)}_{S(u)}(k) \leq \sE_u(k) = \sE_v(k)$ for $k <  G(u) \leq G(v)$ and since $L^{\sE_v, \phi(v)}_{S(v)}(R(v)) = \sE_v(R(v))$ we must have $R(v) \geq G(u)$. If $R(u) < S(u)$ then $R(v) \geq G(u) > R(u)$ and then by left-monotonicity $R(v) \geq S(u)$. Similarly, if $R(u)=S(u)$ then $R(v) \geq G(u) = R(u)=S(u)$ and again $R(v) \geq S(u)$.
		
Now we show that $I^R_w \subseteq I^\phi_w$. 
Suppose $v>w$ but $v \notin I^\phi_w$. Then there exists $u \in [w,v)$ such that $\phi(u)<\phi(v)$. Then by the above argument, $R(v) \geq S(u) \geq S(w) > R(w)$ so that $v \notin I^R_w$. Hence $I^R_w \subseteq I^\phi_w$.
		
For the converse we show that if $v>w$ but $v \notin I^R_w$ (so that $R(v)>R(w)$) then there exists $u \in [w,v)$ such that $\phi(u)<\phi(v)$ and hence $v \notin I^\phi_w$.

So, suppose $R(v)>R(w)$. By left-monotonicity we deduce that $R(v) \geq S(w)$. Note also that since $R(w) < G(w+)$ and $w<v$, $\sE_v(R(w)) = \sE_w(R(w)) = D(R(w))$.

If $\sE_v(S(v)) > L^{\sE_w, \phi(w)}_{S(w)}(S(v))$ then since $R(w)<R(v) \leq G(v)$, by Lemma~\ref{lem:phiRepresentation} we have
\[ \phi(v) \geq \frac{\sE_v(S(v)) - \sE_v(R(w))}{S(v)-R(w)} > 	\frac{L^{\sE_w, \phi(w)}_{S(w)}(S(v))-L^{\sE_w, \phi(w)}_{S(w)}(R(w))}{S(v)-R(w)} = \phi(w). \]
It follows that $\phi(v) > \phi(w)$.

If $\sE_v(S(v)) = L^{\sE_w, \phi(w)}_{S(w)}(S(v))$ then, since $L^{\sE_v,\phi(v)}_{S(v)}\leq\sE^c_v\leq\sE^c_w$ everywhere, we must have that $\phi(v)\geq \phi(w)$. But if $\phi(w)=\phi(v)$ then $L^{\sE_v,\phi(v)}_{S(v)}$ and $L^{\sE_w,\phi(w)}_{S(w)}$ coincide and hence $R(v)=R(w)$, a contradiction to our hypothesis that $R(v)>R(w)$. It follows again that $\phi(v) > \phi(w)$.

Finally consider the case where $\sE_v(S(v))= L^{\sE_v, \phi(v)}_{S(v)}(S(v))< L^{\sE_w, \phi(w)}_{S(w)}(S(v))$. If $\sE_v(R(v)) = L^{\sE_v, \phi(v)}_{S(v)}(R(v))\geq L^{\sE_w, \phi(w)}_{S(w)}(R(v))$ then $R(w)<R(v)<S(v)$ and we must have $L^{\sE_v, \phi(v)}_{S(v)}(R(w)) > L^{\sE_w, \phi(w)}_{S(w)}(R(w))$. Then $\sE_w(R(w)) = \sE_v(R(w)) \geq L^{\sE_v, \phi(v)}_{S(v)}(R(w)) > L^{\sE_w, \phi(w)}_{S(w)}(R(w)) = \sE_w(R(w))$, a contradiction. Hence we must have $\sE_v(R(v)) < L^{\sE_w, \phi(w)}_{S(w)}(R(v))$.

Consider $\sE_v^c$. There exists an interval $[x_1,x_2]$ with $x_1 \leq R(w)< G(w+) < x_2$ such that $\sE^c_v$ is linear on $[x_1,x_2]$, $\sE_v^c(x_1)= \sE_v(x_1)=D(x_1)$ and $\sE_v^c(x_2)= \sE_v(x_2)$. Necessarily $\sE_v(x_1) \geq L^{\sE_w, \phi(w)}_{S(w)}(x_1)$ and $\sE_v(x_2) < L^{\sE_w, \phi(w)}_{S(w)}(x_2)$. There exists $u$ such that $G(u) \leq x_2 \leq G(u+)$. Note that since $x_2 > G(w+)$ we have $u>w$. Also $x_2 \leq R(v) \leq G(v)$, else $R(v) \leq x_1 \leq R(w)$ which contradicts our assumption that $R(v)>R(w)$ and hence we may take $u \leq v$. Finally, we must have $u<v$ since if $u = v$ we find $R(v)=R(u)=x_1\leq R(w)$, again a contradiction. In summary, $u \in (w,v)$.

Since $x_2 \leq G(u+) \leq G(v) \leq S(v)$ by the convexity of $\sE^c_v$ we have $(\sE^c_v)'(x_2-) \leq (\sE^c_v)'(S(v)-)$. If there is equality here then $\sE^c_v$ is linear on $(x_1,S(v))$ and $R(v) \leq x_1 \leq R(w)$, a contradiction. Hence there is strict inequality and $\phi(u) = (\sE^c_v)'(x_2-) < (\sE^c_v)'(S(v)-) = \phi(v)$. Hence $\psi_{w,\phi(w)}(v) \leq \phi(u) <\phi(v)$ and $v \notin I^\phi_w$.
	\end{proof}

Now we introduce a condition which will allow us to prove an analogue of \eqref{eq:claim3C0} in a wider set of circumstances.
\begin{defn}
Fix $x\in\R$ and $v\in(S^{-1}(x),1)$. Define $\sA(v,x) = \{ w : S^{-1}(x) < w \leq v, R(w) < x \}$. Then Condition $A(x)$ is that $\sA(v,x)$ is non-empty for each $v\in(S^{-1}(x),1)$.
\end{defn}

\begin{lem}
\label{lem:D<sE}
Suppose $D(y) < \sE_{S^{-1}(y)}(y)$. Then Condition $A(y)$ holds.
\end{lem}

\begin{proof}
It is clear that $D(y) < \sE_{S^{-1}(y)}(y)$ if and only if $G(S^{-1}(y)+) < y$.

Choose $h_0>0$ such that $G(S^{-1}(y)+h_0)<y$. Then for all $0<h \leq h_0$ we have
\( R(S^{-1}(y)+h) \leq G(S^{-1}(y)+h) \leq G(S^{-1}(y)+h_0)<y. \)
\end{proof}

\begin{lem}
\label{lem:D=sEand}
Suppose $y$ is a continuity point for both $\mu$ and $\nu$. Suppose further that $D(y) = \sE_{S^{-1}(y)}(y)$ and for all $v \in (S^{-1}(y),1)$, there exists $k$ such that $\sE_v(k) < L_y^{\sE_{S^{-1}(y)}, (\sE_{S^{-1}(y)}^c)'(y-)}(k)$. 

Then, Condition $A(y)$ holds.
\end{lem}

\begin{proof}
Write $L$ as shorthand for $L_y^{\sE_{S^{-1}(y)}, (\sE_{S^{-1}(y)}^c)'(y-)}$. Fix $h_0>0$ and choose $k>y$ such that $\sE_{S^{-1}(y)+h_0}(k) < L(k)$.
Such a $k$ exists under the assumptions of the lemma. See Figure \ref{fig:L}.


Since $y$ is a continuity point of both $\mu$ and $\nu$ we have $(\sE_{S^{-1}(y)}^c)'(y-) = (\sE_{S^{-1}(y)}^c)'(y+)$ and $(\sE_{S^{-1}(y)}^c)'(y) =\sE_{S^{-1}(y)}'(y) = D'(y)$.
For $j \in (y,k]$ define $f(j) = \frac{\sE_{S^{-1}(y)+h_0}(j) - D(y)}{j-y}$. Then $f$ is continuous on $(y,k]$ with $f(y+)= \sE_{S^{-1}(y)}'(y)$ and $f(k) < \frac{L(k) - D(y)}{k-y} = (\sE^c_{S^{-1}(y)})'(y)$. In particular $f$ attains its minimum value. Let $\underline{f}$ be this minimal value and let $\underline{j}$ be the smallest value at which it is attained. Then $\underline{j}>y$.
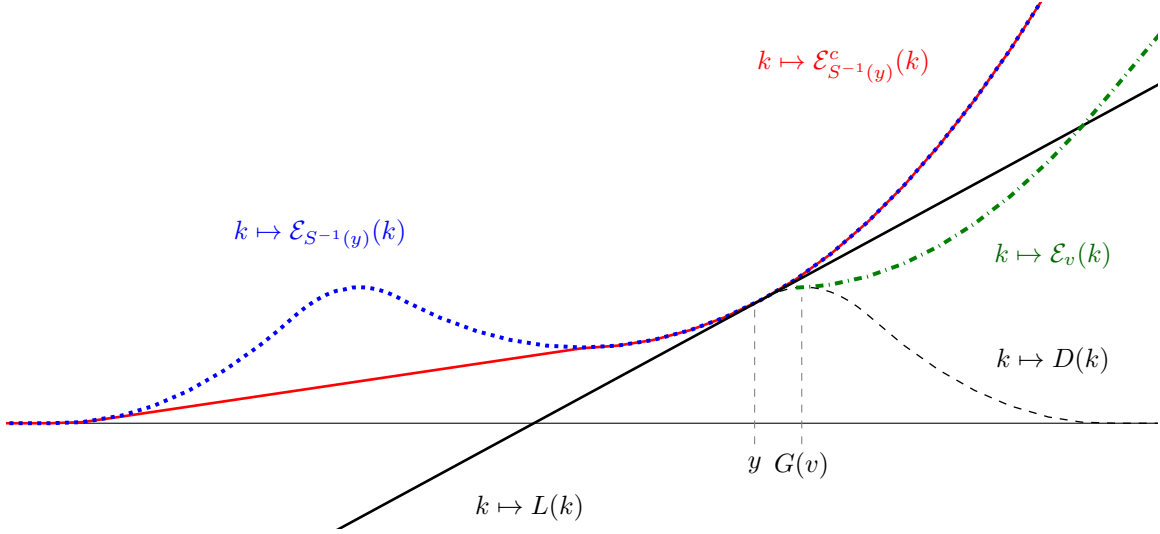
\begin{figure}[H]
	\centering
	\begin{tikzpicture}
	
	\begin{axis}[%
	width=6.028in,
	height=2.754in,
	at={(1.011in,0.642in)},
	scale only axis,
	xmin=-11,
	xmax=11,
	ymin=-0.5,
	ymax=2,
	axis line style={draw=none},
	ticks=none
	]
	\draw[black, thin] (-11,0) -- (11,0);
	\addplot [color=black, dashed, line width=0.5pt, forget plot]
	table[row sep=crcr]{%
		3.77732374207791	0.626769007966386\\
		3.9341164164935	0.637894275973635\\
		4.09090909090909	0.644599435510496\\
		4.25843861122547	0.647323499960791\\
		4.42596813154186	0.644090192466747\\
		4.59349765185824	0.636248956881996\\
		4.76102717217463	0.623378659481076\\
		4.93440674276733	0.604714517707976\\
		5.10778631336004	0.580653467036178\\
		5.28116588395275	0.551171501824431\\
		5.45454545454546	0.516301090797671\\
		6.14084012777783	0.372328187197632\\
		6.81818181818182	0.25309748923382\\
		7.50052689249177	0.156180892143828\\
		8.18181818181818	0.0826477350739729\\
		8.913369432737	0.0295193007739663\\
		9.54545454545454	0.00516201691444174\\
		10.1951138678739	-3.74967337890553e-06\\
		10.9090909090909	5.19110393604194e-06\\
		11.6245592768795	4.57978089229982e-06\\
		12.2727272727273	5.79947207590692e-06\\
		12.7873709632194	-0.000829551138547657\\
		13.6363636363636	2.4651729031433e-06\\
		14.3055157129462	-6.7231104949883e-06\\
		15	5.41978617896177e-06\\
	};
	\addplot [color=red, line width=1.0pt, forget plot]
	table[row sep=crcr]{%
		-15	0\\
		-14.3743426864006	0\\
		-13.6363636363636	0\\
		-12.922520765984	0\\
		-12.2727272727273	0\\
		-11.6037511150776	0\\
		-10.9090909090909	0\\
		-10.1567767603786	0\\
		-9.54545454545454	0.00516529194875333\\
		-8.78555408111793	0.0318688447181529\\
		-8.18181818181818	0.0541770432307099\\
		-7.52404914247162	0.0784817803221705\\
		-6.81818181818182	0.104563761552785\\
		-6.16274148089556	0.128782452497416\\
		-5.45454545454546	0.154950479874861\\
		-4.69396922041941	0.183053969553868\\
		-4.09090909090909	0.205337198196936\\
		-3.35965597297158	0.232357191105751\\
		-2.72727272727273	0.255723916519011\\
		-2.00961780284732	0.282241452640625\\
		-1.36363636363636	0.306110634841086\\
		-0.688142581770709	0.331070305778884\\
		0	0.356497353163162\\
		0.681996956288809	0.368970306835827\\
		1.36363636363636	0.395862010259944\\
		2.07255602717534	0.442839557985265\\
		2.72727272727273	0.503443615761336\\
		3.46373839324673	0.59138098261362\\
		4.09090909090909	0.684754393790167\\
		4.76102717217463	0.80625047021627\\
		5.45454545454546	0.955629694964023\\
		6.14084012777783	1.12711466162036\\
		6.81818181818182	1.31946216482611\\
		7.50052689249177	1.53642621219928\\
		8.18181818181818	1.77628405234954\\
		8.913369432737	2.05966908010521\\
		9.54545454545454	2.3260713980839\\
		10.1951138678739	2.61974906918045\\
		10.9090909090909	2.9481867282283\\
		11.6245592768795	3.27730226787211\\
		12.2727272727273	3.57546044699675\\
		12.7873709632194	3.81218623129582\\
		13.6363636363636	4.20273010502642\\
		14.3055157129462	4.51053028046035\\
		15	4.82999961174161\\
	};
	\addplot [color=blue, dotted, line width=1.5pt, forget plot]
	table[row sep=crcr]{%
		-15	0\\
		-14.3743426864006	0\\
		-13.6363636363636	0\\
		-12.922520765984	0\\
		-12.2727272727273	0\\
		-11.6037511150776	0\\
		-10.9090909090909	0\\
		-10.1567767603786	0\\
		-9.54545454545454	0.00516529194875333\\
		-8.78555408111793	0.0368719582992022\\
		-8.18181818181818	0.0826446406924073\\
		-7.52404914247162	0.153258338675332\\
		-6.81818181818182	0.253099065950084\\
		-6.16274148089556	0.368114076983243\\
		-5.80864346772051	0.43918676050732\\
		-5.45454545454546	0.516291124251744\\
		-5.26440139601394	0.554263855225197\\
		-5.07425733748243	0.585729263460034\\
		-4.88411327895092	0.610686122022407\\
		-4.69396922041941	0.62913544187045\\
		-4.54320418804183	0.635771770168381\\
		-4.39243915566425	0.645048476099583\\
		-4.24167412328667	0.64686852145073\\
		-4.09090909090909	0.64459693608525\\
		-3.90809581142471	0.636353531944546\\
		-3.72528253194034	0.622094567125119\\
		-3.54246925245596	0.601821845091591\\
		-3.35965597297158	0.577684437559403\\
		-2.72727272727273	0.503447332690059\\
		-2.00961780284732	0.437884749712861\\
		-1.36363636363636	0.395735125572964\\
		-0.688142581770709	0.369132963180525\\
		0	0.359996840005688\\
		0.681996956288809	0.368970306835827\\
		1.36363636363636	0.395862010259944\\
		2.07255602717534	0.442839557985265\\
		2.72727272727273	0.503443615761336\\
		3.46373839324673	0.59138098261362\\
		4.09090909090909	0.684754393790167\\
		4.76102717217463	0.80625047021627\\
		5.45454545454546	0.955629694964023\\
		6.14084012777783	1.12711466162036\\
		6.81818181818182	1.31946216482611\\
		7.50052689249177	1.53642621219928\\
		8.18181818181818	1.77628405234954\\
		8.913369432737	2.05966908010521\\
		9.54545454545454	2.3260713980839\\
		10.1951138678739	2.61974906918045\\
		10.9090909090909	2.9481867282283\\
		11.6245592768795	3.27730226787211\\
		12.2727272727273	3.57546044699675\\
		12.7873709632194	3.81218623129582\\
		13.6363636363636	4.20273010502642\\
		14.3055157129462	4.51053028046035\\
		15	4.82999961174161\\
	};
	\addplot [color=black, line width=1.0pt, forget plot]
	table[row sep=crcr]{%
		-15	-1.90124880029243\\
		-14.3743426864006	-1.81678506295652\\
		-13.6363636363636	-1.71715789120152\\
		-12.922520765984	-1.62078910370027\\
		-12.2727272727273	-1.53306698211061\\
		-11.6037511150776	-1.44275520082791\\
		-10.9090909090909	-1.3489760730197\\
		-10.1567767603786	-1.24741366294354\\
		-9.54545454545454	-1.1648851639288\\
		-8.78555408111793	-1.06229860124335\\
		-8.18181818181818	-0.980794254837886\\
		-7.52404914247162	-0.891995434526101\\
		-6.81818181818182	-0.796703345746977\\
		-6.16274148089556	-0.708218900213333\\
		-5.45454545454546	-0.612612436656068\\
		-4.69396922041941	-0.509934645049052\\
		-4.09090909090909	-0.428521527565159\\
		-3.35965597297158	-0.329802356643595\\
		-2.72727272727273	-0.24443061847425\\
		-2.00961780284732	-0.147547203676819\\
		-1.36363636363636	-0.0603397093833408\\
		-0.688142581770709	0.0308519511685226\\
		0	0.123751199707568\\
		0.681996956288809	0.215820788806558\\
		1.36363636363636	0.307842108798477\\
		2.07255602717534	0.403546263376239\\
		2.72727272727273	0.491933017889387\\
		3.46373839324673	0.591355882795876\\
		4.09090909090909	0.676023926980296\\
		4.76102717217463	0.766489867951143\\
		5.45454545454546	0.860114836071205\\
		6.14084012777783	0.952764616957575\\
		6.81818181818182	1.04420574516211\\
		7.50052689249177	1.13632233019396\\
		8.18181818181818	1.22829665425302\\
		8.913369432737	1.32705607312706\\
		9.54545454545454	1.41238756334393\\
		10.1951138678739	1.50009157187055\\
		10.9090909090909	1.59647847243484\\
		11.6245592768795	1.6930667020863\\
		12.2727272727273	1.78056938152575\\
		12.7873709632194	1.85004627974219\\
		13.6363636363636	1.96466029061666\\
		14.3055157129462	2.0549958209553\\
		15	2.14875119970757\\
	};
	\addplot [color=black!50!green, dash dot, line width=1.5pt, forget plot]
	table[row sep=crcr]{%
		4.09090909090909	0.644599435510496\\
		4.42596813154186	0.649962276220434\\
		4.76102717217463	0.65957507795753\\
		5.45454545454546	0.697297765025315\\
		6.14084012777783	0.758289376142287\\
		6.81818181818182	0.841584830290362\\
		7.50052689249177	0.948691400236181\\
		8.18181818181818	1.07886121205453\\
		8.913369432737	1.24446708611985\\
		9.54545454545454	1.40910320688409\\
		10.1951138678739	1.59818533987045\\
		10.9090909090909	1.81167310071809\\
		11.6245592768795	2.02559767990352\\
		12.2727272727273	2.21940141689416\\
		12.7873709632194	2.37326894366716\\
		13.6363636363636	2.62712568795694\\
		14.3055157129462	2.82719254823487\\
		15	3.0348499003245\\
	};
	\draw[gray, thin, dashed] (4.2,-0.1)--(4.2,0.6);
	\node[black] at (4.2,-0.2) {$G(v)$};
	\draw[gray, thin, dashed] (3.3,-0.1)--(3.3,0.55);
	\node[black] at (3.3,-0.2) {$y$};
	\node[black] at (-1,-0.4) {$k\mapsto L(k)$};
	\node[blue] at (-5,0.9) {$k\mapsto \sE_{S^{-1}(y)}(k)$};
	\node[red] at (5,1.7) {$k\mapsto \sE^c_{S^{-1}(y)}(k)$};
	\node[black!60!green] at (9,0.8) {$k\mapsto \sE_v(k)$};
	\node[black] at (9,0.3) {$k\mapsto D(k)$};
	\end{axis}
	\end{tikzpicture}
	\caption{Plot of $\sE_{S^{-1}(y)}$ (dotted curve), $\sE^c_{S^{-1}(y)}$ (solid curve below $\sE_{S^{-1}(y)}$), $D$ (dashed curve), $\sE_{v}$ (dash-dotted curve) and $L\equiv L_y^{\sE_{S^{-1}(y)}, (\sE_{S^{-1}(y)}^c)'(y-)}$ (line tangent to $\sE^c_{S^{-1}(y)}$ at $y$) under the assumptions of \cref{lem:D=sEand}. In the figure, $R(S^{-1}(y))=G(S^{-1}(y))=S(S^{-1}(y))=y$ and, for all $v\in(S^{-1}(y),1)$, $\sE_{v}(k)<L(k)$ for some $k>y$. Furthermore, $\sE_{S^{-1}(y)}'(y)$ and $(\sE^c_{S^{-1}(y)})'(y)$ both exist, and $\phi(S^{-1}(y))=\sE_{S^{-1}(y)}'(y)=(\sE^c_{S^{-1}(y)})'(y)$ is the slope of $L$.}
	\label{fig:L}
\end{figure}

Suppose $\underline{j} \geq G(S^{-1}(y)+h_0)$. Then the line $\hat{L} = L^{\sE_{S^{-1}(y)+h_0}}_{y,\underline{j}}$ joining $(y,\sE_{S^{-1}(y)+h_0}(y))$ to $(\underline{j},\sE_{S^{-1}(y)+h_0}(\underline{j}))$ lies below $L$ on $(y, \underline{j})$, but strictly above $L$ on an interval to the left of $y$. If follows that we can find small enough $\epsilon>0$ and $\check{y}<y$ such that $L^{\sE_{S^{-1}(y)+h_0},\underline{f}+\epsilon}_{\underline{j}}(\check{y}) = D(\check{y})$ and $\sE_{S^{-1}(y)+h_0} > {\hat L}$ on
$(\check{y},\underline{j})$. Then $R(S^{-1}(y)+h_0) \leq Q(S^{-1}(y)+h_0) \leq \check{y}<y < \underline{j} \leq S(S^{-1}(y)+h_0)$. In particular, $R(S^{-1} (y) +h_0)<y$.

Now suppose $\underline{j} < G(S^{-1}(y)+h_0)$. Then there exists $h \in (0, h_0]$ such that $G(S^{-1}(y)+h) \leq \underline{j} \leq G((S^{-1}(y)+h)+)$. It follows that $\sE_{S^{-1}(y)+h}(\underline{j}) = D(\underline{j})$. Then, exactly as before we can find  $\epsilon>0$ and $\check{y}<y$ such that $L^{\sE_{S^{-1}(y)+h},\underline{f}+\epsilon}_{\underline{j}}(\check{y}) = D(\check{y})$ and $D \equiv \sE_{S^{-1}(y)+h} >  L^{\sE_{S^{-1}(y)+h},\underline{f}+\epsilon}_{\underline{j}}$ on
$(\check{y},\underline{j})$. As before it follows that $R(S^{-1}(y)+h)<y$.
\end{proof}

\begin{lem}
\label{lem:claims3C}
Suppose $x\in\R$ is such that Condition $A(x)$ holds. Then 
\begin{equation} \{ v : v>S^{-1}(x), R(v) \leq x  \} = \{ v : v> S^{-1}(x), \psi_{S^{-1}(x), (\sE^c_{S^{-1}(x)})'(x-)}(v) = \phi(v) \}. \label{eq:claim3Ca}
\end{equation}
\end{lem}

\begin{rem}
\label{rem:claims3C}
It will follow from the proof of \cref{lem:claims3C} that if $x$ is such that Condition $A(x)$ holds then also
\begin{equation} \{ v : v>S^{-1}(x), R(v) < x  \} = \{ v : v> S^{-1}(x), \psi_{S^{-1}(x), (\sE^c_{S^{-1}(x)})'(x-)}(v) = \phi(v) \}. \label{eq:claim3Ca2}
\end{equation}
The equivalence of the left-hand-sides of \eqref{eq:claim3Ca} and  \eqref{eq:claim3Ca2} can also be seen directly. It is sufficient to argue that if Condition $A(x)$ holds then $\{ v : v>S^{-1}(x), R(v) = x  \} = \emptyset$. To see this, given $v > S^{-1}(x)$ choose $w \in (S^{-1}(x),v)$. Then, since $\sA(w,x)$ is non-empty there exists $u \in (S^{-1}(x),w]$ such that $R(u) < x < S(u)$. Then $R(v) \notin (R(u),S(u))$ and in particular $R(v) \neq x$.
\end{rem}

\begin{proof}[Proof of Lemma~\ref{lem:claims3C}]
We begin with a definition which will be useful in both the forward and reverse implication of \eqref{eq:claim3Ca}.

Fix $v>S^{-1}(x)$. Define $\ulz = \sup \{ w : w \in \sA(v,x) \}$. Note that by assumption $\sA(v,x)$ is non-empty and $S^{-1}(x)<\ulz\leq v$. We show that $\ulz \in \sA(v,x)$. If this is not immediately the case then since $\ulz > S^{-1}(x)$ there exists $(z_m)_{m \geq 1}$ with $z_m \uparrow \ulz$ and $R(z_m) < x$. Then, by Proposition~\ref{prop:Rlowersemi} we have that $R(\ulz) \leq \liminf_{u \uparrow \ulz} R(u) \leq x$.
Moreover, $R(\ulz) \notin (R(z_m), S(z_m)) \supseteq \{ x \}$. Hence $R(\ulz) < x$ and $\ulz \in \sA(v,x)$.
We also have $R(\ulz)<S(\ulz)$.

Let $H_{S^{-1}(x),z} = \{ w \in (S^{-1}(x),z] : \phi(w) = \psi_{S^{-1}(x),\phi^*}(w) \}$ where $\phi^*$ is shorthand for $(\sE^c_{S^{-1}(x)})'(x-)$. We claim that $\phi(\ulz) \leq (\sE^c_{S^{-1}(x)})'(x-) = \phi^*$ and $\ulz \in H_{S^{-1}(x),v}$.

For the first of these claims, suppose to the contrary that $\phi(\ulz) > \phi^*$. Then, for $k<x$, $\sE_{\ulz}(x) + \phi(\ulz)(k-x) < \sE_{S^{-1}(x)}(x) + \phi^*(k-x) \leq D(k)$, since $\phi^* \in \partial \sE^c_{S^{-1}(x)}(x)$. Since also $\sE_{\ulz}(S(\ulz)) + \phi(\ulz)(x-S(\ulz)) \leq \sE_{\ulz}(x)$, $\sE_{\ulz}(S(\ulz)) + \phi(\ulz)(k-S(\ulz)) \leq \sE_{\ulz}(x) + \phi(\ulz)(k-x) < D(k)$ for all $k\in(-\infty,x)$. In particular, $R(\ulz) \geq x$. But we saw above that $R(\ulz)<x$, and
this is our contradiction.


Now we show $\ulz \in H_{S^{-1}(x),v}$. Suppose not, i.e. suppose $\psi_{S^{-1}(x),\phi^*}(\ulz) < \phi(\ulz)$.
Let $(z_m)_{m\geq1}$ be such that 
$\phi(z_m)\downarrow 
\psi_{S^{-1}(x),\phi^*}(\ulz)<\phi(\ulz)$. Pick $m$ such that $ S^{-1}(y)< z_m<\ulz$ with $\phi(z_m)<\phi(\ulz)$. Then by the left-monotonicity of $R$ (and $S$) we have that $R(\ulz) \notin (R(z_m),S(z_m))$. Since $S(z_m)>x$ and $R(\ulz) < x$ we have $R(\ulz)\leq R(z_m)\leq S(z_m)\leq S(\ulz)$. If $R(z_m)=R(\ulz)<x < S(z_m)$, then $\sE^c_{z_m}(S(z_m)) =\sE_{z_m}(S(z_m)) = \sE_{z_m}(R(z_m)) + \phi(z_m) (S(z_m)-R(z_m)) < \sE_{\ulz}(R(\ulz)) + \phi(\ulz) (S(z_m)-R(\ulz)) =L^{\sE_{\ulz}}_{R(\ulz),S(\ulz)}(S(z_m))=\sE^c_{\ulz}(S(z_m))$, a contradiction to $\sE^c_{z_m} \geq \sE^c_{\ulz}$. On the other hand, if $R(\ulz)<R(z_m)$, by convexity of $\sE^c_{z_m}$ we must have that $D(R(\ulz))=\sE_{z_m}(R(\ulz)) \geq \sE_{z_m}(S(z_m)) - \phi(z_m) (S(z_m) - R(\ulz)) > \sE^c_{\ulz}(S(z_m)) - \phi(\ulz) (S(z_m) - R(\ulz)) = L^{\sE_{\ulz}}_{R(\ulz),S(\ulz)}(R(\ulz))=D(R(\ulz))$, again a contradiction. We conclude that we cannot have $\phi(\ulz)>\psi_{S^{-1}(x),\phi^*}(\ulz)$ and hence $\ulz\in{H}_{S^{-1}(x),v}$ as claimed.

\noindent{\sc Forward implication:} \\
First we show that $\{ v : v> S^{-1}(x), \psi_{S^{-1}(x),\phi^*}(v) = \phi(v) \} \subseteq \{ v : v>S^{-1}(x), R(v) \leq x \}$.
In fact we show the apparently stronger result that $\{ v : v> S^{-1}(x), \psi_{S^{-1}(x),\phi^*}(v) = \phi(v) \} \subseteq \{ v : v>S^{-1}(x), R(v) < x \}$, although by \cref{rem:claims3C} these implications are equivalent.

We wish to show that if $v>S^{-1}(x)$ and $R(v)  \geq x$ then $\phi(v) > \psi_{S^{-1}(x),\phi^*}(v)$.

Suppose $R(v) \geq x$. We argue that $\phi(\ulz) < \phi(v)$ and therefore that $\psi_{S^{-1}(x),\phi^*}(v) < \phi(v)$.

Temporarily let $L = L^{\sE_{\ulz}}_{R(\ulz),S(\ulz)}$.

Since $R(v) \geq x$ we have that $\ulz < v$. First, we claim that $S(\ulz) \leq G(\ulz+)$ (and then also $\sE_{\ulz} = D$ on $[G(\ulz),G(\ulz+)]\supseteq[G(\ulz),S(\ulz)]$, since $P_\mu$ is linear on $(G(\ulz),G(\ulz+))$ as $\mu$ does not charge $(G(\ulz),G(\ulz+))$).  Suppose not: then $S(\ulz)>G(\ulz+)$ and there exists $h \in (0,v - \underline{z})$ such that $G(\ulz+h) < S(\ulz)$. But $R(\ulz+h)\leq G(\ulz+h)$ and, by the left-monotonicity result (\cref{thm:monotone}), $R(\ulz+h)\notin(R(\ulz),S(\ulz))$. It follows that $R(\ulz+h) \leq R(\ulz) < x$ contradicting the supposed maximality of $\ulz$. We conclude that $S(\ulz) \leq G(\ulz+)$.

Second we claim that, for a sufficiently small $h>0$, $\sE_{\ulz+h} \geq L$. Suppose to the contrary that there exists $0<h<v-\ulz$ and $\epsilon>0$ such that $L^{\sE_{\ulz}, \phi(\ulz) - \epsilon}_{R(\ulz)}$ intersects $\sE_{\ulz+h}$ to the right of $S(\ulz)$. Then, letting $\tilde{k} = \tilde{k}(\ulz,h)$ be the $x$-coordinate of the smallest intersection point to the right of $S(\ulz)$ we have $\sE_{\ulz+h} > L^{\sE_{{\ulz}+h}}_{R(\ulz), \tilde{k}}$ on $(R(\ulz),\tilde{k})$. If $G(\ulz+ h) \leq \tilde{k}$ then $Q(\ulz+h) \leq R(\ulz)$ and so $R(\ulz+h) \leq Q(\ulz+h) \leq R(\ulz) < x$.  This contradicts the maximality of $\ulz$. Otherwise, if $G(\ulz+ h) > \tilde{k}$ then by an argument very similar to the last part of the proof of Lemma~\ref{lem:D=sEand} we can again conclude that there exists $0<h_1 \leq h$ for which $R(\ulz + h_1) < x$. Again this contradicts the maximality of $\ulz$.

Define $\olz = \sup \{ z : \sE_z \geq L\textrm{ on }\R \}$. Then
since $\lim_{z \uparrow \olz} \sE_z(k) = \sE_{\olz}(k)$ it follows that $\sE_{\olz} \geq L$ and there exists $k>S(\ulz)$ such that $\sE_{\olz}(k) = L(k)$, and then $R(\olz) = R(\ulz)$ and $\phi(\ulz)=\phi(\olz)$.
Then $R(\olz) < x$ and, by the maximality of $\ulz$, we have that $v<\olz$ (note that $v=\olz$ is excluded since $R(v) \geq  x > R(\ulz)$).
Since $v<\olz$,
$\sE_v \geq L$ on $(S(\ulz), \infty)$ and $\sE_v(S(v)) > L(S(v)) $ (note that if $\sE_v(S(v)) = L(S(v))$ then $R(v) \leq R(\ulz) < x$, a contradiction).
Hence
\[ \phi(v) \geq \frac{\sE_v(S(v)) - D(R(\ulz))}{S(v)-R(\ulz)} > \frac{L(S(v)) - D(R(\ulz))}{S(v)-R(\ulz)} = \phi(\ulz) \geq \psi_{S^{-1}(x),\phi^*}(v) . \]

{\noindent{\sc Reverse implication:}} \\
Now we show that $\{ v : v>S^{-1}(x), R(v) \leq x  \} \subseteq \{ v : v> S^{-1}(x), \psi_{S^{-1}(x),\phi^*}(v) = \phi(v) \}$. We suppose that $v> S^{-1}(x)$ and $ \psi_{S^{-1}(x),\phi^*}(v) < \phi(v)$ and show that $R(v) > x$.

From the opening comments of the proof of the lemma we have that $S^{-1}(x)<\ulz\leq v$, $R(\ulz) <x$ and $\ulz\in H_{S^{-1}(x),v}$. Then the assumption that $ \psi_{S^{-1}(x),\phi^*}(v) < \phi(v)$ implies that $\ulz<v$ so that $\ulz \in (S^{-1}(x), v)$.

Let
$\uulz = \sup \{ w: w \in H_{S^{-1}(x),v} \}$. We will see that $\uulz$ (respectively $\oolz$ introduced below) plays a very similar role to $\ulz$ (respectively $\olz$) from the forward implication.

We have $\uulz \geq \ulz > S^{-1}(x)$.
If $\uulz = v$ then there exists a sequence $(z_n)_{n\geq 1}$ with $z_n\uparrow v$ and $\phi(z_n) = \psi_{S^{-1}(y),\phi^*}(z_n)$. Then, by Proposition~\ref{prop:Rlowersemi} and the left-continuity of $\phi$ and $\psi_{S^{-1}(x),\phi^*}$, $\phi(v) = \liminf \phi(z_n) = \liminf \psi_{S^{-1}(x),\phi^*}(z_n) = \psi_{S^{-1}(x),\phi^*}(v)$ which is a contradiction. Hence we may conclude $S^{-1}(x) < \ulz \leq \uulz < v$ and $\psi_{S^{-1}(x),\phi^*}(v) = \phi(\uulz) < \phi(v)$.

First we show that $S(\uulz) \leq G(\uulz+)$. Suppose to the contrary and take $0<h< v - \uulz$ such that $G(\uulz+h) < S(\uulz)$. We claim that $\phi(\uulz+h)<\phi(\uulz)$, contradicting the maximality of $\uulz$.
	
We have $R(\uulz+h)\leq R(\uulz)<G(\uulz)\leq G(\uulz+h)<S(\uulz)\leq S(\uulz+h)$, so that $\sE^c_w=L^{\sE_w}_{R(w),S(w)}$ on $[R(w),S(w)]$ for $w\in\{\uulz,\uulz+h\}$. Since $G(\uulz+h)<S(\uulz)$, $\sE_{\uulz}(S(\uulz))>\sE_{\uulz+h}(S(\uulz))\geq \sE^c_{\uulz+h}(S(\uulz))$, and therefore $L^{\sE_{\uulz}}_{R(\uulz),S(\uulz)}=L_{S(\uulz)}^{\sE_{\uulz},\phi(\uulz)}> L_{S(\uulz)}^{\sE^c_{\uulz+h},\phi(\uulz)}$ everywhere. Then for $\psi\geq\phi(\uulz)$, $D\geq L^{\sE_{\uulz}}_{R(\uulz),S(\uulz)}>L_{S(\uulz)}^{\sE^c_{\uulz+h},\psi}$ on $(-\infty,G(\uulz)]\supset(-\infty,R(\uulz+h)]$. But $D(R(\uulz+h))=L^{\sE_{\uulz+h}}_{R(\uulz+h),S(\uulz+h)}(R(\uulz+h))=L_{S(\uulz)}^{\sE^c_{\uulz+h},\phi(\uulz+h)}(R(\uulz+h))$ and therefore we conclude that $\phi(\uulz+h)<\phi(\uulz)$. It follows that $S(\uulz)\leq G(\uulz+)$ as claimed.

Second we show that $\sE_{\uulz +h} \geq L^{\sE_{\uulz},\phi(\uulz)}_{R(\uulz)}$ for a sufficiently small $h>0$. If this is not the case then for any $\epsilon>0$ there exists some $h \in (0,\epsilon\wedge (v - \uulz))$ we have that $\delta = \inf_{ { k} > S(\uulz)} \{ \sE_{\uulz + h} ({ k}) - L^{\sE_{\uulz},\phi(\uulz)}_{R(\uulz)}({k}) \} < 0$. Let $L^\delta$ be given by $L^\delta(k) = L^{\sE_{\uulz},\phi(\uulz)}_{R(\uulz)}(k) + \delta$ and let $\uuls>S(\uulz)$ be such that $L^\delta(\uuls) = \sE_{\uulz+h}(\uuls)$. We have $\sE_{\uulz+h} \geq \sE^c_{\uulz + h} \geq L^\delta$ and $\sE_{\uulz+h}(\uuls) = \sE^c_{\uulz + h}(\uuls) = L^\delta(\uuls)$. It follows that $(\sE^c_{\uulz + h})'(\uuls-) \leq \phi(\uulz)$.

Suppose $G(\uulz+h)\leq \uuls$. Then $S(\uulz+h) \leq \uuls$ and $\phi(\uulz+h) = (\sE^c_{\uulz + h})'(S(\uulz+h)-){\leq}(\sE^c_{\uulz + h})'(\uuls-) \leq \phi(\uulz)$. Then $\uulz$ is not maximal in $H_{S^{-1}(x),v}$, a contradiction. Now suppose $G(\uulz+h)> \uuls$. Then there exists $\uulu\in(\uulz,\uulz+h)$ with $G(\uulu) \leq \uuls \leq G(\uulu+)$ for which $\sE_{\uulu}= \sE_{\uulz+h}$ on $(-\infty, \uuls]$. Then by the same argument as in the case $G(\uulz+h)\leq \uuls$ but with $\uulz+h$ replaced by $\uulu$ we find that $\phi(\uulu) \leq \phi(\uulz)$ again contradicting the maximality of $\uulz$ in $H_{S^{-1}(x),v}$. Indeed $\sE_{\uulu}\geq\sE^c_{\uulu}\geq L^\delta$ and $\sE_{\uulu}(\uuls)=\sE^c_{\uulu}(\uuls)=L^\delta(\uuls)$ so that $S(\uulu)\leq\uuls$ and $(\sE^c_{\uulu})'(\uuls-)\leq\phi(\uulz)$. Then $\phi(\uulu)=(\sE^c_{\uulu})'(S(\uulu)-)\leq(\sE^c_{\uulu})'(\uuls-)\leq\phi(\uulz)$.

Now let $L\equiv L^{\sE_{\uulz},\phi(\uulz)}_{R(\uulz)}$ and introduce
\[  \oolz = \sup \{z:  z> \uulz \mbox{ such that $\sE_z \geq L$ on $\R$} \} . \]
It is clear that $\{z:  z> \uulz \mbox{ such that $\sE_z \geq L$ on $\R$} \}$ is non-empty
and thus $\oolz$ is well-defined.

It follows similarly to the forward implication that $\phi(\uulz) = \phi(\oolz)$ and $R(\uulz) = R(\oolz)$. Then we must have $v < \oolz$ since otherwise $\uulz$ is not maximal. Then $\sE_v>\sE_{\oolz}$ on $(G(\oolz),\infty)$ and $\sE_v(S(v))>L(S(v))$ since otherwise $\phi(v) \leq \phi(\uulz)$ which was ruled out above. It follows that the line joining $(S(\uulz), \sE_{v}(S(\uulz)))$ to $(S(v),\sE_v(S(v)))$ has slope steeper than $\phi(\uulz)$. Further, this line lies strictly below $\sE_v \equiv D$ to the left of $S(\uulz)$. Hence $R(v) \geq S(\uulz)$. Finally, by the right-continuity of $S^{-1}$ we have $S(\uulz)>x$ and hence $R(v)>x$ as required.
\end{proof}

We are now ready to prove Proposition~\ref{prop:setequality}.

\begin{proof}[Proof of Proposition~\ref{prop:setequality}]
	Suppose $D(y) < \sE_{S^{-1}(y)}(y)$ then the result follows by combining Lemmas \ref{lem:D<sE} and \ref{lem:claims3C}.
	
	Suppose $D(y) = \sE_{S^{-1}(y)}(y)$ and for all $v>S^{-1}(y)$ there exists $k$ such that $\sE_v(k) < L_y^{\sE_{S^{-1}(y)}, (\sE^c_{S^{-1}(y)})'(y-)}(k)$. Then the result follows from \cref{lem:D=sEand} and \cref{lem:claims3C}.
	
	The remaining case is if $D(y) = \sE_{S^{-1}(y)}(y)$ and there exists $v_0>S^{-1}(y)$ such that $\sE_{v_0} \geq L$ where $L= L_y^{\sE_{S^{-1}(y)}, (\sE^c_{S^{-1}(y)})'(y-)}$. Since $y$ is a continuity point of $\mu$ and $\nu$,  $(\sE^c_{S^{-1}(y)})'$ and $\sE_{S^{-1}(y)}'$ exist at $y$ and $L = L_y^{\sE_{S^{-1}(y)}, (\sE^c_{S^{-1}(y)})'(y)}=L_y^{\sE_{S^{-1}(y)}, \sE_{S^{-1}(y)}'(y)}$.
	
	Define $\tilde{z} = \sup\{ u: u>S^{-1}(y), \sE_u \geq L \}$ and note that it is well-defined due to the existence of $v_0$. Then, $R(\tilde{z}) \leq y$ and $\phi(\tilde{z}) = \sE_{S^{-1}(y)}'(y)=D'(y)$. Further, by the right-continuity of $S^{-1}$ we must have $\tilde{z} \leq \olz$ where $\olz := S^{-1}(S(\tilde{z}))$.

Recall that $\psi_{<y>}$ is shorthand for $\psi_{S^{-1}(y), P_\nu'(y) - S^{-1}(y)}$.	
	We will prove \eqref{eq:claim3D} by proving that
	\begin{equation} \{ v : v \in I, R(v) \leq y  \} = \{ v : v \in I, \psi_{<y>}(v) = \phi(v) \} \label{eq:claim3DI} \end{equation}
	for each of $I=(S^{-1}(y), \tilde{z}]$, $I=(\tilde{z},\olz]$ and $I=(\olz,1)$ separately.
	
	\vspace{2mm}

	\noindent{\bf Case $I= (S^{-1}(y), \tilde{z}]$.} \\
	First note that if $v\in(S^{-1}(y),\tilde{z}]$ then since $\sE_v\geq L$, we have that $\sE_{S^{-1}(y)}\geq \sE_v\geq \sE^c_v\geq L$ with all inequalities being equalities at $y$. Since the subdiferential of $\sE^c_v$ is non-decreasing and $y<S(v)$, we have that $\phi(v)\geq\sE^\prime_{S^{-1}(y)}(y)$ and therefore $\psi_{<y>}(v)=\sE^\prime_{S^{-1}(y)}(y)$. It follows that
	$\psi_{<y>}(v)=\phi(\tilde{z})=\sE^\prime_{S^{-1}(y)}(y)=P^\prime_\nu(y)-S^{-1}(y)$. Further, if $\psi_{<y>}(v)=\phi(v)$ we must have that $\sE_v(S(v))=L(S(v))$ and consequently $R(v)\leq y$.
	
	On the other hand, suppose $v\in(S^{-1}(y),\tilde{z}]$ and $R(v)\leq y$. As $v>S^{-1}(y)$ we have $y<S(v)$. Further, since $\sE^c_{S^{-1}(y)}(y) = \sE^c_v(y) = \sE_{S^{-1}(y)}(y) = \sE_v(y) = D(y)$, by \cref{lem:chagree} we have that $\sE^c_v = \sE^c_{S^{-1}(y)}$ on $(-\infty,y]$. If $R(v)=y$ then $\phi(v) = (\sE^c_v)'(y) = D'(y) = \psi_{<y>}(v)$. Otherwise, if $R(v)<y$ we have $\phi(v) = (\sE^c_v)'(R(v)+) = (\sE^c_{S^{-1}(y)})'(R(v)+) \leq (\sE^c_{S^{-1}(y)})'(y) = \psi_{<y>}(v) \leq \phi(v)$. Again we conclude $\phi(v) = \psi_{<y>}(v) = D'(y)$ as required.
	
	For future reference note that $\psi_{<y>}(\tilde z)=\phi(\tilde z)$.
	
	\vspace{2mm}
	\noindent{\bf Case $I= (\tilde{z},\olz]$.}
	
	Recall $\olz = S^{-1}(S(\tilde{z}))$. Throughout this section we assume $\tilde{z}<\olz$ else there is nothing to prove.
	
	Let $\ols = S(\olz)$. By construction $S$ is constant on $(\tilde{z}, \olz]$ and $\olz = S^{-1}(S(\olz))$. Further, $S(z)>\ols$ if and only if $z>\olz$. We show that both $\{v: v \in (\tilde{z},\olz], R(v)\leq y\}=(\tilde{z},\olz]$ and $\{v: v \in (\tilde{z},\olz], \psi_{<y>}(v)=\phi(v)\}=(\tilde{z},\olz]$. 
	
	We have $S(\tz) \leq \ols$ and $L(S(\tz)) = \sE_{\tz}(S(\tz))$.
	We show that we must have $\sE_{\tz}(\ols) = L(\ols)$ so that if $S(\tilde{z}) < \ols$ then $\sE_{\tz} = L$ on $[S(\tz),\ols]$. Suppose for a contradiction that $\sE_{\tz}(\ols) > L(\ols)$. Then either  $\sE_{\tz}(\ols) = D(\ols) > L(\ols)$ or both $\sE_{\tz}(\ols) > D(\ols)$ and $\sE_{\tz}(\ols) > L(\ols)$. In the former case we can find small enough ${\breve z}\in(\tz,1)$ such that $\sE_{\breve{z}} \geq L$, contradicting the maximality of $\tilde{z}$. In the second case $G(\tz+)<\ols$ (since $\sE_{\tz}(\ols) > D(\ols)$), $S(\tz)<\ols$ (since $\sE_{\tz}(\ols) > L(\ols)$), and there must exist $z_1 \in (\tz,\olz]$ such that, for all $z\in(\tz,z_1)$, $G(z)<\ols$ and $\sE_{\tz}(\ols) > \sE_{z}(\ols) >  L(\ols)\vee D(\ols)$. Then by the maximality of $\tz$ there must be $k_1 \in (G(\tz),\ols)$ such that $\sE_{z_1}(k_1) < L(k_1)$. Fix $z\in(\tz,z_1)$ and let $\phi_1 = \inf_{x>y} \frac{\sE_{z}(x)-\sE_{z}(y)}{x-y}$ and let the infimum be attained at $x_1$. Then $\phi_1 < \phi(\tz)$ and $x_1 \in ( G(\tz),\ols)$. If $G(z) \leq x_1$ then $x_1 \leq S(z) \leq s_1 = \sup \{ s:s > x_1, \sE_{z}(s) \leq L(s) \} < \ols$, a contradiction since $S(z) = S(\olz) = \ols$. Conversely, if $x_1 < G(z)$ then there exists $z_2 \in (\tz,z)$ such that $G(z_2) \leq x_1 \leq G(z_2+)$. Then if $s_2  = \sup \{ s:s > x_1, \sE_{z_2}(s) \leq L(s) \}$ we have $s_2 < \ols$ and $S(z_2) \leq s_2 < \ols$, again a contradiction. Hence we must have $\sE_{\tz}(\ols) = L(\ols)$.
	
	Note that for $v\in(\tilde {z},\olz]$, $S(v)=S(S^{-1}(S(\tilde z)))= S(\olz) = \ols$.
Furthermore, we must have that $G(\tilde z)\leq G(\tilde z+)<\ols$. (If $G(\tilde z+) \geq \ols$, then, for all $v\in(\tilde z,\olz]$, $G(\tilde z+) \leq G(v) \leq S(v)=\ols \leq G(\tilde{z}+)$ and hence $G(v)= S(v)$. Then using \eqref{eq:Ediff} we deduce that $\sE_v=\sE_{\tilde z}$ on $(-\infty,\ols=S(v)]$. Then $\sE^c_v(\ols) = \sE_v(\ols) = \sE_{\tz}(\ols) = L(\ols) =\sE^c_{\tz}(\ols)$ which implies by \cref{lem:chagree} that $\sE^c_v = \sE^c_{\tz}$ on $(-\infty, \ols]$. Then, since $L= \sE^c_v$ on an interval to the left of $\ols$ we have $L \leq \sE^c_v$ everywhere, contradicting the maximality of $\tilde z$.) It follows by \eqref{eq:Ediff} again that $\sE_{v}(\ols) <\sE_{\tilde z}(\ols)=L(\ols)$ for all $v\in(\tilde z, \olz]$.

First suppose that $G(\tz)\leq S(\tz)\leq G(\tz+)<\ols$ and continue to take $v \in (\tz,\olz]$. Since $\sE_{\tz}\equiv L$ on $[G(\tz+),\ols]$ and $P_{\mu_{\tz}}$ is linear on $[G(\tz+),\infty)$, we have that $P_\nu=\sE_{\tz}+P_{\mu_{\tz}}$ is also linear on $[G(\tz+),\ols]$. Then $D=P_\nu-P_\mu$ must be concave and below $L$ on ${(} G(\tz+),\ols]$ (recall that $D(\ols)\leq\sE_{v}(\ols)<L(\ols))$. If $R(v)\in[S(\tz),G(\tz+)]$, then $\sE_v(R(v))=D(R(v))=L(R(v))$ and $L^{\sE_v}_{R(v),\ols}$ must cross $L$ at $R(v)$. If $R(v)\in(G(\tz+),\ols)$ then, since $D$ is concave and equal to or below $L$ at $R(v)$, $L^{\sE_v}_{R(v),\ols}$ crosses $L$ at some $k\in[G(\tz+),\ols)$. In both cases we have that $\sE^c_{\tz}=L<L^{\sE_v}_{R(v),\ols}\leq\sE^c_v$ on an interval to the left of the corresponding crossing point, a contradiction to $\sE^c_v\leq\sE^c_{\tz}$. Hence, since $R(v)\notin(R(\tz),S(\tz))$, either $R(v)\leq R(\tz)\leq y$ or $R(v)=G(v)=S(v)=\ols$. However the latter cannot happen due to concavity of $D$ on $[G(\tz+),\ols]$. Hence $\{v : v \in (\tz,\olz], R(v) \leq y \} =  (\tz,\olz]$.

We now show that $\phi$ is non-increasing on $(\tz, \olz]$, and hence, since $\psi_{<y>}(\tz)=\phi(\tz)$, that $\psi_{<y>}(v) = \phi(v)$ on $(\tz,\olz]$. Take $\tz<v < w \leq \olz$. Since $R(w) < S(w)=S(v)=\ols$ we have $R(w) \leq R(v)<S(v)$. Then, since
$\partial\sE^c_{\tz}$ is non-decreasing, and using Lemma~\ref{lem:sEc'} for the second inequality,
\[ \phi(w) = (\sE_w^c)'(R(w)+) \leq (\sE_w^c)'(R(v)+) \leq (\sE_v^c)'(R(v)+) = \phi(v) \]
and the result follows.
		
		Now suppose $G(\tz+)<S(\tz)$. Then $\{v\in(\tilde z, \olz]:G(v)<S({\tz})\}$ is non-empty. Define $\hat z:=\sup\{v\in(\tilde z, \olz]:G(v)<S(\tz)\}$ and note that $\tilde z<\hat z\leq \olz$. We now show that
	\begin{equation}
	\label{eq:euqalTildeZ}
	\{v\in (\tilde z,\hat z]: R(v)\leq y\}=\{v\in (\tilde z,\hat z]: \psi_{<y>}(v)=\phi(v)\} = (\tz,\hat{z}].
	\end{equation}
	For $u\in(\tilde z,\hat z)$, $R(u)\notin(R(\tilde z),S(\tilde z))$ and $R(u)\leq G(u)<S(\tilde z)$, and therefore $R(u)\leq R(\tilde z)\leq y$. Then by \cref{prop:Rlowersemi}(iii), $R(\hat z)\leq\liminf_{u\uparrow \hat z}R(u)\leq y$. On the other hand, the same argument shows that $R(v)\leq R(u)\leq y$ for $u,v\in[\tilde z,\hat z)$ with $u<v$. Then we must have that $\sE^c_z(R(z))=\sE^c_{\tilde z}(R(z))$ for $z\in\{u,v\}$. Since the subdifferential of $\sE^c_{\tilde z}$ is non-decreasing and $\sE^c_u(S(u))=\sE_u(S(u))>\sE_v(S(v))=\sE^c_v(S(v))$, for $\tz \leq u < v <\hat{z}$ we have that $\phi(u) \geq \phi(v)$. Then since $\phi(\tz) = \psi_{<y>}(\tz)$ we have $\psi_{<y>}(\tilde z)\geq \phi(u)\geq\phi(v)\geq \lim_{u\uparrow \hat z}\phi(u) =\phi(\hat z)$, which proves the claim.
	
	Now, if $\hat z= \olz$, then \eqref{eq:claim3DI} follows for $I=(\tz,\olz]$. Therefore, suppose that $\hat z<\olz$. We claim that
	$$
	\{v\in (\hat z,\olz]: R(v)\leq y\}=\{v\in (\hat z,\olz]: \psi_{<y>}(v)=\phi(v)\} = (\hat{z}, \olz].
	$$
	Indeed, using the definition of $\hat z$, we have that $G(v)=\ols$ for all $v\in(\hat z,\olz]$, and therefore $G(\hat z +)= \ols$. Then by \eqref{eq:Ediff}, for all  $v\in(\hat z,\olz]$,  $\sE_v=\sE_{\hat z}$ on $(-\infty,\ols=S(v)]$. It follows from \cref{cor:R=R} with $u=\hat z$ that for all  $v\in(\hat z,\olz]$ we have $R(v)=R(\hat z)\leq y$ and $\psi_{<y>}(\hat z)=\psi_{<y>}(v)=\phi(v)$.
	We conclude that \eqref{eq:claim3DI} holds for $I=(\hat z,\olz]$, and hence, given \eqref{eq:euqalTildeZ}, for $I=(\tz,\olz]$.

	\vspace{2mm}
	\noindent{\bf Case $v \in (S^{-1}(S(\tilde{z})),1)$.}
	
	In the case $\tilde{z}=\olz$ then $\tilde{z}=\olz = S^{-1}(S(\olz)) = S^{-1}(S(\tilde{z})) $ and then from the first case $R(\olz)= R(\tilde{z}) \leq y<S(\olz)$ and $\psi_{<y>}(\olz) = \psi_{<y>}(\tilde{z}) =\phi(\tilde{z}) = \phi(\olz)$. In the case where $\tilde{z} < \olz$ then again we have $\olz=S^{-1}(S(\olz))$ and then from the second case $R(\olz) \leq y<S(\olz)$ and $\psi_{<y>}(\olz) =\phi(\olz)$.

	By the left-monotonicity of $R$, for $v> \olz$ we have $R(v) \notin (R(\olz),S(\olz))$. Since $R(\olz ) \leq y<S(\olz)$ we conclude $R(v) \leq y$ is equivalent to $R(v)  \leq R(\olz)$. Then $\{ v : v > \olz, R(v) \leq y  \} = \{ v : v > \olz, R(v)  \leq  R(\olz)  \}$.
	
	Note that for $v>\olz$ we have that $\psi_{<y>}(v) = \psi_{\olz, \phi(\olz)}(v)$. Hence to show \eqref{eq:claim3DI} for $I=(\olz,1)$ we need to show that
	\begin{equation}
	\{ v : v> \olz, R(v) \leq R(\olz))  \} = \{ v : v> \olz, \psi_{\olz,\phi(\olz)}(v) = \phi(v) \}.
	\label{eq:claim3DzzA}
	\end{equation}
	But, since $R(\olz)<S(\olz)$ this is immediate from \cref{lem:Rolz<Solz}.
\end{proof}
\end{document}